\documentclass[11pt]{article}

\usepackage{amssymb,amsmath,bm}
\usepackage{amsthm}

\usepackage{textcomp}
\usepackage{enumerate}      
\usepackage{graphicx}        
\usepackage{caption}

\usepackage{mathrsfs}

\usepackage{url}

\renewcommand{\setminus}{{\smallsetminus}}

\newcommand{\cp}[1]{\vcenter{\hbox{#1}}}

\usepackage{amssymb,amsmath,bm}
\usepackage{amsthm}
\usepackage{hyperref}
\usepackage{mathrsfs}
\usepackage{textcomp}
\usepackage{enumerate}
\usepackage{graphicx}
\usepackage{tikz}
\usepackage{verbatim}

\newtheorem{theorem}{Theorem}[section]
\newtheorem{lemma}[theorem]{Lemma}
\newtheorem{proposition}[theorem]{Proposition}
\newtheorem{definition}[theorem]{Definition}
\newtheorem{corollary}[theorem]{Corollary}
\newtheorem{conjecture}[theorem]{Conjecture}

\theoremstyle{remark}
\newtheorem{remark}[theorem]{Remark}

\theoremstyle{remark}

\numberwithin{equation}{section}

\begin{document}
\title{\bf Relative Reshetikhin-Turaev invariants, hyperbolic cone metrics and discrete Fourier transforms I}

\author{Ka Ho Wong and Tian Yang}

\date{}

\maketitle

\begin{abstract}  We propose the Volume Conjecture for the relative  Reshetikhin-Turaev invariants of a closed oriented $3$-manifold with a colored framed link inside it whose asymptotic behavior is related to the volume and the Chern-Simons invariant of the hyperbolic cone metric on the manifold with singular locus the link and cone angles determined by the coloring. We prove the conjecture in the case that the ambient $3$-manifold is obtained by doing an integral surgery along some components of a fundamental shadow link and the complement of the link in the ambient manifold is homeomorphic to the fundamental shadow link complement, for sufficiently small cone angles. Together with Costantino and Thurston's result that all compact oriented $3$-manifolds with toroidal or empty boundary can be obtained by doing an integral surgery along some components of a suitable fundamental shadow link, this provides a possible approach of solving Chen-Yang's Volume Conjecture for the Reshetikhin-Turaev invariants of closed oriented hyperbolic $3$-manifolds. We also introduce a family of topological operations (the change-of-pair operations) that connect all pairs of a closed oriented $3$-manifold and a framed link inside it that have homeomorphic complements, which correspond to doing the partial discrete Fourier transforms to the corresponding relative Reshetikhin-Turaev invariants. As an application, we find a Poisson Summation Formula for the discrete Fourier transforms. 
\end{abstract}

\section{Introduction}

We propose the Volume Conjecture for the relative  Reshetikhin-Turaev invariants of a closed oriented $3$-manifold with a colored framed link inside it  whose asymptotic behavior is related to the volume and the Chern-Simons invariant of the hyperbolic cone metric on the manifold with singular locus the link and cone angles determined by the coloring. See Conjecture \ref{VC}, and Sections \ref{RRT} and \ref{CCS} for a review of the relative Reshetikhin-Turaev invariants and the hyperbolic cone manifolds.

We prove the conjecture in the case that the ambient $3$-manifold is obtained by doing an integral surgery along some components of a fundamental shadow link and the complement of the link in the ambient manifold is homeomorphic to the fundamental shadow link complement, for sufficiently small cone angles. See Theorem \ref{main}, and Section \ref{fsl} for a review of the fundamental shadow links. A result of Costantino and Thurston\,\cite{CT} shows that all compact oriented $3$-manifolds with toroidal or empty boundary can be obtained by doing an integral surgery along some components of a suitable fundamental shadow link. On the other hand, it is expected that hyperbolic cone metrics interpolate the complete cusped hyperbolic metric on the $3$-manifold with toroidal boundary and the smooth hyperbolic metric on the Dehn-filled $3$-manifold, corresponding to the colors running from $\frac{r-1}{2}$ to $0$ or $r-2.$ Therefore, if one can push the cone angles in Theorem \ref{main} from sufficiently small all the way up to $2\pi,$ then one proves the Volume Conjecture for the Reshetikhin-Turaev invariants of closed oriented hyperbolic $3$-manifolds proposed by Chen and the second author\,\cite{CY}. 

This thus suggests a possible approach of solving Chen-Yang's Volume Conjecture.  In \cite{WY2}, we prove Conjecture \ref{VC} for all pairs $(M,K)$ such that $M\setminus K$ is homeomorphic to the figure-$8$ knot complement in $S^3$ with all possible cone angles, showing the plausibility of this new approach.

We also introduce a family of the change-of-pair operations (see Section \ref{COP}) that connect all pairs of a closed oriented $3$-manifold and a framed link inside it that have homeomorphic complements, which correspond to doing the partial discrete Fourier transforms (see Section \ref{DDFT}) to the corresponding relative Reshetikhin-Turaev invariants. As an application, we find a Poisson Summation formula for the discrete Fourier transforms (see Formula (\ref{PSF})). 


\subsection{Volume Conjecture for the relative Reshetikhin-Turaev invariants}

\begin{conjecture}\label{VC} Let $M$ be a closed oriented $3$-manifold and let $L$ be a framed hyperbolic link in $M$ with $n$ components. For an odd integer $r\geqslant 3,$ let  $\mathbf m=(m_1,\dots,m_n)$ and let $\mathrm{RT}_r(M,L,\mathbf m)$ be the $r$-th relative Reshetikhin-Turaev invariant of $M$ with $L$ colored by $\mathbf m$ and evaluated at the root of unity $q=e^{\frac{2\pi\sqrt{-1}}{r}}.$ 
For a sequence $\mathbf m^{(r)}=(m^{(r)}_1,\dots,m^{(r)}_n),$ let 
$$\theta_k=\bigg|2\pi- \lim_{r\to\infty}\frac{4\pi m^{(r)}_k}{r}\bigg|,$$
 and let $\boldsymbol\theta=(\theta_1,\dots, \theta_n).$
If $M_{L_{\boldsymbol\theta}}$ is a hyperbolic cone manifold consisting of $M$ and a hyperbolic cone metric on $M$ with singular locus $L$ and cone angles $\boldsymbol\theta,$ then 
$$\lim_{r\to\infty} \frac{4\pi }{r}\log \mathrm{RT}_r(M,L,\mathbf m^{(r)})=\mathrm{Vol}(M_{L_{\boldsymbol\theta}})+\sqrt{-1}\mathrm{CS}(M_{L_{\boldsymbol\theta}})\quad\quad\text{mod } \sqrt{-1}\pi^2\mathbb Z,$$
where $r$ varies over all positive odd integers.
\end{conjecture}

We note that if $M=S^3,$ then Conjecture \ref{VC} covers Kashaev's Volume Conjecture for the colored Jones polynomials of hyperbolic links\,\cite{Ka2, MM, MMOTY, DKY} and its generalization\,\cite{MY}, at the root of unity $q=e^{\frac{2\pi\sqrt{-1}}{r}}.$ See also \cite{G} and \cite[Section 4.2]{C1} for a discussion of the values at the root $q=e^{\frac{\pi\sqrt{-1}}{r}}.$ If the framed link $L=\emptyset $ or the coloring $\mathbf m=\mathbf 0$ or $\mathbf{r-2},$ then Conjecture \ref{VC} covers  Chen-Yang's Volume Conjecture for the Reshetikhin-Turaev invariants of closed oriented hyperbolic $3$-manifolds.

The main result of this paper is the following Theorem \ref{main} (see also Theorem \ref{main2} for a more precise statement), where the change-of-pair operation is described in the next section. 

\begin{theorem}\label{main} Conjecture \ref{VC} is true for all pairs $(M,L)$ obtained by doing a change-of-pair operation from the pair $(M_c, L_{\text{FSL}})$ with sufficiently small cone angles, where $M_c=\#^{c+1}(S^2\times S^1)$ and $L_{\text{FSL}}$ is a fundamental shadow link in $M_c.$
 \end{theorem}

As a consequence of Theorem \ref{main}, we prove the Generalized Volume Conjecture\,\cite{MY, G} for the colored Jones polynomials of the universal families of links respectively considered by Purcell\,\cite{P}, van der Veen\,\cite{V} and Kumar\,\cite{Ku}. See Proposition \ref{TOFAL} and Theorems \ref{Cor} and \ref{Cor2} for more details.

\subsection{The change-of-pair operations}\label{COP}

Let $M$ be a closed oriented $3$-manifold and let $L$ be a framed link in $M.$ In this section, we introduce a topological operation that  changes the pair $(M,L)$ without changing the complement $M\setminus L,$ and show that these operations connect all such pairs that have homeomorphic complements.  

Suppose $L=L_1\cup\dots\cup L_n.$ For each component  $L_i:[0,1]\times S^1\to M$ of $L,$ we call $L_i(\{0\}\times S^1)\subset M$ the core curve of $L_i$ and $L_i(\{1\}\times S^1)\subset M$ the parallel copy. Let $(I,J)$ be a partition of $\{1,\dots,n\},$ $L_I=\cup_{i\in I}L_i$ and $L_J=\cup_{j\in J}L_j.$ For each $i\in I,$ let $L_i^*$ be the framed knot in $M\setminus L$ whose core curve is isotopic to the meridian of the tubular neighborhood of $L_i,$ and let $L^*_I=\cup_{i\in I}L^*_i.$ Let $M^*$ be the closed $3$-manifold obtained from $M$ by doing the surgery along $L_I$ and let $L^*$ be the framed link obtained from $L$ by replacing $L_I$ by $L^*_I,$ ie., $M^*=M_{L_I}$ and $L^*=L_I^*\cup L_J.$  The \emph{change-of-pair operation} $T_{(L_I;L^*_I)}$  is defined by sending $(M,L)$ to  $(M^*,L^*).$ See Figure \ref{TFT}.

\begin{figure}[htbp]
\centering
\includegraphics[scale=0.3]{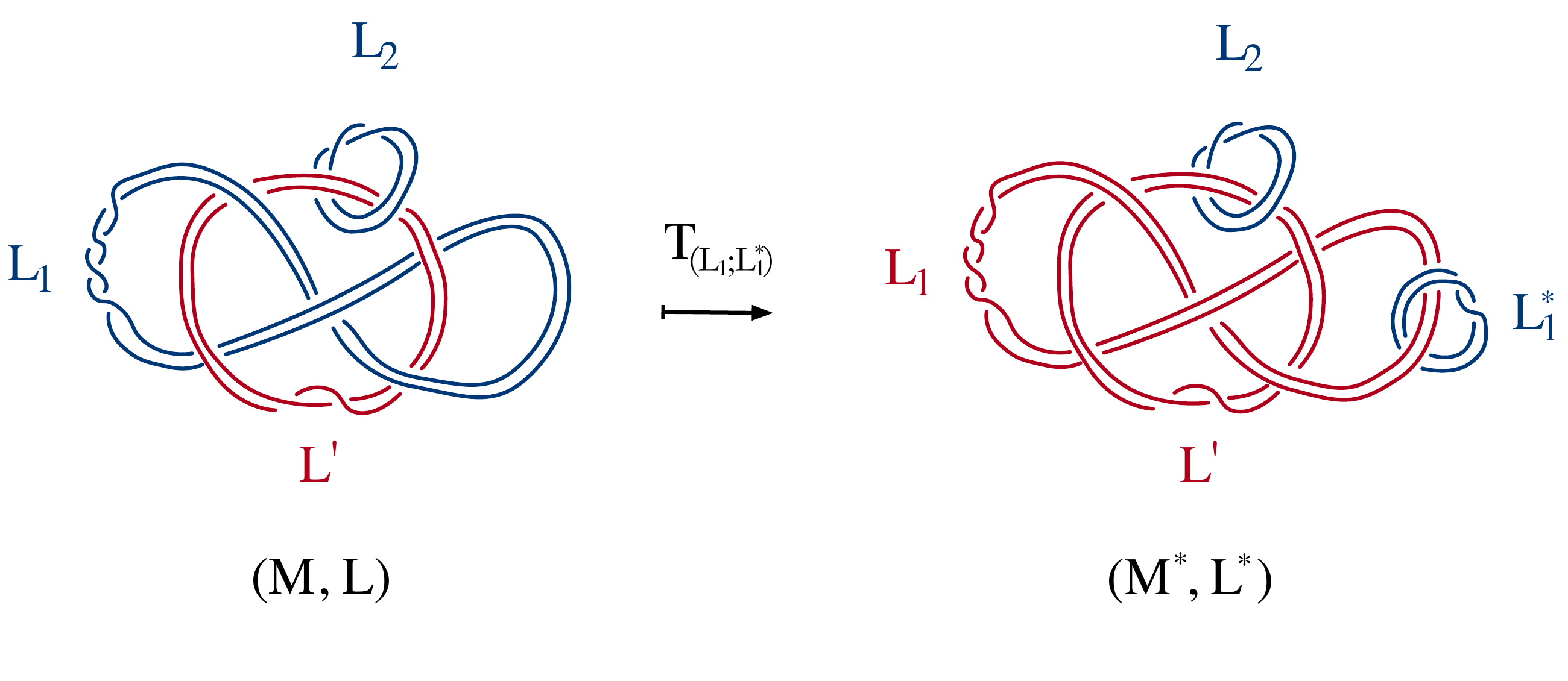}
\caption{$M$ is obtained from $S^3$ by the surgery along $L',$ and $L=L_1\cup L_2.$ $(I,J)=(\{1\}, \{2\}).$  $M^*$ is obtained from $M$ by doing the surgery along $L_1,$ hence is obtained from $S^3$ by doing the surgery along $L'\cup L_1;$ and $L^*=L^*_1\cup L_2.$ }
\label{TFT}
\end{figure}

By the way it is chosen, the core curve of each $L^*_i$ is isotopic to a curve on the tubular neighborhood of $L_i$ that intersects the parallel copy of $L_i$ once, hence in $M^*$ the core curve of each $L^*_i$ is isotopic to the core of the filled in solid torus of the surgery along $L_i.$ As a consequence, we have 

\begin{proposition}\label{equal} If $(M^*,L^*)$ is obtained form $(M,L)$ by doing a change-of-pair operation, then $M^*\setminus L^*$ is homeomorphic to $M\setminus L.$
\end{proposition}

Conversely, if $(M^*,L^*)$ is a pair such that $M^*\setminus L^*$ is homeomorphic to $M\setminus L,$ then $M^*$ is obtained from $M$ by doing a rational Dehn-surgery along some components $L_I$ of $L.$
By e.g. \cite[p. 273]{R},  $M^*$  can be obtained from $M$  by doing an integral surgery along a framed link $L'$ obtained from $L_I$ by iteratively linking in framed unknots, corresponding to doing a sequence of the change-of-pair operations.  As a consequence, we have 

\begin{proposition}\label{connect} Every two pairs $(M,L)$ and $(M^*,L^*)$ such that $M\setminus L$ is homeomorphic to $M^*\setminus L^*$ are related by a sequence of the change-of-pair operations.
\end{proposition}


\subsection{Relationship with discrete Fourier transform}\label{DDFT}

In the computation of the relative Reshetikhin-Turaev invariants, the change-of-pair operation corresponds to replacing the coloring $\mathbf m_I$ on $L_I$ by the Kirby colorings $\Omega_r$ and cabling $L_I^*$ by the Chebyshev polynomials corresponding to the new coloring $\mathbf n_I,$ sending $\mathrm{RT}_r(M,L,\mathbf m)$ to $\mathrm{RT}_r(M^*,L^*,(\mathbf n_I, \mathbf m_J)),$ where $\mathbf m_J$ is the coloring on $L_J.$ See Section \ref{RRT} for a review of the relative Reshetikhin-Turaev invariants and Figure \ref{DFT} for an example.

\begin{figure}[htbp]
\centering
\includegraphics[scale=0.3]{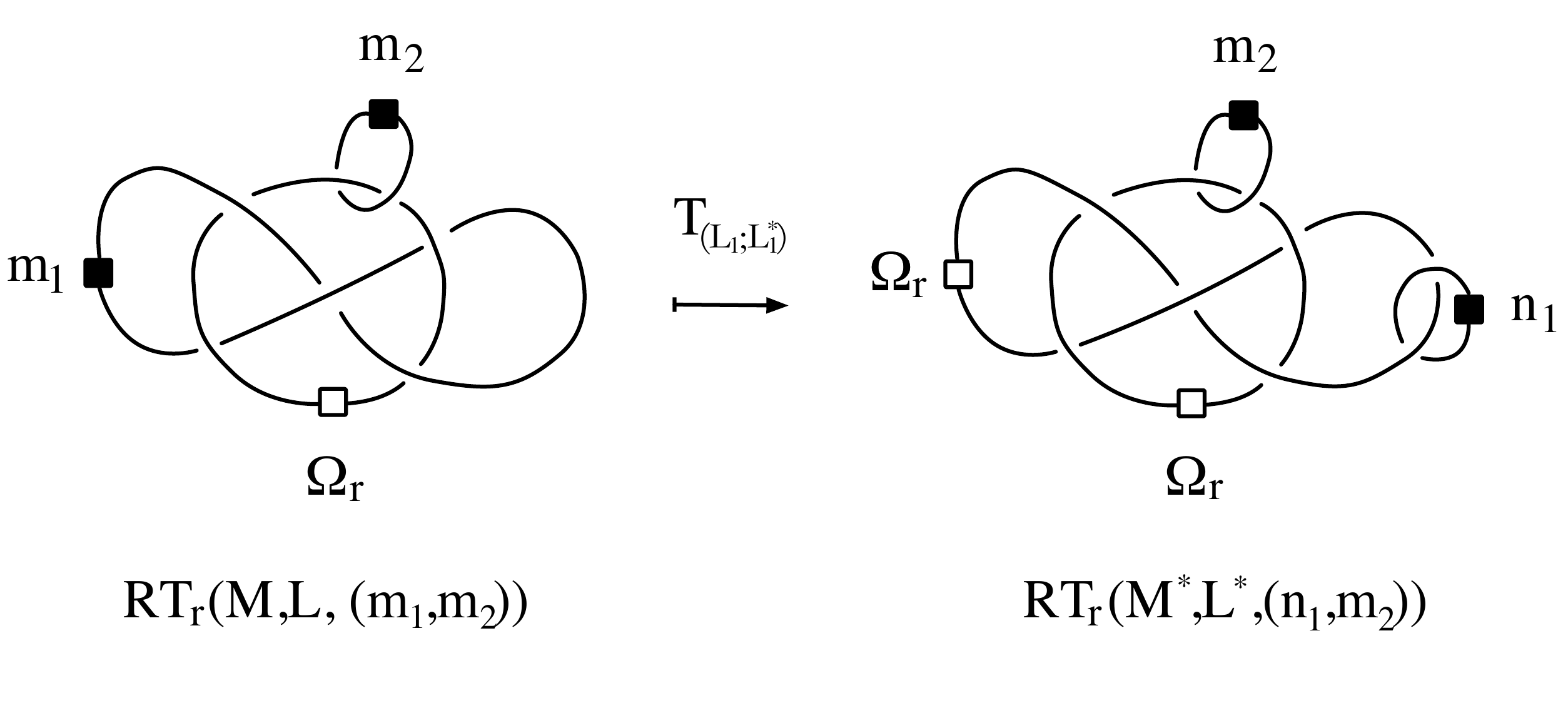}
\caption{For the computation of $\mathrm{RT}_r(M,L,(m_1,m_2))$ on the left, $L'$ is cabled by the Kirby coloring $\Omega_r$ and $L_1$ and $L_2$ are respectively cabled by the Chebyshev polynomials $e_{m_1}$ and $e_{m_2}.$ For $\mathrm{RT}_r(M^*,L^*,(n_1,m_2))$ on the right, $L'$ and $L_1$ are cabled by $\Omega_r,$ $L^*_1$ is cabled by $e_{n_1}$ and $L_2$ is cabled by $e_{m_2}.$   The framings are omitted in the Figure.}
\label{DFT}
\end{figure}

These operations pictorially represent the discrete Fourier transforms. See \cite{Ba} for the original definition and \cite{B} for an exposition in the language of skein theory. They were also shown to be a particular case of a general construction for modular tensor categories (see  \cite[Section 1]{Liu} and
references therein). To be more precise, let $\mathrm{I}_r=\{1,2,\dots, r-2\}$ and $\mu_r=\frac{\sqrt{2}\sin\frac{\pi}{r}}{\sqrt r}$  in the $SU(2)$ theory and at $q=e^{\frac{\pi \sqrt{-1}}{r}},$ and let $\mathrm{I}_r=\{0,2,\dots, r-3\}$ and $\mu_r=\frac{2\sin\frac{2\pi}{r}}{\sqrt r}$ in the $SO(3)$ theory and at $q=e^{\frac{2\pi \sqrt{-1}}{r}}.$ Let 
$$f:\mathrm{I}_r^n\to \mathbb C$$ 
be a complex valued function on $\mathrm{I}_r^n$ for some positive integer $n.$ Let $(I,J)$ be a partition of $\{1,\dots, n\}$ and let $\mathbf n_I=(n_i)_{i\in I}$  be a $|I|$-tuple of elements of $\mathrm{I}_r.$ Then the $\mathbf n_I$-th partial discrete Fourier coefficient of $f$ is the function 
$$\widehat f(\mathbf n_I):\mathrm{I}_r^J\to\mathbb C$$
defined for all $\mathbf m_J$ in $\mathrm{I}_r^J$ by
$$\widehat f(\mathbf n_I)(\mathbf m_J)=\mu_r^{|I|}\sum_{\mathbf m_I}\prod_{i\in I}\mathrm H(m_i,n_i)f(\mathbf m_I,\mathbf m_J),$$
 where the sum is over all $|I|$-tuples $\mathbf m_I=(m_i)_{i\in I}$ of elements of $\mathrm{I}_r,$ and 
$$\mathrm H(m,n)=(-1)^{m+n}\frac{q^{(m+1)(n+1)}-q^{-(m+1)(n+1)}}{q-q^{-1}}.$$

Since the coefficients $\mathrm H(m_i,n_i)$ above are exactly the coefficients of the following skein-theoretical computation 
$$\Bigg\langle \cp{\includegraphics[width=1.6cm]{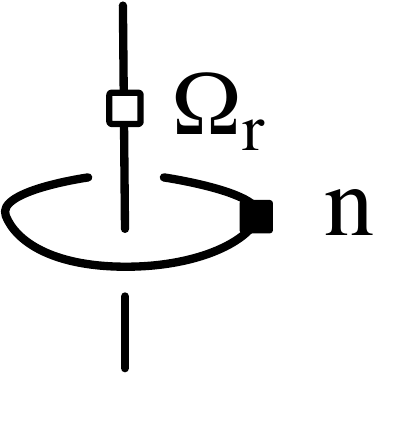}}\Bigg\rangle=\mu_r\sum_{m\in \mathrm{I}_r}\mathrm H(m,n)\Bigg\langle \cp{\includegraphics[width=0.7cm]{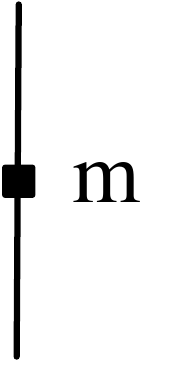}}\Bigg\rangle,$$
 the relative Reshetikhin-Turaev invariants $\mathrm{RT}_r(M^*,L^*,(\mathbf n_I, \mathbf m_J))$ of the pair $(M^*,L^*)$ obtained from $(M,L)$ by a change-of-pair operation $T_{(L_I;L^*_I)}$ is up to scalar the value at $\mathbf m_J$ of the $\mathbf n_I$-th partial discrete Fourier coefficient of the function $\mathrm{RT}_r(M,L,\_\ ):\mathrm{I}_r^n\to\mathbb C.$ The scalar is a power of $q$ depending on the framings of $L^*_I.$

It is proved in \cite{DKY, BDKY} that Turaev-Viro invariant of the complement $M\setminus L$ can be computed by the relative Reshetikhin-Turaev invariants of the pair $(M,L).$

\begin{proposition}[\cite{DKY, BDKY}]\label{TVRT} 
$$\mathrm{TV}_r(M\setminus L)=c \sum_{\mathbf m} \big|\mathrm{RT}_r(M,L,\mathbf m)\big|^2,$$
where the sum is over all multi-elements $\mathbf m$ of $\mathrm{I}_r,$ and the constant $c$ equals $1$ in the $SU(2)$ theory and equals $2^{\text{rank} H_2(M\setminus L,\mathbb Z_2)}$ in the $SO(3)$ theory. 
\end{proposition}

By Proposition \ref{equal}, if $(M^*,L^*)$ and $(M,L)$ differ by a change-of-pair operation, then $M\setminus L$ and $M^*\setminus L^*$ are homeomorphic to each other . As a consequence, we have

\begin{proposition}\label{sum}
$$\sum_{\mathbf m}\big|\mathrm{RT}_r(M,L,\mathbf m)\big|^2=\sum_{\mathbf n}\big|\mathrm{RT}_r(M^*,L^*,\mathbf n)\big|^2,$$
where the sums are over all multi-elements $\mathbf m$ and $\mathbf n$ of $\mathrm{I}_r.$ 
\end{proposition}

Propositions \ref{equal}, \ref{connect}, \ref{TVRT} and \ref{sum} together provide infinitely many different ways to compute the Turaev-Viro invariants of $M\setminus L,$ all of which are up to scalar related by a sequence of partial discrete Fourier transforms. It is hopeful that among these different expressions, some are suitable for the purpose of solving the Volume Conjecture for the Turaev-Viro invariants\,\cite{CY}.

Finally, as a special case, let $f=\mathrm{RT}_r(M,L, \_\ ):\mathrm{I}_r^n\to \mathbb C$ for a pair $(M,L),$ and suppose that $I=\{1,\dots,n\}$ and all the framings of $L^*_I$ are zero and $(M^*,L^*)$ is obtained from $(M,L)$ by $T_{(L_I,L^*_I)}.$ Then $\mathrm{RT}_r(M^*,L^*,\mathbf n)=\widehat f(\mathbf n)$ is the $\mathbf n$-th (full) discrete Fourier coefficient of $f.$ As a consequence of Proposition \ref{sum}, we have 
\begin{equation}\label{PSF}
\sum_{\mathbf m}|f(\mathbf m)|^2=\sum_{\mathbf n}|\widehat f(\mathbf n)|^2,
\end{equation}
where $\mathbf m$ and $\mathbf n$ are over all multi-elements of $\mathrm{I}_r.$ This could be considered as a Poisson Summation Formula for the discrete Fourier transforms. (See also \cite{B} for an asymptotic version of the Poisson Summation Formula in the setting of Yokota invariants for colored  planar graphs.)

\subsection{Outline of the proof of Theorem \ref{main}} 
We follow the guideline of Ohtsuki's method. In Proposition \ref{computation}, we compute the relative Reshetikhin-Turaev invariants of $(M,L)$ writing them as a sum of values of a holomorphic function $f_r$ at integer points. The function $f_r$ comes from Faddeev's quantum dilogarithm function. Using Poisson Summation Formula, we in Proposition \ref{Poisson} write the invariants as a sum of the Fourier coefficients of $f_r$ computed  in Propositions \ref{4.2}.
In Proposition \ref{crit} we show that the critical value of the functions in the leading Fourier coefficients has real part the volume and imaginary part the Chern-Simons invariant. The key observation there is a relationship between the asymptotics of quantum $6j$-symbols and the Neumann-Zagier potential function  (Proposition \ref{NeuZ}), which is of interest in its own right. Then we estimate the leading Fourier coefficients in Sections \ref{leading} using the Saddle Point Method (Proposition \ref{saddle}). Finally, we estimate the non-leading Fourier coefficients and the error term respectively in Sections \ref{ot} and \ref{ee} showing that they are neglectable, and prove Theorem \ref{main2}, which is a refined version of Theorem \ref{main}, in Section \ref{pf}.
\\

\noindent\textbf{Acknowledgments.}  The authors would like to thank Giulio Belletti, Francis Bonahon, Qingtao Chen, Effie Kalfagianni, Sanjay Kumar, Zhengwei Liu, Feng Luo, Hongbin Sun and Roland van der Veen for helpful discussions. They are also grateful for the referees'  invaluable suggestions and comments. The first author would like to thank the organizers of the ``Nearly Carbon Neutral Geometric Topology Conference 2020'' for the invitation and for providing a friendly environment for inspiring discussions. The second author is supported by NSF Grants DMS-1812008 and DMS--2203334.


\section{Preliminaries}

\subsection{Relative Reshetikhin-Turaev invariants}\label{RRT}

In this article we will follow the skein theoretical approach of the relative Reshetikhin-Turaev invariants\,\cite{BHMV, Li} and focus on the $SO(3)$-theory and the values at the root of unity  $q=e^{\frac{2\pi\sqrt{-1}}{r}}$ for odd integers $r\geqslant 3.$ To be more precise, we will follow the normalization of the invariants as described in \cite{Li}.

A framed link in an oriented $3$-manifold $M$ is a smooth embedding $L$ of a disjoint union of finitely many thickened circles $\mathrm S^1\times [0,\epsilon],$ for some $\epsilon>0,$ into $M.$ The Kauffman bracket skein module $\mathrm K_r(M)$ of $M$ is the $\mathbb C$-module generated by the isotopic classes of framed links in $M$  modulo the follow two relations: 

\begin{enumerate}[(1)]
\item  \emph{Kauffman Bracket Skein Relation:} \ $\cp{\includegraphics[width=1cm]{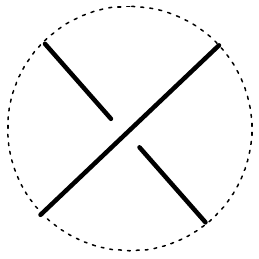}}\ =\ e^{\frac{\pi\sqrt{-1}}{r}}\ \cp{\includegraphics[width=1cm]{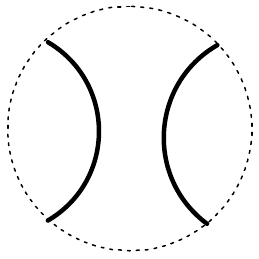}}\  +\ e^{-\frac{\pi\sqrt{-1}}{r}}\ \cp{\includegraphics[width=1cm]{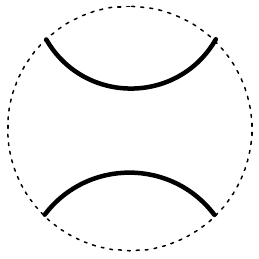}}.$ 

\item \emph{Framing Relation:} \ $L \cup \cp{\includegraphics[width=0.8cm]{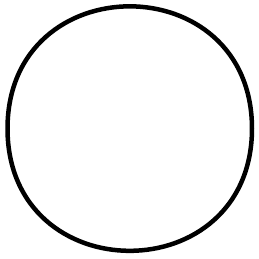}}=(-e^{\frac{2\pi\sqrt{-1}}{r}}-e^{-\frac{2\pi\sqrt{-1}}{r}})\ L.$ 
\end{enumerate}

There is a canonical isomorphism 
$$\langle\ \rangle:\mathrm K_r(\mathrm S^3)\to\mathbb C$$
defined by sending the empty link to $1.$ The image $\langle L\rangle$ of the framed link $L$ is called the Kauffman bracket of $L.$

Let $\mathrm K_r(A\times [0,1])$ be the Kauffman bracket skein module of the product of an annulus $A$ with a closed interval. For any link diagram $D$ in $\mathbb R^2$ with $k$ ordered components and $b_1, \dots, b_k\in \mathrm K_r(A\times [0,1]),$ let 
$$\langle b_1,\dots, b_k\rangle_D$$
be the complex number obtained by cabling $b_1,\dots, b_k$ along the components of $D$ considered as a element of $K_r(\mathrm S^3)$ then taking the Kauffman bracket $\langle\ \rangle.$

On $\mathrm K_r(A\times [0,1])$ there is a commutative multiplication induced by the juxtaposition of annuli, making it a $\mathbb C$-algebra; and as a $\mathbb C$-algebra $\mathrm K_r(A\times [0,1])  \cong \mathbb C[z],$ where $z$ is the core curve of $A.$ For an integer $n\geqslant 0,$ let $e_n(z)$ be the $n$-th Chebyshev polynomial defined recursively by
$e_0(z)=1,$ $e_1(z)=z$ and $e_n(z)=ze_{n-1}(z)-e_{n-2}(z).$ Let $\mathrm{I}_r=\{0,2,\dots,r-3\}$ be the set of even integers in between $0$ and $r-2.$ Then the Kirby coloring $\Omega_r\in\mathrm K_r(A\times [0,1])$ is defined by 
$$\Omega_r=\mu_r\sum_{n\in \mathrm{I}_r}[n+1]e_{n},$$
where
 $$\mu_r=\frac{2\sin\frac{2\pi}{r}}{\sqrt r}$$ 
 and $[n]$ is the quantum integer defined by
$$[n]=\frac{e^{\frac{2n\pi\sqrt{-1}}{r}}-e^{-\frac{2n\pi\sqrt{-1}}{r}}}{e^{\frac{2\pi\sqrt{-1}}{r}}-e^{-\frac{2\pi\sqrt{-1}}{r}}}.$$

Let $M$ be a closed oriented $3$-manifold and let $L$ be a framed link in $M$ with $n$ components. Suppose $M$ is obtained from $S^3$ by doing a surgery along a framed link $L',$ $D(L')$ is a standard diagram of $L'$ (ie, the blackboard framing of $D(L')$ coincides with the framing of $L'$). Then $L$ adds extra components to $D(L')$ forming a linking diagram $D(L\cup L')$ with $D(L)$ and $D(L')$ linking in possibly a complicated way. Let
$U_+$ be the diagram of the unknot with framing $1,$ $\sigma(L')$ be the signature of the linking matrix of $L'$ and $\mathbf m=(m_1,\dots,m_n)$ be a multi-elements of $I_r.$ Then the $r$-th \emph{relative Reshetikhin-Turaev invariant of $M$ with $L$ colored by $\mathbf m$} is defined as
\begin{equation}\label{RT}
\mathrm{RT}_r(M,L,\mathbf m)=\mu_r \langle e_{m_1},\dots,e_{m_n}, \Omega_r, \dots, \Omega_r\rangle_{D(L\cup L')}\langle \Omega_r\rangle _{U_+}^{-\sigma(L')}.
\end{equation}

Note that if $L=\emptyset$ or $m_1=\dots=m_n=0,$ then $\mathrm{RT}_r(M,L,\mathbf m)=\mathrm{RT}_r(M),$ the $r$-th Reshetikhin-Turaev invariant of $M;$ and if $M=S^3,$ then $\mathrm{RT}_r(M,L,\mathbf m)=\mu_r\mathrm J_{\mathbf m, L}\big(e^{\frac{4\pi\sqrt{-1}}{r}}\big),$ the value of the $\mathbf m$-th unnormalized colored Jones polynomial of $L$ at $t=e^{\frac{4\pi\sqrt{-1}}{r}}.$

\subsection{Hyperbolic cone manifolds}\label{CCS}

According to \cite{CHK}, a $3$-dimensional \emph{hyperbolic cone-manifold} is a $3$-manifold $M,$ which can be triangulated so that the link of each simplex is piecewise linear homeomorphic to a standard sphere and $M$ is equipped with a complete path metric such that the restriction of the metric to each simplex is isometric to a hyperbolic geodesic simplex. The \emph{singular locus} $L$ of a cone-manifold $M$ consists of the points with no neighborhood isometric to a ball in a Riemannian manifold. It follows that
\begin{enumerate}[(1)]
\item $L$ is a link in $M$ such that each component is a closed geodesic.

\item At each point of $L$ there is a \emph{cone angle} $\theta$ which is the sum of dihedral angles of $3$-simplices containing the point.

\item The restriction of the metric on $M\setminus L$ is a smooth hyperbolic metric, but is incomplete if $L\neq \emptyset.$
\end{enumerate}

Hodgson-Kerckhoff\,\cite{HK} proved that hyperbolic cone metrics on $M$ with singular locus $L$ are locally parametrized by the cone angles provided all the cone angles are less than or equal to $2\pi,$ and Kojima\,\cite{Ko} proved that hyperbolic cone manifolds $(M,L)$ are globally rigid provided all the cone angles are less than or equal to $\pi.$ It is expected to be  globally rigid if all the cone angles are less than or equal to $2\pi.$

Given a $3$-manifold $N$ with boundary a union of tori $T_1,\dots,T_n,$ a choice of generators $(u_i,v_i)$ for each $\pi_1(T_i)$ and pairs of relatively prime integers $(p_i, q_i),$ one can do the $(\frac{p_1}{q_1},\dots,\frac{p_n}{q_n})$-Dehn filling on $N$ by attaching a solid torus to each $T_i$ so that $p_iu_i+q_iv_i$ bounds a disk.  If $\mathrm H(u_i)$ and  $\mathrm H(v_i)$ are respectively the logarithmic holonomy for $u_i$ and $v_i,$ then a solution to  
\begin{equation}\label{DF}
p_i\mathrm H(u_i)+q_i\mathrm H(v_i)=\sqrt{-1}\theta_i
\end{equation}
 near the complete structure gives a cone-manifold structure on the resulting manifold $M$ with the cone angle $\theta_i$ along the core curve $L_i$ of the solid torus attached to $T_i;$ it is a smooth structure if $\theta_1=\dots=\theta_n= 2\pi.$

In this setting, the Chern-Simons invariant for a hyperbolic cone manifold $(M,L)$ can be defined by using the Neumann-Zagier potential function\,\cite{NZ}. To do this, we need a framing on each component, namely, a choice of a curve $\gamma_i$ on $T_i$ that is isotopic to the core curve $L_i$ of the solid torus attached to $T_i.$ We choose the orientation of $\gamma_i$ so that $(p_iu_i+q_iv_i)\cdot \gamma_i=1.$ Then we consider the following function 
$$\frac{\Phi(\mathrm{H}(u_1),\dots,\mathrm{H}(u_n))}{\sqrt{-1}}-\sum_{i=1}^n\frac{\mathrm H(u_i)\mathrm H(v_i)}{4\sqrt{-1}}+\sum_{i=1}^n\frac{\theta_i\mathrm H(\gamma_i)}{4},$$
where $\Phi$ is the Neumann-Zagier potential function (see \cite{NZ}) defined on the deformation space of hyperbolic structures on $M\setminus L$ parametrized by the holonomy of the meridians $\{\mathrm H(u_i)\},$ characterized by 
\begin{equation}\label{char}
\left \{\begin{array}{l}
\frac{\partial \Phi(\mathrm{H}(u_1),\dots,\mathrm{H}(u_n))}{\partial \mathrm{H}(u_i)}=\frac{\mathrm H(v_i)}{2},\\
\\
\Phi(0,\dots,0)=\sqrt{-1}\bigg(\mathrm{Vol}(M\setminus L)+\sqrt{-1}\mathrm{CS}(M\setminus L)\bigg)\quad\quad\mathrm{mod}\ \pi^2\mathbb Z,
\end{array}\right.
\end{equation} 
where $M\setminus L$ is with the complete hyperbolic metric. Another important feature of $\Phi$ is that it is even in each of its variables $\mathrm H(u_i).$

Following the argument in \cite[Sections 4 \& 5]{NZ}, one can prove that if  the cone angles of components of $L$ are $\theta_1,\dots,\theta_n,$ then
\begin{equation}\label{VOL}
\mathrm{Vol}(M_{L_{\boldsymbol\theta}})=\mathrm{Re}\bigg(\frac{\Phi(\mathrm{H}(u_1),\dots,\mathrm{H}(u_n))}{\sqrt{-1}}-\sum_{i=1}^n\frac{\mathrm H(u_i)\mathrm H(v_i)}{4\sqrt{-1}}+\sum_{i=1}^n\frac{\theta_i\mathrm H(\gamma_i)}{4}\bigg).
\end{equation}
Indeed, in this case, one can replace the $2\pi$ in Equations (33) (34) and (35) of \cite{NZ} by $\theta_i,$ and as a consequence can replace the $\frac{\pi}{2}$ in Equations (45), (46) and (48) by $\frac{\theta_i}{4},$ proving the result.

In \cite{Y}, Yoshida proved that when $\theta_1=\dots=\theta_n=2\pi,$
$$\mathrm{Vol}(M)+\sqrt{-1}\mathrm{CS}(M)=\frac{\Phi(\mathrm{H}(u_1),\dots,\mathrm{H}(u_n))}{\sqrt{-1}}-\sum_{i=1}^n\frac{\mathrm H(u_i)\mathrm H(v_i)}{4\sqrt{-1}}+\sum_{i=1}^n\frac{\theta_i\mathrm H(\gamma_i)}{4}\quad\quad\text{mod }\sqrt{-1}\pi^2\mathbb Z.$$

Therefore, we can make the following 

\begin{definition}\label{CS} The \emph{Chern-Simons invariant} of a hyperbolic cone manifold $M_{L_{\boldsymbol\theta}}$ with a choice of the framing $(\gamma_1,\dots,\gamma_n)$ 
is defined as
$$\mathrm{CS}(M_{L_{\boldsymbol\theta}})=\mathrm{Im}\bigg(\frac{\Phi(\mathrm{H}(u_1),\dots,\mathrm{H}(u_n))}{\sqrt{-1}}-\sum_{i=1}^n\frac{\mathrm H(u_i)\mathrm H(v_i)}{4\sqrt{-1}}+\sum_{i=1}^n\frac{\theta_i\mathrm H(\gamma_i)}{4}\bigg)\quad\quad\text{mod }\pi^2\mathbb Z.$$
\end{definition}

Then together with (\ref{VOL}), we have
\begin{equation}\label{VCS}
\mathrm{Vol}(M_{L_{\boldsymbol\theta}})+\sqrt{-1}\mathrm{CS}(M_{L_{\boldsymbol\theta}})=\frac{\Phi(\mathrm{H}(u_1),\dots,\mathrm{H}(u_n))}{\sqrt{-1}}-\sum_{i=1}^n\frac{\mathrm H(u_i)\mathrm H(v_i)}{4\sqrt{-1}}+\sum_{i=1}^n\frac{\theta_i\mathrm H(\gamma_i)}{4}\quad\quad\text{mod }\sqrt{-1}\pi^2\mathbb Z.
\end{equation}

\begin{remark} It is an interesting question to find a direct geometric definition of the Chern-Simons invariants for hyperbolic cone manifolds.
\end{remark}

\subsection{Quantum $6j$-symbols} 

A triple $(m_1,m_2,m_3)$ of even integers in $\{0,2,\dots,r-3\}$ is \emph{$r$-admissible} if 
\begin{enumerate}[(1)]
\item  $m_i+m_j-m_k\geqslant 0$ for $\{i,j,k\}=\{1,2,3\},$ and
\item $m_1+m_2+m_3\leqslant 2(r-2).$
\end{enumerate}

For an $r$-admissible triple $(m_1,m_2,m_3),$ define 
$$\Delta(m_1,m_2,m_3)=\sqrt{\frac{[\frac{m_1+m_2-m_3}{2}]![\frac{m_2+m_3-m_1}{2}]![\frac{m_3+m_1-m_2}{2}]!}{[\frac{m_1+m_2+m_3}{2}+1]!}}$$
with the convention that $\sqrt{x}=\sqrt{|x|}\sqrt{-1}$ when the real number $x$ is negative, and recall that the quantum factorial  is defined as 
$$[n]!=\prod_{k=1}^n[k]$$
with the convention that  $[0]!=1.$

A  6-tuple $(m_1,\dots,m_6)$  is \emph{$r$-admissible} if the triples $(m_1,m_2,m_3),$ $(m_1,m_5,m_6),$ $(m_2,m_4,m_6)$ and $(m_3,m_4,m_5)$ are $r$-admissible

\begin{definition}
The \emph{quantum $6j$-symbol} of an $r$-admissible 6-tuple $(m_1,\dots,m_6)$ is 
\begin{multline*}
\bigg|\begin{matrix} m_1 & m_2 & m_3 \\ m_4 & m_5 & m_6 \end{matrix} \bigg|
= \sqrt{-1}^{-\sum_{i=1}^6m_i}\Delta(m_1,m_2,m_3)\Delta(m_1,m_5,m_6)\Delta(m_2,m_4,m_6)\Delta(m_3,m_4,m_5)\\
\sum_{k=\max \{T_1, T_2, T_3, T_4\}}^{\min\{ Q_1,Q_2,Q_3\}}\frac{(-1)^k[k+1]!}{[k-T_1]![k-T_2]![k-T_3]![k-T_4]![Q_1-k]![Q_2-k]![Q_3-k]!},
\end{multline*}
where $T_1=\frac{m_1+m_2+m_3}{2},$ $T_2=\frac{m_1+m_5+m_6}{2},$ $T_3=\frac{m_2+m_4+m_6}{2}$ and $T_4=\frac{m_3+m_4+m_5}{2},$ $Q_1=\frac{m_1+m_2+m_4+m_5}{2},$ $Q_2=\frac{m_1+m_3+m_4+m_6}{2}$ and $Q_3=\frac{m_2+m_3+m_5+m_6}{2}.$
\end{definition}

Here we recall a classical result of Costantino \cite{C1} which was originally stated at the root of unity $q=e^{\frac{\pi \sqrt{-1}}{r}}.$ At the root of unity $q=e^{\frac{2\pi\sqrt{-1}}{r}},$ see \cite[Appendix]{BDKY} for a detailed proof.

\begin{theorem}[\cite{C1}]\label{Vol}
Let $\{(m_1^{(r)},\dots,m_6^{(r)})\}$ be a sequence of $r$-admissible
$6$-tuples, and
let 
$$\theta_i=\Big|\pi-\lim_{r\rightarrow\infty}\frac{2\pi m_i^{(r)}}{r}\Big|.$$
If $\theta_1,\dots,\theta_6$ are the dihedral angles of a truncated hyperideal tetrahedron $\Delta,$ then
as $r$ varies over all the odd integers 
$$\lim_{r\to\infty}\frac{2\pi}{r}\log \bigg|\begin{array}{ccc}m_1^{(r)} & m_2^{(r)} & m_3^{(r)} \\m_4^{(r)} & m_5^{(r)} & m_6^{(r)} \\\end{array} \bigg|_{q=e^{\frac{2\pi \sqrt{-1}}{r}}}=Vol(\Delta).$$
\end{theorem}

Closely related, a triple $(\alpha_1,\alpha_2,\alpha_3)\in [0,2\pi]^3$ is \emph{admissible} if 
\begin{enumerate}[(1)]
\item $\alpha_i+\alpha_j-\alpha_k\geqslant 0$ for $\{i,j,k\}=\{1,2,3\},$ and
\item $\alpha_i+\alpha_j+\alpha_k\leqslant 4\pi.$
\end{enumerate}
A $6$-tuple $(\alpha_1,\dots,\alpha_6)\in [0,2\pi]^6$ is \emph{admissible} if the triples $(\alpha_1,\alpha_2,\alpha_3),$ $(\alpha_1,\alpha_5,\alpha_6),$ $(\alpha_2,\alpha_4,\alpha_6)$ and $(\alpha_3,\alpha_4,\alpha_5)$ are admissible.


\subsection{Fundamental shadow links}\label{fsl}

In this section we recall the construction and basic properties of the fundamental shadow links. The building blocks for the fundamental shadow links are truncated tetrahedra as in the left of Figure \ref{bb}. If we take $c$ building blocks $\Delta_1,\dots, \Delta_c$ and glue them together along the triangles of truncation, we obtain a (possibly non-orientable) handlebody of genus $c+1$ with a link in its boundary consisting of the edges of the building blocks, such as in the right of Figure \ref{bb}. By taking the orientable double (the orientable double
covering with the boundary quotient out by the deck involution) of this handlebody, we obtain a link $L_{\text{FSL}}$ inside $M_c=\#^{c+1}(S^2\times S^1).$ We call a link obtained this way a \emph{fundamental shadow link}, and its complement in $M_c$ a \emph{fundamental shadow link complement}.  Alternatively, to construct a fundamental shadow link complement, we can also take the double of each tetrahedron first along the hexagonal faces and then glue the resulting pieces together along homeomorphisms between the $3$-puncture spheres coming from the double of the triangles of truncation.

\begin{figure}[htbp]
\centering
\includegraphics[scale=0.25]{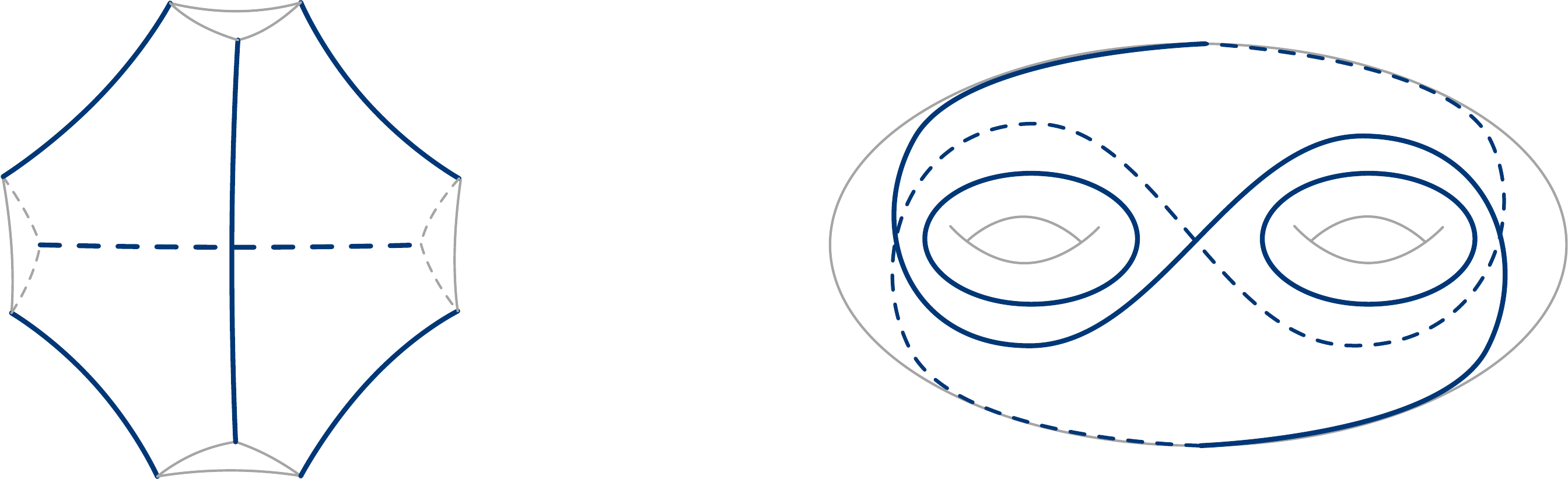}
\caption{The handlebody on the right is obtained from the truncated tetrahedron on the left by identifying the triangles on the top and the bottom by a horizontal reflection and the triangles on the left and the right by a vertical reflection.}
\label{bb}
\end{figure}

The fundamental importance of the family of the fundamental shadow links is the following.

\begin{theorem}[\cite{CT}]\label{CT} Any compact oriented $3$-manifold with toroidal or empty boundary can be obtained from a suitable fundamental shadow link complement by doing an integral Dehn-filling to some of the boundary components.
\end{theorem}

A hyperbolic cone metric on $M_c$ with singular locus $L_{\text{FSL}}$ and with sufficiently small cone angles $\theta_1,\dots,\theta_n$ can be constructed as follows. For each $s\in \{1,\dots, c\},$ let $e_{s_1},\dots,e_{s_6}$ be the edges of the building block $\Delta_s,$ and let $\theta_{s_j}$ be the cone angle of the component of $L$ containing $e_{s_j}.$ If $\theta_i$'s are sufficiently small, then $\big\{\frac{\theta_{s_1}}{2},\dots,\frac{\theta_{s_6}}{2}\big\}$ form the set of dihedral angles of a truncated hyperideal tetrahedron, by abuse of notation still denoted by $\Delta_s.$ Then the hyperbolic cone manifold $M_c$ with singular locus $L_{\text{FSL}}$ and cone angles $\theta_1,\dots,\theta_n$ is obtained by glueing $\Delta_s$'s together along isometries of the triangles of truncation, and taking the double. In this metric, the logarithmic holonomy of the meridian $u_i$ of the tubular neighborhood $N(L_i)$ of $L_i$ satisfies 
\begin{equation}\label{m}
\mathrm{H}(u_i)=\sqrt{-1}\theta_i.
\end{equation}
A preferred longitude $v_i$ on the boundary of $N(L_i)$ can be chosen as follows. Recall that a fundamental shadow link is obtained from the double of a set of truncated tetrahedra (along the hexagonal faces) glued together by orientation preserving homeomorphisms between the trice-punctured spheres coming from the double of the triangles of truncation, and recall also that 
 the mapping class group of trice-punctured sphere is generated by mutations, which could be represented by the four $3$-braids in Figure \ref{mutation}. For each mutation, we assign an integer $\pm 1$ to each component of the braid as in Figure \ref{mutation}; and for a composition of a sequence of mutations, we assign the sum of the  $\pm 1$ assigned by the mutations to each component of the $3$-braid.

 \begin{figure}[htbp]
\centering
\includegraphics[scale=0.5]{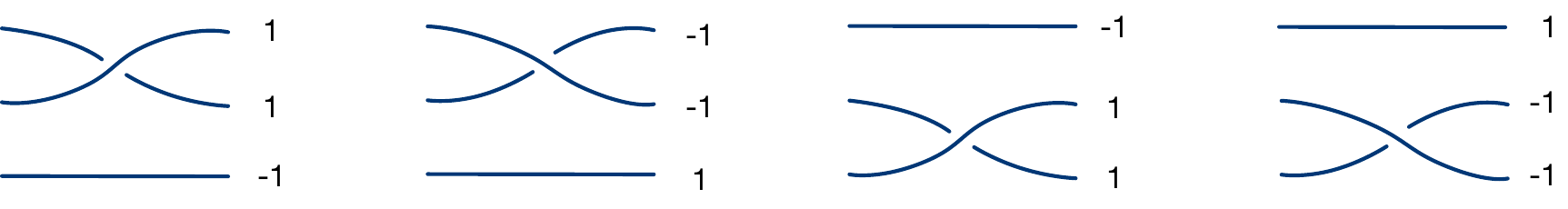}
\caption{ }
\label{mutation}
\end{figure}

  In this way, each  orientation preserving homeomorphisms between the trice-punctured spheres assigns three integers to three of the components of $L_{\text{FSL}},$ one for each. For each $i\in \{1,\dots, n\},$ let $\iota_i$ be the sum of all the integers on $L_i$ assigned by the homeomorphisms between the trice-punctured spheres. Then we can choose a preferred longitude $v_i$ such that $u_i\cdot v_i=1$ and the logarithmic holomony satisfies
\begin{equation}\label{l}
\mathrm{H}(v_i)=-l_i+\frac{\iota_i\sqrt{-1}\theta_i}{2},
\end{equation}
where $l_i$ is the length of the closed geodesic $L_i.$ In this way, a framing on $L_i$ gives an integer $p_i$ in the way that the parallel copy of $L_i$ on $N(L_i)$ is isotopic to the curve representing $p_iu_i+v_i.$

\begin{proposition}[\cite{C1,C2}]\label{FSL} If $L_{\text{FSL}}=L_1\cup\dots\cup L_n\subset M_c$ is a framed fundamental shadow link with framing $p_i$ on $L_i,$ and $\mathbf m=(m_1,\dots,m_n)$ is a coloring of its components with even integers in $\{0,2,\dots,r-3\},$ then
$$\mathrm{RT_r}(M_c,L_{\text{FSL}},\mathbf m)=\Bigg(\frac{2\sin\frac{2\pi}{r}}{\sqrt{r}}\Bigg)^{-c}\prod_{i=1}^n(-1)^{\frac{\iota_im_i}{2}}q^{(p_i+\frac{\iota_i}{2})\frac{m_i(m_i+2)}{2}}\prod_{s=1}^c\bigg|\begin{matrix}
        m_{s_1} & m_{s_2} & m_{s_3} \\
        m_{s_4} & m_{s_5} & m_{s_6} 
      \end{matrix} \bigg|,$$
where $m_{s_1},\dots,m_{s_6}$ are the colors of the edges of the building block $\Delta_s$ inherited  from the color $\mathbf m$ on $L_{\text{FSL}}.$    
\end{proposition}

Next, we talk about the volume and the Chern-Simons invariant of $M_c \setminus L_{\text{FSL}}$ at the complete hyperbolic structure. In the complete hyperbolic metric, since $M_c \setminus L_{\text{FSL}}$ is the union of $2c$ regular ideal octahedra, we have 
\begin{equation} \label{VolFSL}
\mathrm{Vol}(M_c \setminus L_{\text{FSL}}) = 2c v_8,
\end{equation} 
where $v_8$ is the volume of the regular ideal octahedron.

For the Chern-Simons invariant, in the case that the truncated tetrahedra $\Delta_1,\dots, \Delta_c$ are glued together along the triangles of truncation via orientation reversing maps, $M_c\setminus L_{\text{FSL}}$ is the ordinary double of the orientable handlebody, which admits an orientation reversing self-homeomorphism. Hence by \cite[Corollary 2.5]{MO}, \begin{equation*}
\mathrm {CS}(M_c \setminus L_{\text{FSL}}) = 0      \quad\quad\text{mod }\pi^2\mathbb Z
\end{equation*}
at the complete hyperbolic structure. In the general case, a fundamental shadow link complement $M_c\setminus L_{\text{FSL}}$ can be obtained from one from the previous case by doing a sequence of mutations along the thrice-punctured spheres coming from the double of the triangles of truncation. Therefore, by \cite[Theorem 2.4]{MR} that a mutation along an incompressible trice-punctured sphere in a hyperbolic three manifold changes the Chern-Simons invariant by $\frac{\pi^2}{2},$ we have
\begin{equation}\label{CSFSL}
\mathrm {CS}(M_c \setminus L_{\text{FSL}}) = \Big(\sum_{i=1}^n \frac{\iota_i}{2}\Big) \pi^2  \quad\quad\text{mod }\pi^2\mathbb Z.
\end{equation}

Together with Theorem \ref{Vol} and the construction of the hyperbolic cone structure, we see that Conjecture \ref{VC} is true for $(M_c,L_{\text{FSL}}).$ This was first proved by Costantino in \cite{C1} at the root of unity $q=e^{\frac{\pi \sqrt{-1}}{r}}.$

\subsection{Dilogarithm and quantum dilogarithm functions}

Let $\log:\mathbb C\setminus (-\infty, 0]\to\mathbb C$ be the standard logarithm function defined by
$$\log z=\log|z|+\sqrt{-1}\arg z$$
with $-\pi<\arg z<\pi.$
 
The dilogarithm function $\mathrm{Li}_2: \mathbb C\setminus (1,\infty)\to\mathbb C$ is defined by
$$\mathrm{Li}_2(z)=-\int_0^z\frac{\log (1-u)}{u}du$$
where the integral is along any path in $\mathbb C\setminus (1,\infty)$ connecting $0$ and $z,$ which is holomorphic in $\mathbb C\setminus [1,\infty)$ and continuous in $\mathbb C\setminus (1,\infty).$

The dilogarithm function satisfies the follow properties (see eg. Zagier\,\cite{Z}).
\begin{enumerate}[(1)]
\item \begin{equation}\label{Li2}
\mathrm{Li}_2\Big(\frac{1}{z}\Big)=-\mathrm{Li}_2(z)-\frac{\pi^2}{6}-\frac{1}{2}\big(\log(-z)\big)^2.
\end{equation} 
\item In the unit disk $\big\{z\in\mathbb C\,\big|\,|z|<1\big\},$ 
\begin{equation}\label{Li1}
\mathrm{Li}_2(z)=\sum_{n=1}^\infty\frac{z^n}{n^2}.
\end{equation}
\item On the unit circle $\big\{ z=e^{2\sqrt{-1}\theta}\,\big|\,0 \leqslant \theta\leqslant\pi\big\},$ 
\begin{equation}\label{dilogLob}
\mathrm{Li}_2(e^{2\sqrt{-1}\theta})=\frac{\pi^2}{6}+\theta(\theta-\pi)+2\sqrt{-1}\Lambda(\theta).
\end{equation}
\end{enumerate}
Here  $\Lambda:\mathbb R\to\mathbb R$  is the Lobachevsky function defined by
$$\Lambda(\theta)=-\int_0^\theta\log|2\sin t|dt,$$
which is an odd function of period $\pi.$ See eg. Thurston's notes\,\cite[Chapter 7]{T}.


The following variant of Faddeev's quantum dilogarithm functions\,\cite{F, FKV} will play a key role in the proof of the main result. 
Let $r\geqslant 3$ be an odd integer. Then the following contour integral
\begin{equation}
\varphi_r(z)=\frac{4\pi\sqrt{-1}}{r}\int_{\Omega}\frac{e^{(2z-\pi)x}}{4x \sinh (\pi x)\sinh (\frac{2\pi x}{r})}\ dx
\end{equation}
defines a holomorphic function on the domain $$\Big\{z\in \mathbb C \ \Big|\ -\frac{\pi}{r}<\mathrm{Re}z <\pi+\frac{\pi}{r}\Big\},$$  
  where the contour is
$$\Omega=\big(-\infty, -\epsilon\big]\cup \big\{z\in \mathbb C\ \big||z|=\epsilon, \mathrm{Im}z>0\big\}\cup \big[\epsilon,\infty\big),$$
for some $\epsilon\in(0,1).$
Note that the integrand has poles at $n\sqrt{-1},$ $n\in\mathbb Z,$ and the choice of  $\Omega$ is to avoid the pole at $0.$
\\

The function $\varphi_r(z)$ satisfies the following fundamental properties, whose proof can be found in \cite[Section 2.3]{WY}. 
\begin{lemma}
\begin{enumerate}[(1)]
\item For $z\in\mathbb C$ with  $0<\mathrm{Re}z<\pi,$
\begin{equation}\label{fund}
1-e^{2 \sqrt{-1}z}=e^{\frac{r}{4\pi\sqrt{-1}}\Big(\varphi_r\big(z-\frac{\pi}{r}\big)-\varphi_r\big(z+\frac{\pi}{r}\big)\Big)}.
 \end{equation}
 
 \item For $z\in\mathbb C$ with  $-\frac{\pi}{r}<\mathrm{Re}z<\frac{\pi}{r},$
 \begin{equation}\label{f2}
1+e^{r\sqrt{-1}z}=e^{\frac{r}{4\pi\sqrt{-1}}\Big(\varphi_r(z)-\varphi_r\big(z+\pi\big)\Big)}.
\end{equation}
\end{enumerate}
\end{lemma}

Using (\ref{fund}) and (\ref{f2}), for $z\in\mathbb C$ with $\pi+\frac{2(n-1)\pi}{r}< \mathrm{Re}z< \pi+\frac{2n\pi}{r},$ we can define $\varphi_r(z)$  inductively by the relation
\begin{equation}\label{extension}
\prod_{k=1}^n\Big(1-e^{2 \sqrt{-1} \big(z-\frac{(2k-1)\pi}{r}\big)}\Big)=e^{\frac{r}{4\pi\sqrt{-1}}\Big(\varphi_r\big(z-\frac{2n\pi}{r}\big)-\varphi_r(z)\Big)},
\end{equation}
extending $\varphi_r(z)$ to a meromorphic function on $\mathbb C.$  The poles of $\varphi_r(z)$ have the form $(a+1)\pi+\frac{b\pi}{r}$ or $-a\pi-\frac{b\pi}{r}$ for all nonnegative integer $a$ and positive odd integer $b.$

Let $q=e^{\frac{2\pi\sqrt{-1}}{r}},$
and let $$(q)_n=\prod_{k=1}^n(1-q^{2k}).$$

\begin{lemma}\label{fact}
\begin{enumerate}[(1)]
\item For $0\leqslant n \leqslant r-2,$
\begin{equation}
(q)_n=e^{\frac{r}{4\pi\sqrt{-1}}\Big(\varphi_r\big(\frac{\pi}{r}\big)-\varphi_r\big(\frac{2\pi n}{r}+\frac{\pi}{r}\big)\Big)}.
\end{equation}
\item For $\frac{r-1}{2}\leqslant n \leqslant r-2,$
\begin{equation} \label{move}
(q)_n=2e^{\frac{r}{4\pi\sqrt{-1}}\Big(\varphi_r\big(\frac{\pi}{r}\big)-\varphi_r\big(\frac{2\pi n}{r}+\frac{\pi}{r}-\pi\big)\Big)}.
\end{equation}
\end{enumerate}
\end{lemma}

We consider (\ref{move}) because there are poles in $(\pi,2\pi),$ and to avoid the poles we move the variables to $(0,\pi)$ by subtracting $\pi.$

If we let 
$$\{n\}=q^{n}-q^{-n}$$
and let 
$$\{n\}!=\prod_{k=1}^n\{k\}\quad\text{and}\quad \{0\}!=0,$$
then
$$\{n\}!=(-1)^nq^{-\frac{n(n+1)}{2}}(q)_n.$$
As a consequence of Lemma \ref{fact}, we have

\begin{lemma}\label{factorial}
\begin{enumerate}[(1)]
\item For $0\leqslant n \leqslant r-2,$
\begin{equation}
\{n\}!=e^{\frac{r}{4\pi\sqrt{-1}}\Big(-2\pi\big(\frac{2\pi n}{r}\big)+\big(\frac{2\pi}{r}\big)^2(n^2+n)+\varphi_r\big(\frac{\pi}{r}\big)-\varphi_r\big(\frac{2\pi n}{r}+\frac{\pi}{r}\big)\Big)}.
\end{equation}
\item For $\frac{r-1}{2}\leqslant n \leqslant r-2,$
\begin{equation} 
\{n\}!=2e^{\frac{r}{4\pi\sqrt{-1}}\Big(-2\pi\big(\frac{2\pi n}{r}\big)+\big(\frac{2\pi }{r}\big)^2(n^2+n)+\varphi_r\big(\frac{\pi}{r}\big)-\varphi_r\big(\frac{2\pi n}{r}+\frac{\pi}{r}-\pi\big)\Big)}.
\end{equation}
\end{enumerate}
\end{lemma}

The function $\varphi_r(z)$ and the dilogarithm function are closely related as follows.

\begin{lemma}\label{converge}  \begin{enumerate}[(1)]
\item For every $z$ with $0<\mathrm{Re}z<\pi,$ 
\begin{equation}\label{conv}
\varphi_r(z)=\mathrm{Li}_2(e^{2\sqrt{-1}z})+\frac{2\pi^2e^{2\sqrt{-1}z}}{3(1-e^{2\sqrt{-1}z})}\frac{1}{r^2}+O\Big(\frac{1}{r^4}\Big).
\end{equation}
\item For every $z$ with $0<\mathrm{Re}z<\pi,$ 
\begin{equation}\label{conv}
\varphi_r'(z)=-2\sqrt{-1}\log(1-e^{2\sqrt{-1}z})+O\Big(\frac{1}{r^2}\Big).
\end{equation}
\item \cite[Formula (8)(9)]{O2}
$$\varphi_r\Big(\frac{\pi}{r}\Big)=\mathrm{Li}_2(1)+\frac{2\pi\sqrt{-1}}{r}\log\Big(\frac{r}{2}\Big)-\frac{\pi^2}{r}+O\Big(\frac{1}{r^2}\Big).$$
\end{enumerate}\end{lemma}



\section{Computation of the relative Reshetikhin-Turaev invariants}\label{comp}

The goal of this section is to compute the relative Reshetikhin-Turaev invariants of $(M,L).$  In Proposition \ref{computation}, we write the invariants as a sum of values of a fixed holomorphic function at the integer points. The holomorphic function comes from Faddeev's quantum dilogarithm function. Using the Poisson Summation Formula, we in Proposition \ref{Poisson} write the invariants as a sum of the Fourier coefficients of the holomorphic function, which is computed  in Propositions \ref{4.2}.

 Let $L_{\text{FSL}}=L_1\cup\dots\cup L_n$  be a fundamental shadow link in $M_c=\#^{c+1}(S^2\times S^1),$ and let $L'\subset S^3$ be the disjoint union of $c+1$ unknots with the $0$-framings by doing surgery along which we get $M_c.$ Let $(I,J)$ be a partition of $\{1,\dots,n\},$ 
 and let $(M,L)$ be the pair obtained from $(M_c,L_{\text{FSL}})$ by doing a change-of-pair operation $T_{(L_I;L^*_I)}$ as introduced in Section \ref{COP}, ie., $M=(M_c)_{L_I}$ and $L=L^*_I\cup L_J,$ where $L^*_I=\cup_{i\in I}L^*_i$ and $L^*_i$ is the framed unknot in $M_c\setminus L_{\text{FSL}}$ with the core curve isotopic to the meridian of the tubular neighborhood of $L_i.$  Let $\mathbf n_I$ be a coloring of $L^*_I$ and let $\mathbf m_J$ be a coloring of $L_J.$  Then Theorem \ref{main} can be rephrased as follows.

 \begin{theorem}\label{main2} For $i\in I,$ let
$$\theta_i=\Big|2\pi-\lim_{r\to\infty}\frac{4\pi n_i}{r}\Big|;$$
and for $j\in J,$ let 
$$\theta_j=\Big|2\pi-\lim_{r\to\infty}\frac{4\pi m_j}{r}\Big|.$$
Let $\boldsymbol\theta=(\theta_1,\dots,\theta_n)$ and let $M_{L_{\boldsymbol\theta}}$ be the hyperbolic cone manifold consisting of $M$ and a hyperbolic cone metric on $M$  with singular locus $L$ and cone angles $\boldsymbol\theta.$ Then there exists an $\epsilon>0$ such that if all the cone angles $\theta_k$ are less than $\epsilon,$ then as $r$ varies over all the odd integers

$$\lim_{r\to\infty} \frac{4\pi }{r}\log \mathrm{RT}_r(M,L,(\mathbf n_I,\mathbf m_J))=\mathrm{Vol}(M_{L_{\boldsymbol\theta}})+\sqrt{-1}\mathrm{CS}(M_{L_{\boldsymbol\theta}})\quad\quad\text{mod } \sqrt{-1}\pi^2\mathbb Z.$$
\end{theorem}

The goal of the rest of this paper is to prove Theorem \ref{main2}. 

Suppose $L^*_i$ has the framing $q_i$ and $L_i$ has the framing  $p_i$  for each $i\in I,$ and  $L_j$ has the framing $p_j$ for each $j\in J.$  Then the $r$-th relative Reshetikin-Turaev invariant of $M$ with $L$ colored by $(\mathbf n_I,\mathbf m_J)$ can be computed as
\begin{equation}\label{rt}
\begin{split}
\mathrm{RT}_r&(M,L,(\mathbf n_I,\mathbf m_J))\\
=&\bigg(\frac{2\sin\frac{2\pi}{r}}{\sqrt{r}}\bigg)^{|I|-c}e^{-\sigma(L'\cup L_I)\big(-\frac{3}{r}-\frac{r+1}{4}\big)\sqrt{-1}\pi}\prod_{i\in I}q^{\frac{q_in_i(n_i+2)}{2}}\prod_{j\in J}(-1)^{\frac{\iota_jm_j}{2}}q^{\big(p_j+\frac{\iota_j}{2}\big)\frac{m_j(m_j+2)}{2}}\\
&\Bigg(\sum_{\mathbf m_I}\prod_{i\in I}(-1)^{\frac{\iota_im_i}{2}}q^{\big(p_i+\frac{\iota_i}{2}\big)\frac{m_i(m_i+2)}{2}}[(m_i+1)(n_i+1)]\prod_{s=1}^c\bigg|\begin{matrix}
        m_{s_1} & m_{s_2} & m_{s_3} \\
        m_{s_4} & m_{s_5} & m_{s_6} 
      \end{matrix} \bigg|\Bigg),
\end{split}
\end{equation}
where the sum is over all multi-even integers $\mathbf m_I=(m_i)_{i\in I}$ in $\{0,2,\dots,r-3\},$ and $m_{s_1},\dots,m_{s_6}$ are the colors of the edges of the building block $\Delta_s$ inherited  from the colors on $L_{\text{FSL}}.$

 In the rest of this section, we aim to write $\mathrm{RT}_r(M,L,(\mathbf n_I;\mathbf m_J))$ into a sum of integrals using the Poisson Summation Formula. This requires writing the invariant into the sum of the values of a fixed holomorphic function. To this end, we look at the a single quantum $6j$-symbol first.

\begin{definition} An $r$-admissible $6$-tuple $(m_1,\dots,m_6)$  is of the \emph{hyperideal type} if for $\{i,j,k\}=\{1,2,3\},$ $\{1,5,6\},$ $\{2,4,6\}$ and $\{3,4,5\},$
\begin{enumerate}[(1)]
\item $0\leqslant m_i+m_j-m_k<r-2,$ and
\item $r-2\leqslant m_i+m_j+m_k\leqslant 2(r-2).$
\end{enumerate}
\end{definition}

\begin{definition} A $6$-tuple $(\alpha_1,\dots,\alpha_6)\in [0,2\pi]^6$ is of the \emph{hyperideal type} if for $\{i,j,k\}=\{1,2,3\},$ $\{1,5,6\},$ $\{2,4,6\}$ and $\{3,4,5\},$
\begin{enumerate}[(1)]
\item $0\leqslant \alpha_i+\alpha_j-\alpha_k\leqslant 2\pi,$ and
\item $2\pi\leqslant \alpha_i+\alpha_j+\alpha_k\leqslant 4\pi.$
\end{enumerate}
\end{definition}
We notice that the six numbers $|\alpha_1-\pi|,\dots,|\alpha_6-\pi|$ are the dihedral angles of an ideal or a hyperideal tetrahedron if and only if $(\alpha_1,\dots,\alpha_6)$ is of the hyperideal type.

As a consequence of Lemma \ref{factorial} we have

\begin{proposition}\label{6jqd} The quantum $6j$-symbol at the root of unity $q=e^{\frac{2\pi\sqrt{-1}}{r}}$ can be computed as 
$$\bigg|
\begin{matrix}
        m_1 & m_2 & m_3 \\
        m_4 & m_5 & m_6 
      \end{matrix} \bigg|=\frac{\{1\}}{2}\sum_{k=\max\{T_1,T_2,T_3,T_4\}}^{\min\{Q_1,Q_2,Q_3,r-2\}}e^{\frac{r}{4\pi\sqrt{-1}}U_r\big(\frac{2\pi m_1}{r},\dots,\frac{2\pi m_6}{r},\frac{2\pi k}{r}\big)},$$
 where $U_r$ is defined as follows. If $(m_1,\dots,m_6)$ is of the hyperideal type, then
\begin{equation}\label{term}
\begin{split}
U_r(\alpha_1,\dots,\alpha_6,\xi)=&\pi^2-\Big(\frac{2\pi}{r}\Big)^2+\frac{1}{2}\sum_{i=1}^4\sum_{j=1}^3(\eta_j-\tau_i)^2-\frac{1}{2}\sum_{i=1}^4\Big(\tau_i+\frac{2\pi}{r}-\pi\Big)^2\\
&+\Big(\xi+\frac{2\pi}{r}-\pi\Big)^2-\sum_{i=1}^4(\xi-\tau_i)^2-\sum_{j=1}^3(\eta_j-\xi)^2\\
&-2\varphi_r\Big(\frac{\pi}{r}\Big)-\frac{1}{2}\sum_{i=1}^4\sum_{j=1}^3\varphi_r\Big(\eta_j-\tau_i+\frac{\pi}{r}\Big)+\frac{1}{2}\sum_{i=1}^4\varphi_r\Big(\tau_i-\pi+\frac{3\pi}{r}\Big)\\
&-\varphi_r\Big(\xi-\pi+\frac{3\pi}{r}\Big)+\sum_{i=1}^4\varphi_r\Big(\xi-\tau_i+\frac{\pi}{r}\Big)+\sum_{j=1}^3\varphi_r\Big(\eta_j-\xi+\frac{\pi}{r}\Big),\\
\end{split}
\end{equation}
where $\tau_1=\frac{\alpha_1+\alpha_2+\alpha_3}{2},$ $\tau_2=\frac{\alpha_1+\alpha_5+\alpha_6}{2},$ $\tau_3=\frac{\alpha_2+\alpha_4+\alpha_6}{2}$ and $\tau_4=\frac{\alpha_3+\alpha_4+\alpha_5}{2},$ $\eta_1=\frac{\alpha_1+\alpha_2+\alpha_4+\alpha_5}{2},$ $\eta_2=\frac{\alpha_1+\alpha_3+\alpha_4+\alpha_6}{2}$ and $\eta_3=\frac{\alpha_2+\alpha_3+\alpha_5+\alpha_6}{2}.$
If $(m_1,\dots,m_6)$ is not of the hyperideal type, then $U_r$ will be changed according to Lemma \ref{factorial}.
\end{proposition}

As a consequence, we have

\begin{proposition}\label{computation}
$$\mathrm{RT}_r(M,L,(\mathbf n_I,\mathbf m_J))=\kappa_r\sum_{\mathbf m_I,\mathbf k}\Big(\sum_{\epsilon_I}g_r^{\epsilon_I}(\mathbf m_I,\mathbf k)\Big),$$
where 
$$\kappa_r=\frac{2^{|I|-2c}}{\{1\}^{|I|-c}}\bigg(\frac{\sin\frac{2\pi}{r}}{\sqrt{r}}\bigg)^{|I|-c}e^{\big(-\sigma(L'\cup L_I)(-\frac{3}{r}-\frac{r+1}{4})-\frac{r}{4}(\sum_{i\in I}q_i+\sum_{i=1}^np_i+2|I|)\big)\sqrt{-1}\pi},$$
$\epsilon_I=(\epsilon_i)_{i\in I}\in\{1,-1\}^{|I|}$ runs over all multi-signs, $\mathbf m_I=(m_i)_{i\in I}$ runs over all multi-even integers in $\{0,2,\dots,r-3\}$ so that for each $s\in\{1,\dots,c\}$ the triples $(m_{s_1},m_{s_2},m_{s_3}),$  $(m_{s_1},m_{s_5},m_{s_6}),$  $(m_{s_2},m_{s_4},m_{s_6})$ and  $(m_{s_3},m_{s_4},m_{s_5})$ are $r$-admissible, and $\mathbf k=(k_1,\dots,k_c)$ runs over all multi-integers with each $k_s$ lying in between $\max\{T_{s_i}\}$ and $\min\{Q_{s_j},r-2\},$ with
$$g_r^{\epsilon_I}(\mathbf m_I,\mathbf k)=e^{\frac{2\pi\sqrt{-1}}{r}\Big(\sum_{i\in I}q_in_i+\sum_{i=1}^n\big(p_i+\frac{\iota_i}{2}\big)m_i+\sum_{i\in I}\epsilon_i(m_i+n_i+1)\Big)+\frac{r}{4\pi\sqrt{-1}}\mathcal W_r^{\epsilon_I}(\frac{2\pi \mathbf m_I}{r},\frac{2\pi\mathbf k}{r})}$$
where $\frac{2\pi \mathbf m_I}{r}=\Big(\frac{2\pi m_i}{r}\Big)_{i\in I},$ $\frac{2\pi \mathbf k}{r}=\Big(\frac{2\pi k_1}{r},\dots,\frac{2\pi k_c}{r}\Big),$ and
\begin{equation*}
\begin{split}
\mathcal W_r^{\epsilon_I}(\boldsymbol{\alpha}_I,\boldsymbol{\xi})=&-\sum_{i\in I}q_i(\beta_i-\pi)^2-\sum_{j\in J}p_j(\alpha_j-\pi)^2\\
&-\sum_{i\in I}p_i(\alpha_i-\pi)^2-\sum_{i\in I}2\epsilon_i(\alpha_i-\pi)(\beta_i-\pi)\\
&-\sum_{i=1}^n\frac{\iota_i}{2}(\alpha_i-\pi)^2+\sum_{s=1}^c U_r(\alpha_{s_1},\dots,\alpha_{s_6},\xi_s)+\Big(\sum_{i=1}^n\frac{\iota_i}{2}\Big)\pi^2
\end{split}
\end{equation*}
with
$\beta_i=\frac{2\pi n_i}{r}$ for $i\in I$ and $\alpha_j=\frac{2\pi m_j}{r}$ for $j\in J.$
\end{proposition}

\begin{proof} For even integers $m$ and $n,$ we have
$$q^{\frac{n(n+2)}{2}}=\Big(e^{\frac{\pi\sqrt{-1}}{4}}\Big)^{-r}q^{\frac{1}{2}\big(n-\frac{r}{2}\big)^2+n},$$
$$(-1)^{\frac{n}{2}}q^{\frac{n(n+2)}{4}}=\Big(e^{\frac{\pi\sqrt{-1}}{8}}\Big)^{-r}q^{\frac{1}{4}\big(n-\frac{r}{2}\big)^2+\frac{n}{2}}=e^{\frac{r}{4\pi\sqrt{-1}}\big(\frac{\pi^2}{2}\big)}q^{\frac{1}{4}\big(n-\frac{r}{2}\big)^2+\frac{n}{2}},$$
$$q^{(m+1)(n+1)}=\Big(e^{\frac{\pi\sqrt{-1}}{2}}\Big)^{-r}q^{\big(m-\frac{r}{2}\big)\big(n-\frac{r}{2}\big)+m+n+1},$$
and
$$q^{-(m+1)(n+1)}=-\Big(e^{\frac{\pi\sqrt{-1}}{2}}\Big)^{-r}q^{-\big(m-\frac{r}{2}\big)\big(n-\frac{r}{2}\big)-m-n-1}.$$
As a consequence, we have
\begin{equation*}
\begin{split}
[(m_i+1)(n_i+1)]=&\frac{1}{q-q^{-1}}\Big(q^{(m_i+1)(n_i+1)}-q^{-(m_i+1)(n_i+1)}\bigg)\\
=&\frac{(e^{\frac{\pi\sqrt{-1}}{2}})^{-r}}{q-q^{-1}}\bigg(q^{\big((m_i-\frac{r}{2})(n_i-\frac{r}{2})+m_i+n_i+1\big)}+q^{-\big((m_i-\frac{r}{2})(n_i-\frac{r}{2})+m_i+n_i+1\big)}\bigg)\\
=&\frac{e^{-\frac{r\pi\sqrt{-1}}{2}}}{\{1\}}\sum_{\epsilon_i\in\{-1,1\}}q^{\epsilon_i\big((m_i-\frac{r}{2})(n_i-\frac{r}{2})+m_i+n_i+1\big)}\\
=&\frac{e^{-\frac{r\pi\sqrt{-1}}{2}}}{\{1\}}\sum_{\epsilon_i\in\{-1,1\}}e^{\epsilon_i\frac{2\sqrt{-1}\pi(a_i+b_i+1)}{r}+\frac{r}{4\pi\sqrt{-1}}\Big(-2\epsilon_i\big(\frac{2\pi a_i}{r}-\pi\big)\big(\frac{2\pi b_i}{r}-\pi\big)\Big)},\\
\end{split}
\end{equation*}
and hence
\begin{equation*}
\begin{split}\prod_{i\in I}&[(m_i+1)(n_i+1)]\\
=&\frac{e^{-\frac{r\pi\sqrt{-1}}{2}|I|}}{\{1\}^{|I|}}\sum_{\epsilon_I\in\{-1,1\}^{|I|}}e^{\sum_{i\in I}\epsilon_i\frac{2\sqrt{-1}\pi(m_i+n_i+1)}{r}+\frac{r}{4\pi\sqrt{-1}}\sum_{i\in I}\Big(-2\epsilon_i\big(\frac{2\pi m_i}{r}-\pi\big)\big(\frac{2\pi n_i}{r}-\pi\big)\Big)}\\
=&\frac{e^{-\frac{r\pi\sqrt{-1}}{2}|I|}}{\{1\}^{|I|}}\sum_{\epsilon_I\in\{-1,1\}^{|I|}}e^{\sum_{i\in I}\epsilon_i\sqrt{-1}(\alpha_i+\beta_i+\frac{2\pi}{r})+\frac{r}{4\pi\sqrt{-1}}\sum_{i\in I}\big(-2\epsilon_i(\alpha_i-\pi)(\beta_i-\pi)\big)}.\\
\end{split}
\end{equation*}

Then the result follows from (\ref{rt}) and Proposition \ref{6jqd}.
\end{proof}

We notice that the summation in Proposition \ref{computation} is finite, and to use the Poisson Summation Formula, we need an infinite sum over integral points. To this end, we consider the following regions and a bump function over them. 

Let $\beta_i=\frac{2\pi n_i}{r}$ for $i\in I,$ $\alpha_i=\frac{2\pi m_i}{r}$ for $i=1,\dots,n,$ $\xi_s=\frac{2\pi k_s}{r}$ for $s=1,\dots,c,$ $\tau_{s_i}=\frac{2\pi T_{s_i}}{r}$ for $i=1,\dots,4,$ and $\eta_{s_j}=\frac{2\pi Q_{s_j}}{r}$ for $j=1,2,3.$ For a fixed $(\alpha_j)_{j\in J},$ let
$$\mathrm {D_A}=\Big\{(\boldsymbol{\alpha}_I,\boldsymbol{\xi})\in\mathbb R^{|I|+c}\ \Big|\ (\alpha_{s_1},\dots,\alpha_{s_6}) \text{ is admissible, } \max\{\tau_{s_i}\}\leqslant \xi_s\leqslant \min\{\eta_{s_j}, 2\pi\}, s=1,\dots,c\Big\}.$$
and let
$$\mathrm {D_H}=\Big\{(\boldsymbol{\alpha}_I,\boldsymbol{\xi})\in\mathrm {D_A} \ \Big|\ (\alpha_{s_1},\dots,\alpha_{s_6}) \text{ is of the hyperideal type, } s=1,\dots, c \Big\}.$$
For a sufficiently small $\delta >0,$ let 
$$\mathrm {D_H^\delta}=\Big\{(\boldsymbol{\alpha}_I,\boldsymbol{\xi})\in\mathrm {D_H}\ \Big|\ d((\boldsymbol{\alpha}_I,\boldsymbol{\xi}), \partial\mathrm {D_H})>\delta \Big\},$$
where $d$ is the Euclidean distance on $\mathbb R^n.$
Let $\psi:\mathbb R^{|I|+c}\to\mathbb R$ be the $C^{\infty}$-smooth bump function supported on $(\mathrm{D_H}, \mathrm{D_H^\delta}),$ ie,  \begin{equation*}
\left \{\begin{array}{rl}
\psi(\boldsymbol{\alpha}_I,\boldsymbol{\xi})=1, & (\boldsymbol{\alpha}_I,\boldsymbol{\xi})\in \overline{\mathrm{D_H^\delta}}\\
0<\psi(\boldsymbol{\alpha}_I,\boldsymbol{\xi})<1, &  (\boldsymbol{\alpha}_I,\boldsymbol{\xi})\in \mathrm{D_H}\setminus \overline{\mathrm{D_H^\delta}}\\
\psi(\boldsymbol{\alpha}_I,\boldsymbol{\xi})=0, & (\boldsymbol{\alpha}_I,\boldsymbol{\xi})\notin \mathrm{D_H},\\
\end{array}\right.
\end{equation*}
and let 
$$f^{\epsilon_I}_r(\mathbf m_I,\mathbf k)=\psi\Big(\frac{2\pi \mathbf m_I}{r},\frac{2\pi \mathbf k}{r}\Big)g^{\epsilon_I}_r(\mathbf m_I,\mathbf k).$$

In Proposition \ref{computation}, $\mathbf m_I$ runs over  multi-even integers. On the other hand, to use the Poisson Summation Formula, we need a sum over all multi-integers. For this purpose, we for each $i\in I$ let $m_i=2m_i'$ and let $\mathbf m_I'=(m_i')_{i\in I}.$ Then by Proposition \ref{computation}, 
$$\mathrm{RT}_r(M,L,(\mathbf n_I,\mathbf m_J))=\kappa_r\sum_{(\mathbf m_I',\mathbf k)\in\mathbb Z^{|I|+c}}\Big(\sum_{\epsilon_I\in\{1,-1\}^{|I|}} f_r^{\epsilon_I}\big(2\mathbf m_I',\mathbf k\big)\Big)+\text{error term}.$$
Let $$f_r=\sum_{\epsilon_I\in\{1,-1\}^{|I|}} f_r^{\epsilon_I}.$$ Then
$$\mathrm{RT}_r(M,L,(\mathbf n_I,\mathbf m_J))=\kappa_r\sum_{(\mathbf m_I',\mathbf k)\in\mathbb Z^{|I|+c}}f_r\big(2\mathbf m_I',\mathbf k\big)+\text{error term}.$$

Since $f_r$ is $C^{\infty}$-smooth and equals zero out of $\mathrm{D_H},$ it is in the Schwartz space on $\mathbb R^{|I|+c}.$ Then by the Poisson Summation Formula (see e.g. \cite[Theorem 3.1]{SS}),
$$\sum_{(\mathbf m_I',\mathbf k)\in\mathbb Z^{|I|+c}}f_r\big(2\mathbf m_I',\mathbf k\big)=\sum_{(\mathbf a_I,\mathbf b)\in\mathbb Z^{|I|+c}}\widehat {f_r}(\mathbf a_I,\mathbf b),$$
where $\mathbf a_I=(a_i)_{i\in I}\in \mathbb Z^I,$ $\mathbf b=(b_1,\dots,b_c)\in\mathbb Z^c$ and  $\widehat f_r(\mathbf a_I,\mathbf b)$ is the $(\mathbf a_I,\mathbf b)$-th Fourier coefficient of $f_r$ defined by
\begin{equation*}
\begin{split}
\widehat {f_r}(\mathbf a_I,\mathbf b)=\int_{\mathbb R^{|I|+c}}&f_r\big(2\mathbf m_I',\mathbf k\big)e^{\sum_{i\in I}2\pi \sqrt{-1}a_im_i'+\sum_{s=1}^c2\pi \sqrt{-1}b_sk_s}d\mathbf m'_Id\mathbf k,
\end{split}
\end{equation*}
where $d\mathbf m'_Id\mathbf k=\prod_{i\in I}dm'_i\prod_{s=1}^cdk_s.$

By a change of variable, and by changing $2m_i'$ back to $m_i,$ the Fourier coefficients can be computed as
\begin{proposition}\label{4.2}
$$\widehat{f_r}(\mathbf a_I,\mathbf b)=\sum_{\epsilon_I\in\{1,-1\}^{|I|}} \widehat{f^{\epsilon_I}_r}(\mathbf a_I,\mathbf b)$$
with
\begin{equation*}
\begin{split}
\widehat{f^{\epsilon_I}_r}(\mathbf a_I,\mathbf b)=\frac{r^{|I|+c}}{2^{2|I|+c}\cdot\pi^{|I|+c}}&\int_{\mathrm{D_H}}\psi(\boldsymbol{\alpha}_I,\boldsymbol{\xi})e^{\big(\sum_{i\in I}q_i\beta_i+\sum_{i=1}^n\big(p_i+\frac{\iota_i}{2}\big)\alpha_i+\sum_{i\in I}\epsilon_i(\alpha_i+\beta_i+\frac{2\pi}{r})\big)\sqrt{-1}}\\ 
&\cdot e^{\frac{r}{4\pi \sqrt{-1}}\big(\mathcal W_r^{\epsilon_I}(\boldsymbol{\alpha}_I,\boldsymbol{\xi})-\sum_{i\in I}2\pi a_i\alpha_i-\sum_{s=1}^c4\pi b_s\xi_s\big)}d\boldsymbol{\alpha}_Id\boldsymbol{\xi},
\end{split}
\end{equation*}
where $d\boldsymbol{\alpha}_Id\boldsymbol{\xi}=\prod_{i\in I}d\alpha_i\prod_{s=1}^cd\xi_s,$ and $\mathcal W_r^{\epsilon_I}(\boldsymbol{\alpha}_I,\boldsymbol{\xi})$ is as defined in Proposition \ref{computation}.
\end{proposition}

\begin{proposition}\label{Poisson}
$$\mathrm{RT}_r(M,L,(\mathbf n_I,\mathbf m_J))=\kappa_r\sum_{(\mathbf a_I,\mathbf b)\in\mathbb Z^{|I|+c}}\widehat{ f_r}(\mathbf a_I,\mathbf b)+\text{error term}.$$
\end{proposition}

We will estimate the leading Fourier coefficients, the non-leading Fourier coefficients and the error term respectively in Sections \ref{leading}, \ref{ot} and \ref{ee}, and prove Theorem \ref{main2} in Section \ref{pf}.


\section{Relationship with the Neumann-Zagier potential function}

The goal of this section is to show the relationship between $\mathcal W_r^{\epsilon_I}$ and the Neumann-Zagier potential function\,\cite{NZ} of the fundamental shadow link complement $M_c\setminus L_{\text{FSL}}.$  To this end, we need to first look at the function $U$ coming from a single $6j$-symbol, and to recall its relationship with the volume of a truncated hyperideal tetrahedron.

By Lemma \ref{converge}, $\mathcal W_r^{\epsilon_I}$ is approximated by the holomorphic function $\mathcal W^{\epsilon_I}$ defined below, which will play an important role later. (The approximation will be specified in the proof of Proposition \ref{critical}.) The function $\mathcal W^{\epsilon_I}$ is defined by
\begin{equation*}
\begin{split}
\mathcal W^{\epsilon_I}(\boldsymbol{\alpha}_I,\boldsymbol{\xi})=&-\sum_{i\in I}q_i(\beta_i-\pi)^2-\sum_{j\in J}p_j(\alpha_j-\pi)^2\\
&-\sum_{i\in I}p_i(\alpha_i-\pi)^2-\sum_{i\in I}2\epsilon_i(\alpha_i-\pi)(\beta_i-\pi)\\
&-\sum_{i=1}^n\frac{\iota_i}{2}(\alpha_i-\pi)^2+\sum_{s=1}^c U(\alpha_{s_1},\dots,\alpha_{s_6},\xi_s)+\Big(\sum_{i=1}^n\frac{\iota_i}{2}\Big)\pi^2
\end{split}
\end{equation*}
with $U$ defined by
\begin{equation}\label{term}
\begin{split}
U(\alpha_1,\dots,\alpha_6,\xi)=&\pi^2+\frac{1}{2}\sum_{i=1}^4\sum_{j=1}^3(\eta_j-\tau_i)^2-\frac{1}{2}\sum_{i=1}^4(\tau_i-\pi)^2\\
&+(\xi-\pi)^2-\sum_{i=1}^4(\xi-\tau_i)^2-\sum_{j=1}^3(\eta_j-\xi)^2\\
&-2\mathit{Li}_2(1)-\frac{1}{2}\sum_{i=1}^4\sum_{j=1}^3\mathit{Li}_2\big(e^{2i(\eta_j-\tau_i)}\big)+\frac{1}{2}\sum_{i=1}^4\mathit{Li}_2\big(e^{2i(\tau_i-\pi)}\big)\\
&-\mathit{Li}_2\big(e^{2i(\xi-\pi)}\big)+\sum_{i=1}^4\mathit{Li}_2\big(e^{2i(\xi-\tau_i)}\big)+\sum_{j=1}^3\mathit{Li}_2\big(e^{2i(\eta_j-\xi)}\big).\\
\end{split}
\end{equation}
We note that $U$ defines a holomorphic function on the region 
$$\mathrm{B_{H,\mathbb C}}=\big\{(\boldsymbol\alpha,\xi)\in\mathbb C^7\ |\ \mathrm{Re}(\boldsymbol\alpha)\text{ is of the hyperideal type, }\max\{\mathrm{Re}(\tau_i)\}\leqslant \mathrm{Re}(\xi)\leqslant \min\{\mathrm{Re}(\eta_j), 2\pi\}\big\}$$
 where $\boldsymbol\alpha=(\alpha_1,\dots,\alpha_6)$ and $\mathrm{Re}(\boldsymbol\alpha)=(\mathrm{Re}(\alpha_1),\dots,\mathrm{Re}(\alpha_6));$ and $\mathcal W^{\epsilon_I}$ is continuous on 
$$\mathrm{D_{H,\mathbb C}}=\big\{(\boldsymbol{\alpha}_I,\boldsymbol{\xi})\in\mathbb C^{|I|+c}\ \big|\ (\mathrm{Re}(\boldsymbol{\alpha}_I),\mathrm{Re}(\boldsymbol\xi))\in \mathrm{D_{H}}\big\}$$ and for any $\delta>0$ is analytic on 
$$\mathrm{D^\delta_{H,\mathbb C}}=\big\{(\boldsymbol{\alpha}_I,\boldsymbol{\xi})\in\mathbb C^{|I|+c}\ \big|\ (\mathrm{Re}(\boldsymbol{\alpha}_I),\mathrm{Re}(\boldsymbol\xi))\in \mathrm{D^\delta_{H}}\big\},$$ 
where $\mathrm{Re}(\boldsymbol{\alpha}_I)=(\mathrm{Re}(\alpha_i))_{i\in I}$ and $\mathrm{Re}(\boldsymbol\xi)=(\mathrm{Re}(\xi_1),\dots, \mathrm{Re}(\xi_c)).$

Let 
$$\mathrm{B_H}=\mathrm{B_{H,\mathbb C}}\cap \mathbb R^7.$$
Then by (\ref{dilogLob}), for $(\alpha_1,\dots,\alpha_6,\xi)\in\mathrm{B_H},$
\begin{equation}\label{UV} 
U(\alpha_1,\dots,\alpha_6,\xi)=2\pi^2+2\sqrt{-1} V(\alpha_1,\dots,\alpha_6,\xi)
\end{equation}
for $V:\mathrm{B_H}\to\mathbb R$ defined by 
\begin{equation}\label{V}
\begin{split}
V(\alpha_1,\dots,\alpha_6,\xi)=\,&\delta(\alpha_1,\alpha_2,\alpha_3)+\delta(\alpha_1,\alpha_5,\alpha_6)+\delta(\alpha_2,\alpha_4,\alpha_6)+\delta(\alpha_3,\alpha_4,\alpha_5)\\
&-\Lambda(\xi)+\sum_{i=1}^4\Lambda(\xi-\tau_i)+\sum_{j=1}^3\Lambda(\eta_j-\xi),
\end{split}
\end{equation}
where $\delta$ is defined by
$$\delta(x,y,z)=-\frac{1}{2}\Lambda\Big(\frac{x+y-z}{2}\Big)-\frac{1}{2}\Lambda\Big(\frac{y+z-x}{2}\Big)-\frac{1}{2}\Lambda\Big(\frac{z+x-y}{2}\Big)+\frac{1}{2}\Lambda\Big(\frac{x+y+z}{2}\Big).$$

A result of Costantino\,\cite{C1} shows that for each $\boldsymbol\alpha=(\alpha_1,\dots,\alpha_6)$ of the hyperideal type, there exists a unique $\xi(\boldsymbol\alpha)$ so that $(\boldsymbol\alpha,\xi(\boldsymbol\alpha))\in\mathrm{B_H}$ and 
\begin{equation}\label{criticalequation}
\frac{\partial V(\boldsymbol\alpha,\xi)}{\partial \xi}\Big|_{\xi=\xi(\boldsymbol\alpha)}=0.
\end{equation}
Indeed, he proves that for each $\boldsymbol\alpha,$ $V$ is strictly concave down in $\xi$ with derivatives $\pm\infty$ at the boundary points of the interval of $\xi,$ hence there is a unique critical point $\xi(\boldsymbol\alpha)$ at which $V$ achieves the absolute maximum. Moreover, by using the Murakami-Yano formula\,\cite{MYa,U} he shows that 
\begin{equation}\label{VV}
V(\boldsymbol\alpha,\xi(\boldsymbol\alpha))=\mathrm{Vol}(\Delta_{|\boldsymbol\alpha-\boldsymbol\pi|}),
\end{equation}
the volume of the ideal or the truncated hyperideal tetrahedron with dihedral angles $|\alpha_1-\pi|,\dots, |\alpha_6-\pi|.$ Let $l_1,\dots, l_6$ be the  lengths of the edges of $\Delta_{|\boldsymbol\alpha-\boldsymbol\pi|},$ and let $u_i=2\sqrt{-1}|\alpha_i-\pi|$ for $i\in\{1,\dots,6\}.$ 
Then by the Schl\"afi formula, we have
\begin{equation}\label{sch}
\frac{\partial U(\boldsymbol\alpha,\xi(\boldsymbol\alpha))}{\partial u_i}=-\frac{l_i}{2}.
\end{equation}

For $\boldsymbol\alpha=(\alpha_1,\dots,\alpha_6)\in\mathbb C^6$ such that $(\mathrm{Re}(\alpha_1),\dots,\mathrm{Re}(\alpha_6))$ is of the hyperideal type, we let $\xi(\boldsymbol\alpha)$ be such that
\begin{equation}\label{xia}
\frac{\partial U(\boldsymbol\alpha,\xi)}{\partial \xi}\Big|_{\xi=\xi(\boldsymbol\alpha)}=0.
\end{equation}
Following the idea of \cite{MY}, see also \cite{BY}, it is proved that $e^{-2\sqrt{-1}\xi(\boldsymbol\alpha)}$ satisfies a concrete quadratic equation. Therefore, for each such $\boldsymbol\alpha,$ there is at most one $\xi(\boldsymbol\alpha)$ such that $(\boldsymbol\alpha,\xi(\boldsymbol\alpha))\in \mathrm{B_{H,\mathbb C}}.$ At this point, we do not know whether $(\boldsymbol\alpha,\xi(\boldsymbol\alpha))\in \mathrm{B_{H,\mathbb C}}$ for all such $\boldsymbol\alpha,$ but in the next section we will show that it is the case if all $\mathrm{Re}(\alpha_1),\dots,\mathrm{Re}(\alpha_6)$ are sufficiently close to $\pi.$ 
\\

For $s\in\{1,\dots,c\},$ let $\boldsymbol{\alpha}_s=(\alpha_{s_1},\dots,\alpha_{s_6}).$ We define the following function
$$\mathcal U(\boldsymbol{\alpha}_I,\boldsymbol{\alpha}_J)=-\sum_{i=1}^n\frac{\iota_i}{2}(\alpha_i-\pi)^2+\sum_{s=1}^cU(\boldsymbol{\alpha}_s,\xi(\boldsymbol{\alpha}_s))+\Big(\sum_i^n\frac{\iota_i}{2}\Big)\pi^2
$$
for all $\boldsymbol{\alpha}_I$ such that $(\boldsymbol{\alpha}_s,\xi(\boldsymbol{\alpha}_s))\in \mathrm{B_{H,\mathbb C}}$ for all $s\in\{1,\dots,c\}.$

The next proposition shows that with an appropriate choice of the meridians and longitudes, the value of  $\mathcal U$ coincides with the value of the Neumann-Zagier potential function\,\cite{NZ} defined on a neighborhood of the complete structure in the deformation space of $M_c\setminus L_{\text{FSL}}.$

\begin{proposition}\label{NeuZ} For each component $T_i$ of the boundary of $M_c\setminus L_{\text{FSL}},$ choose the basis $(u_i,v_i)$ of $\pi_1(T_i)$ as in (\ref{m}) and (\ref{l}), and let $\Phi$ be the Neumann-Zagier potential function characterized by
\begin{equation}\label{char2}
\left \{\begin{array}{l}
\frac{\partial \Phi(\mathrm{H}(u_1),\dots,\mathrm{H}(u_n))}{\partial \mathrm{H}(u_i)}=\frac{\mathrm H(v_i)}{2},\\
\\
\Phi(0,\dots,0)=\sqrt{-1}\bigg(\mathrm{Vol}(M_c\setminus L_{\text{FSL}})+\sqrt{-1}\mathrm{CS}(M_c\setminus L_{\text{FSL}})\bigg)\quad\quad\text{mod }\pi^2\mathbb Z,
\end{array}\right.
\end{equation} 
where $ M_c\setminus L_{\text{FSL}}$ is with the complete hyperbolic metric. If $\mathrm H(u_i)=\pm 2\sqrt{-1}(\alpha_i-\pi)$ for each $i\in \{1,\dots,n\},$ then
$$\mathcal U(\boldsymbol{\alpha}_I,\boldsymbol{\alpha}_J)=2c\pi^2+\Phi(\mathrm H(u_1),\dots,\mathrm H(u_n)).$$
\end{proposition}

\begin{proof} We assume that  $\mathrm H(u_i)= 2\sqrt{-1}(\alpha_i-\pi),$ and the case $\mathrm H(u_i)= -2\sqrt{-1}(\alpha_i-\pi)$ follows from the fact that $\Phi$ is even in its variables.

 From the construction of $L_{\text{FSL}}$ in Section \ref{fsl},  the component $L_i=\cup e_{s_j},$ where $e_{s_j}$ is the $j$-th edge of the building block $\Delta_s$ for $s$ coming from a subset $S$ of $\{1,\dots, c\}$ and $j$ coming from a subset of $\{1,\dots, 6\}$ depending on $s.$  Let $l_i$ be the lengths of $L_i$ and let $l_{s_j}$ be the length of $e_{s_j}$ for each $s_j.$ Then by (\ref{m}), (\ref{l}) and  (\ref{sch}), we have
\begin{equation}\label{phi1}
\begin{split}
\frac{\partial \mathcal U(\boldsymbol{\alpha}_I,\boldsymbol{\alpha}_J)}{\partial \mathrm H(u_i)}=&\sum_{s\in S}\frac{\partial U(\boldsymbol{\alpha}_s,\xi(\boldsymbol{\alpha}_s))}{\partial \mathrm H(u_i)}+\frac{\iota_i\mathrm H(u_i)}{4}\\
=&\sum_{s_j}-\frac{l_{s_j}}{2}+\frac{\iota_i\sqrt{-1}\theta_i}{4}=-\frac{l_i}{2}+\frac{\iota_i\sqrt{-1}\theta_i}{4}=\frac{\mathrm H(v_i)}{2}.
\end{split}
\end{equation}
This means $\frac{\partial \mathcal U}{\partial \mathrm H(u_i)}$ and $\frac{\partial \Phi}{\partial \mathrm H(u_i)}$ coincide for purely imaginary variables for each $i\in I.$ Then by Lemma \ref{MCV} below, $\frac{\partial \mathcal U}{\partial \mathrm H(u_i)}$ and $\frac{\partial \Phi}{\partial \mathrm H(u_i)}$ coincide, verifying the first equality of (\ref{char2}).

For the second equality of (\ref{char2}), we have 
$$\mathcal U(\pi,\dots,\pi)=2c\pi^2+2cv_8\sqrt{-1}+\Big(\sum_{i=1}^n \frac{\iota_i}{2}\Big) \pi^2,$$
where $v_8$ is the volume of the regular ideal octahedron. Then by (\ref{VolFSL}) and (\ref{CSFSL}),  we have
$$\mathcal U(\pi,\dots,\pi)=2c\pi^2+\Phi(0,\dots,0)\quad\quad\text{mod }\pi^2\mathbb Z.$$
\end{proof}

\begin{lemma}\label{MCV} Suppose $D$ is a domain of $\mathbb C^n$ and $F_1$ and $F_2$ are two holomorphic functions on $D.$  If $F_1$ and $F_2$ coincide on $D\cap (\sqrt{-1}\mathbb R)^n,$ then $F_1$ and $F_2$ coincide on $D.$
\end{lemma}

\begin{proof}  We use induction on $n.$ If $n=1,$ then the result follows from the Identity Theorem of a single variable analytic function. Now suppose the result is true for $n\leqslant k.$ For each fixed $(z_2,\dots,z_k)\in (\sqrt{-1}\mathbb R)^{k-1},$ by the assumption of the lemma, we have 
$F_1(z_1,z_2,\dots,z_k)=F_2(z_1,z_2,\dots,z_k)$ for any purely imaginary $z_1.$ Then by the single variable case $F_1(z_1,z_2,\dots,z_k)=F_2(z_1,z_2,\dots,z_k)$ for any complex $z_1.$ This equality can also be understood as for any fixed complex $z_1,$ $F_1(z_1,z_2,\dots,z_k)=F_2(z_1,z_2,\dots,z_k)$  for all purely imaginary $(z_2,\dots,z_k).$ Then by the induction hypothesis, we have $F_1(z_1,z_2,\dots,z_k)=F_2(z_1,z_2,\dots,z_k)$ for all $(z_2,\dots,z_k).$
\end{proof}


\section{Asymptotics}\label{Asy}

The goal of this section is to prove Theorem \ref{main2}. The main tool we use is Proposition \ref{saddle}, which is a generalization of the standard Saddle Point Approximation\,\cite{O}. For the readers' convenience, we include a proof of Proposition \ref{saddle} in Appendix A.

\begin{proposition}\label{saddle}
Let $D_{\mathbf z}$ be a region in $\mathbb C^n$ and let $D_{\mathbf a}$ be a region in $\mathbb R^k.$ Let $f(\mathbf z,\mathbf a)$ and $g(\mathbf z,\mathbf a)$ be complex valued functions on $D_{\mathbf z}\times D_{\mathbf a}$  which are holomorphic in $\mathbf z$ and smooth in $\mathbf a.$ For each positive integer $r,$ let $f_r(\mathbf z,\mathbf a)$ be a complex valued function on $D_{\mathbf z}\times D_{\mathbf a}$ holomorphic in $\mathbf z$ and smooth in $\mathbf a.$
For a fixed $\mathbf a\in D_{\mathbf a},$ let $f^{\mathbf a},$ $g^{\mathbf a}$ and $f_r^{\mathbf a}$ be the holomorphic functions  on $D_{\mathbf z}$ defined by
$f^{\mathbf a}(\mathbf z)=f(\mathbf z,\mathbf a),$ $g^{\mathbf a}(\mathbf z)=g(\mathbf z,\mathbf a)$ and $f_r^{\mathbf a}(\mathbf z)=f_r(\mathbf z,\mathbf a).$ Suppose $\{\mathbf a_r\}$ is a convergent sequence in $D_{\mathbf a}$ with $\lim_r\mathbf a_r=\mathbf a_0,$ $f_r^{\mathbf a_r}$ is of the form
$$ f_r^{\mathbf a_r}(\mathbf z) = f^{\mathbf a_r}(\mathbf z) + \frac{\upsilon_r(\mathbf z,\mathbf a_r)}{r^2},$$
$\{S_r\}$ is a sequence of embedded real $n$-dimensional closed disks in $D_{\mathbf z}$ sharing the same boundary and converging to an embedded $n$-dimensional disk $S_0$, and $\mathbf c_r$ is a point on $S_r$ such that $\{\mathbf c_r\}$ is convergent  in $D_{\mathbf z}$ with $\lim_r\mathbf c_r=\mathbf c_0.$ If for each $r$
\begin{enumerate}[(1)]
\item $\mathbf c_r$ is a critical point of $f^{\mathbf a_r}$ in $D_{\mathbf z},$
\item $\mathrm{Re}f^{\mathbf a_r}(\mathbf c_r) > \mathrm{Re}f^{\mathbf a_r}(\mathbf z)$ for all $\mathbf z \in S_r\setminus \{\mathbf c_r\},$
\item the domain $\{\mathbf z\in D_{\mathbf z}\ |\ \mathrm{Re} f^{\mathbf a_r}(\mathbf z) < \mathrm{Re} f^{\mathbf a_r}(\mathbf c_r)\}$ deformation retracts to $S_r\setminus\{\mathbf c_r\},$
\item $|g^{\mathbf a_r}(\mathbf c_r)|$ is bounded from below by a positive constant independent of $r,$
\item $|\upsilon_r(\mathbf z, \mathbf a_r)|$ is bounded from above by a constant independent of $r$ on $D_{\mathbf z},$ and
\item  the Hessian matrix $\mathrm{Hess}(f^{\mathbf a_0})$ of $f^{\mathbf a_0}$ at $\mathbf c_0$ is non-singular,
\end{enumerate}
then
\begin{equation*}
\begin{split}
 \int_{S_r} g^{\mathbf a_r}(\mathbf z) e^{rf_r^{\mathbf a_r}(\mathbf z)} d\mathbf z= \Big(\frac{2\pi}{r}\Big)^{\frac{n}{2}}\frac{g^{\mathbf a_r}(\mathbf c_r)}{\sqrt{-\det\mathrm{Hess}(f^{\mathbf a_r})(\mathbf c_r)}} e^{rf^{\mathbf a_r}(\mathbf c_r)} \Big( 1 + O \Big( \frac{1}{r} \Big) \Big).
 \end{split}
 \end{equation*}
\end{proposition}

In the rest of this paper, we assume that 	$\theta_1,\dots,\theta_n$ are sufficiently close to $0,$ or equivalently, $\{\beta_i\}_{i\in I}$ and $\{\alpha_j\}_{j\in J}$ are sufficiently close to $\pi.$ In the special case that $\beta_i=\alpha_j=\pi$ for all $i\in I$ and $j\in J,$ by solving equation (\ref{criticalequation}) for $(\alpha_1,\dots,\alpha_6)=(\pi,\dots,\pi),$ we have $\xi(\pi,\dots,\pi)=\frac{7\pi}{4}.$ For $\delta>0,$ we denote by $\mathrm{D_{\delta,\mathbb C}}$  the $L^1$ $\delta$-neighborhood  of $\big(\pi,\dots,\pi,\frac{7\pi}{4},\dots, \frac{7\pi}{4}\big)$ in $\mathbb C^{|I|+c},$  that is 
$$\mathrm{D_{\delta,\mathbb C}}=\Big\{(\boldsymbol{\alpha}_I,\boldsymbol{\xi})\in \mathbb C^{|I|+c}\ \Big|\ d_{L^1}\Big((\boldsymbol{\alpha}_I,\boldsymbol{\xi}),\Big(\pi,\dots,\pi,\frac{7\pi}{4},\dots, \frac{7\pi}{4}\Big)\Big)<\delta\Big\},$$
where $d_{L^1}$ is the real $L^1$-norm on $\mathbb C^n$ defined by
$$d_{L^1}(\mathbf x,\mathbf y)=\max_{i\in\{1,\dots,n\}}\{|\mathrm {Re}(x_i)-\mathrm{Re}(y_i)|, |\mathrm {Im}(x_i)-\mathrm{Im}(y_i)| \},$$
where $\mathbf x=(x_1,\dots,x_n)$ and $\mathbf y=(y_1,\dots,y_n).$ We will also consider the region 
$$\mathrm{D_{\delta}}=\mathrm{D_{\delta,\mathbb C}}\cap \mathbb R^{|I|+c}.$$

\subsection{Critical points and critical values of $\mathcal W^{\epsilon_I}$}

Suppose $\{\beta_i\}_{i\in I}$ and $\{\alpha_j\}_{j\in J}$ are sufficiently close to $\pi.$ For $i\in I,$ let $\theta_i=2|\beta_i-\pi|,$ and let $\mu_i=1$ if $\beta_i\geqslant\pi$ and let $\mu_i=-1$ if $\beta_i\leqslant \pi$ so that $\mu_i\theta_i=2(\beta_i-\pi).$  

\begin{proposition}\label{crit} For each $i\in I,$ let $\mathrm H(u_i)$ be the logarithmic holonomy of $u_i$ of the hyperbolic cone manifold $M_{L_{\boldsymbol\theta}}$ and let
\begin{equation}\label{alpha}
\alpha^*_i=\pi+\frac{\epsilon_i\mu_i\sqrt{-1}}{2}\mathrm H(u_i).
\end{equation}
For $s\in\{1,\dots,c\},$ let $\xi^*_s=\xi(\alpha^*_{s_1},\dots,\alpha^*_{s_6})$ be as defined in (\ref{xia}).
 Then $\mathcal W^{\epsilon_I}$ has a critical point 
$$\mathbf z^{\epsilon_I}=\Big(\big(\alpha^*_i\big)_{i\in I}, \big(\xi^*_s\big)_{s=1}^c\Big)$$
in $\mathrm{D_{\delta,\mathbb C}}$ with critical value $$2c\pi^2+\sqrt{-1}\Big(\mathrm{Vol}(M_{L_{\boldsymbol\theta}})+\sqrt{-1}\mathrm{CS}(M_{L_{\boldsymbol\theta}})\Big).$$
\end{proposition}

\begin{proof} For $s\in\{1,\dots,c\},$ let $\boldsymbol{\alpha}_s=(\alpha_{s_1},\dots,\alpha_{s_6})$ and let $\boldsymbol{\alpha}^*_s=(\alpha^*_{s_1},\dots,\alpha^*_{s_6}).$

If all $\theta_i$'s are sufficiently close to $0,$ then the cone metric is sufficiently close to the complete metric. As a consequence, $\mathrm H(u_i)$ is sufficiently close to $0$ for each $i,$ and hence $\alpha_i$ is sufficiently close to $\pi.$ Then by the continuity of $\xi(\boldsymbol{\alpha}_s),$ $\mathbf z^{\epsilon_I}\in \mathrm{D_{\delta,\mathbb C}}$ for sufficiently small $\theta_1,\dots,\theta_n.$

First, for each $s\in\{1,\dots,c\},$ we have 
\begin{equation}\label{c1}
\frac{\partial \mathcal W^{\epsilon_I}}{\partial \xi_s}\Big|_{\mathbf z^{\epsilon_I}}=\frac{\partial U(\boldsymbol{\alpha}^*_s,\xi_s)}{\partial \xi_s}\Big|_{\xi^*_s}=0.
\end{equation}

Then by the Chain Rule, for each $s\in\{1,\dots,c\}$ and $i\in I,$ 
$$\frac{\partial U(\boldsymbol{\alpha}_s,\xi(\boldsymbol{\alpha}_s))}{\partial \alpha_i}\Big|_{\boldsymbol{\alpha}^*_s}=\frac{\partial U(\boldsymbol{\alpha}_s,\xi_s)}{\partial \alpha_i}\Big|_{(\boldsymbol{\alpha}^*_s,\xi^*_s)}+\frac{\partial U(\boldsymbol{\alpha}_s,\xi_s)}{\partial \xi_s}\Big|_{(\boldsymbol{\alpha}^*_s,\xi^*_s)}\cdot\frac{\partial \xi(\boldsymbol{\alpha}_s)}{\partial \alpha_i}\Big|_{\boldsymbol{\alpha}^*_s}=\frac{\partial U(\boldsymbol{\alpha}_s,\xi_s)}{\partial \alpha_i}\Big|_{(\boldsymbol{\alpha}^*_s,\xi^*_s)},$$
hence
$$\frac{\partial \big(\sum_{s=1}^c U(\boldsymbol{\alpha}_s,\xi_s)\big)}{\partial \alpha_i}\Big|_{\mathbf z^{\epsilon_I}}=\frac{\partial \mathcal U}{\partial \alpha_i}\Big|_{\boldsymbol{\alpha}^*_I}=-\epsilon_i\mu_i\sqrt{-1}\mathrm H(v_i),$$
where $\boldsymbol{\alpha}^*_I=(\alpha^*_i)_{i\in I}$ and the last equation comes from (\ref{phi1}) and (\ref{alpha}).
As a consequence, for each $i\in I,$ we have
\begin{equation}\label{c2}
\begin{split}
\frac{\partial \mathcal W^{\epsilon_I}}{\partial \alpha_i}\Big|_{\mathbf z^{\epsilon_I}}=&-2p_i(\alpha^*_i-\pi)-2\epsilon_i(\beta_i-\pi) + \frac{\partial \mathcal U}{\partial \alpha_i}\Big|_{\boldsymbol{\alpha}^*_I}\\
=&-\epsilon_i\mu_i\sqrt{-1}\Big(p_i\mathrm H(u_i)-\sqrt{-1}\theta_i+\mathrm H(v_i)\Big)=0,
\end{split}
\end{equation}
where the last equality comes from the $(p_i,1)$-Dehn filling equation (\ref{DF}) with the cone angle $\theta_i.$ Equations (\ref{c1}) and (\ref{c2}) show that $\mathbf z^{\epsilon_I}$ is a critical point of $\mathcal W^{\epsilon_I}.$

To compute the critical value, by Proposition \ref{NeuZ},  we first have
\begin{equation}\label{CV1}
\mathcal U(\boldsymbol{\alpha}^*_I,\boldsymbol{\alpha}_J)=2c\pi^2+\Phi(\mathrm H(u_1),\dots, \mathrm H(u_n)).
\end{equation}

For each $i\in I,$ let $\gamma_i=-u_i+q_i(p_iu_i+v_i)$ so that it is the curve on the boundary of a tubular neighborhood of $L^*_i$ that is isotopic to $L^*_i$ given by the framing $q_i$ of $L^*_i$ and with the orientation so that $(p_iu_i+v_i)\cdot \gamma_i=1.$ Then we have 
$\theta_i=2\mu_i(\beta_i-\pi),$ $\mathrm H(u_i)=-2\epsilon_i\mu_i\sqrt{-1}(\alpha^*_i-\pi),$ 
\begin{equation*}
\begin{split}
\mathrm H(v_i)=&\sqrt{-1}\theta_i-p_i\mathrm H(u_i)=2\mu_i\sqrt{-1}(\beta_i-\pi)+2p_i\epsilon_i\mu_i\sqrt{-1}(\alpha^*_i-\pi)
\end{split}
\end{equation*}
 and 
 \begin{equation*}
 \begin{split}
 \mathrm H(\gamma_i)=&-\mathrm H(u_i)+q_i(p_i\mathrm H(u_i)+\mathrm H(v_i))=2\epsilon_i\mu_i\sqrt{-1}(\alpha^*_i-\pi)+2q_i\mu_i\sqrt{-1}(\beta_i-\pi).
 \end{split}
 \end{equation*}
As a consequence, we have
\begin{equation}\label{CV2}
\begin{split}
&-\sum_{i\in I}\frac{\mathrm H(u_i)\mathrm H(v_i)}{4}+\sum_{i\in I}\frac{\sqrt{-1}\theta_i\mathrm H(\gamma_i)}{4}\\
=&-\sum_{i\in I} q_i(\beta_i-\pi)^2-\sum_{i\in I} p_i(\alpha^*_i-\pi)^2-\sum_{i\in I} 2\epsilon_i(\alpha^*_i-\pi)(\beta_i-\pi).
\end{split}
\end{equation}

For each $j\in J,$ let $\gamma_j=p_ju_j+v_j$ so that it is the curve on the boundary of a tubular neighborhood of $L_j$ that is isotopic to $L_j$ given by the framing $p_j$ of $L_j$ and with the orientation so that $u_j\cdot \gamma_j=1.$  Then we have 
$\theta_j=2|\alpha_j-\pi|,$ $\mathrm H(u_j)=2\sqrt{-1}|\alpha_j-\pi|$ and $\mathrm H(\gamma_j)=p_j\mathrm H(u_j)+\mathrm H(v_j).$ 
As a consequence, we have
\begin{equation}\label{CV3}
\begin{split}
-\sum_{j\in J}\frac{\mathrm H(u_j)\mathrm H(v_j)}{4}+\sum_{j\in J}\frac{\sqrt{-1}\theta_j\mathrm H(\gamma_j)}{4}=&-\sum_{j\in J}\frac{\mathrm H(u_j)\mathrm H(v_j)}{4}+\sum_{j\in J}\frac{\mathrm H(u_j)\big(p_j\mathrm H(u_j)+\mathrm H(v_j)\big)}{4}\\
=&\sum_{j\in J}\frac{p_j\mathrm H(u_j)^2}{4}=\sum_{j\in J}-p_j(\alpha_j-\pi)^2.
\end{split}
\end{equation}
Putting (\ref{CV1}), (\ref{CV2}), (\ref{CV3}) and (\ref{VCS}) together ,  we have
\begin{equation*}
\begin{split}
\mathcal W^{\epsilon_I}&(\mathbf z^{\epsilon_I})\\
=&\,\mathcal U(\boldsymbol{\alpha}_I,\boldsymbol{\alpha}_J)-\sum_{i\in I}q_i(\beta_i-\pi)^2-\sum_{i\in I}p_i(\alpha^*_i-\pi)^2-\sum_{j\in j}p_j(\alpha_j-\pi)^2-\sum_{i\in I}2\epsilon_i(\alpha^*_i-\pi)(\beta_i-\pi)\\
=&2c\pi^2+\Phi(\mathrm H(u_1),\dots, \mathrm H(u_n))-\sum_{i=1}^n\frac{\mathrm H(u_i)\mathrm H(v_i)}{4}+\sum_{i=1}^n\frac{\sqrt{-1}\theta_i\mathrm H(\gamma_i)}{4}\\
=&2c\pi^2+\sqrt{-1}\Big(\mathrm{Vol}(M_{L_{\boldsymbol\theta}})+\sqrt{-1}\mathrm{CS}(M_{L_{\boldsymbol\theta}})\Big).
\end{split}
\end{equation*}
\end{proof}


\subsection{ Convexity of $\mathcal W^{\epsilon_I}$}

\begin{proposition}\label{convexity} 
There exists a $\delta_0>0$ such that if all $\{\alpha_j\}_{j\in J}$ are in $(\pi-\delta_0,\pi+\delta_0),$ then for any $\epsilon_I,$ $\mathrm{Im}\mathcal W^{\epsilon_I}(\boldsymbol{\alpha}_I,\boldsymbol{\xi})$ is strictly concave down in $\{\mathrm{Re}(\alpha_i)\}_{i\in I}$ and $\{\mathrm{Re}(\xi_s)\}_{s=1}^c,$ and is strictly concave up in $\{\mathrm{Im}(\alpha_i)\}_{i\in I}$ and $\{\mathrm{Im}(\xi_s)\}_{s=1}^c$  on $\mathrm{D_{\delta_0,\mathbb C}}.$
\end{proposition}

\begin{proof} We first consider the special case $\{\alpha_i\}_{i\in I}$ and $\{\xi_s\}_{s=1}^c$ are real.  In this case, 
$$\mathrm{Im}\mathcal W^{\epsilon_I}(\boldsymbol{\alpha}_I,\boldsymbol{\xi})=\sum_{s=1}^c2V(\alpha_{s_1},\dots,\alpha_{s_6},\xi_s)$$
with $V$ defined in (\ref{V}).

At $\big(\pi,\dots,\pi,\frac{7\pi}{4}\big),$ we have $\frac{\partial ^2 V}{\partial \alpha_{s_i}^2} =-2$ for $s_i\in I\cap\{s_1,\dots,s_6\},$ $\frac{\partial ^2V}{\partial \alpha_{s_i}\alpha_{s_j}} =-1$ for $s_i\neq s_j$ in $I\cap \{s_1,\dots,s_6\},$ $\frac{\partial ^2 V}{\partial \alpha_{s_i}\xi_s} =2$ for $s_i\in I\cap\{s_1,\dots,s_6\}$ and  $\frac{\partial ^2 V}{\partial \xi_s^2} =-8.$ Then a direct computation shows that,  at $\big(\pi,\dots,\pi,\frac{7\pi}{4}\big),$ the Hessian matrix 
of $V$ in $\{\mathrm{Re}(\alpha_i)\}_{i\in I\cap \{s_1,\dots,s_6\}}$ and $\mathrm{Re}(\xi_s)$ is negative definite. As a consequence, the Hessian matrix 
of $\mathrm{Im}\mathcal W^{\epsilon_I}$ in $\{\mathrm{Re}(\alpha_i)\}_{i\in I}$ and $\{\mathrm{Re}(\xi_s)\}_{s=1}^c$ is negative definite at $\big(\pi,\dots,\pi,\frac{7\pi}{4},\dots,\frac{7\pi}{4} \big).$

Then by the continuity, there exists a sufficiently small $\delta_0>0$ such that for all $\{\alpha_j\}_{j\in J}$ in $(\pi-\delta_0,\pi+\delta_0)$ and $(\boldsymbol{\alpha}_I,\boldsymbol{\xi})\in \mathrm{D_{\delta_0,\mathbb C}},$ the Hessian matrix of $\mathrm{Im}\mathcal W^{\epsilon_I}$ with respect to $\{\mathrm{Re}(\alpha_i)\}_{i\in I}$ and $\{\mathrm{Re}(\xi_s)\}_{s=1}^c$ is still negative definite, implying that $\mathrm{Im}\mathcal W^{\epsilon_I}$ is strictly concave down in $\{\mathrm{Re}(\alpha_i)\}_{i\in I}$ and  $\{\mathrm{Re}(\xi_s)\}_{s=1}^c$ on $\mathrm{D_{\delta_0,\mathbb C}}.$ Since $\mathcal W^{\epsilon_I}$ is holomorphic, $\mathrm{Im}\mathcal W^{\epsilon_I}$ is strictly concave up in $\{\mathrm{Im}(\alpha_i)\}_{i\in I}$ and $\{\mathrm{Im}(\xi_s)\}_{s=1}^c$ on $\mathrm{D_{\delta_0,\mathbb C}}.$
\end{proof}

\begin{proposition}\label{nonsingular} If all $\{\alpha_j\}_{j\in J}$ are in $(\pi-\delta_0,\pi+\delta_0),$ then
the Hessian matrix $\mathrm{Hess}\mathcal W^{\epsilon_I}$ of $\mathcal W^{\epsilon_I}$ with respect to $\{\alpha_i\}_{i\in I}$ and $\{\xi_s\}_{s=1}^c$ is non-singular on $\mathrm{D_{\delta_0,\mathbb C}}.$
\end{proposition}

\begin{proof} By Proposition \ref{convexity}, the real part of the $\mathrm{Hess}\mathcal W^{\epsilon_I}$ is negative definite. Then by \cite[Lemma]{L}, it is nonsingular.
\end{proof}


\subsection{Asymptotics of the leading Fourier coefficients}\label{leading}

\begin{proposition}\label{critical}
Suppose $\{\beta_i\}_{i\in I}$ and $\{\alpha_j\}_{j\in J}$ are in $(\pi-\epsilon,\pi+\epsilon)$ for a sufficiently small $\epsilon>0.$ For $\epsilon_I\in\{1,-1\}^{|I|},$ let  $\mathbf z^{\epsilon_I}$ be the critical point of $\mathcal W^{\epsilon_I}$ described in Proposition \ref{crit}. Then
$$\widehat{f^{\epsilon_I}_r}(0,\dots,0)=\frac{C^{\epsilon_I}(\mathbf z^{\epsilon_I})}{\sqrt{-\det\mathrm{Hess}\Big(\frac{\mathcal W^{\epsilon_I}(\mathbf z^{\epsilon_I})}{4\pi\sqrt{-1}}\Big)}}e^{\frac{r}{4\pi}\big(\mathrm{Vol}(M_{L_{\boldsymbol\theta}})+\sqrt{-1}\mathrm{CS}(M_{L_{\boldsymbol\theta}})\big)}\Big( 1 + O \Big( \frac{1}{r} \Big) \Big)$$
where each $C^{\epsilon_I}(\mathbf z^{\epsilon_I})$ depends continuously on $\{\beta_i
\}_{i\in I}$ and $\{\alpha_j\}_{j\in J};$ and when $\beta_i=\alpha_j=\pi,$
$$C^{\epsilon_I}(\mathbf z^{\epsilon_I})=(-1)^{\sum_{i\in I}q_i+\sum_{i=1}^n\big(p_i+\frac{\iota_i}{2}\big)+c}\frac{r^{\frac{|I|-c}{2}}}{2^{\frac{3|I|+c}{2}}\pi^{\frac{|I|+c}{2}}}.$$
\end{proposition}

 For the proof of Proposition \ref{critical}, we need the following lemma.

\begin{lemma}\label{absm} For each $\epsilon_I\in\{1,-1\}^{|I|}$ and any fixed $\{\alpha_j\}_{j\in J},$ 
$$\max_{\mathrm{D_H}}\mathrm{Im}\mathcal W^{\epsilon_I} \leqslant \mathrm{Im}\mathcal W^{\epsilon_I}\Big(\pi,\dots,\pi,\frac{7\pi}{4},\dots,\frac{7\pi}{4}\Big)=2cv_8$$
where $v_8$ is the volume of the regular ideal octahedron, and the equality holds if and only if $\alpha_1=\dots=\alpha_n=\pi$ and $\xi_1=\dots=\xi_c=\frac{7\pi}{4}.$
\end{lemma}

\begin{proof} On $\mathrm{D_H},$ we have 
$$\mathrm{Im}\mathcal W^{\epsilon_I}(\boldsymbol{\alpha}_I,\boldsymbol{\xi})=\sum_{s=1}^c2V(\alpha_{s_1},\dots,\alpha_{s_6},\xi_s)$$ 
for $V$ defined in (\ref{V}). Then the result is a consequence of the result of Costantino\,\cite{C1} and the Murakami-Yano formula\,\cite{MY} (see Ushijima\,\cite{U} for the case of hyperideal tetrahedra). Indeed, by \cite{C1}, for a fixed $\boldsymbol\alpha=(\alpha_1,\dots,\alpha_6)$ of the hyperideal type, the function $f(\xi)$ defined by $f(\xi)=V(\boldsymbol\alpha,\xi)$ is strictly concave down and the unique maximum point $\xi(\boldsymbol\alpha)$ exists and lies in $(\max\{\tau_i\},\min\{\eta_j,2\pi\}),$ ie, $(\boldsymbol\alpha,\xi(\boldsymbol\alpha))\in\mathrm{B_H}.$ Then by \cite{U}, $V(\boldsymbol\alpha,\xi(\boldsymbol\alpha))=\mathrm{Vol}(\Delta_{|\boldsymbol\alpha-\boldsymbol\pi|}),$ the volume of the truncated hyperideal tetrahedron $\Delta_{|\boldsymbol\alpha-\boldsymbol\pi|}$ with dihedral angles $|\alpha_1-\pi|,\dots, |\alpha_6-\pi|.$ Since $\xi(\pi,\dots,\pi)=\frac{7\pi}{4}$ and the regular ideal cctahedron $\Delta_{(0,\dots,0)}$ has the maximum volume among all the truncated hyperideal tetrahedra, $V\big(\pi,\dots,\pi,\frac{7\pi}{4}\big)=v_8=\mathrm{Vol}(\Delta_{(0,\dots,0)})\geqslant \mathrm{Vol}(\Delta_{|\boldsymbol\alpha-\boldsymbol\pi|})=V(\boldsymbol\alpha,\xi(\boldsymbol\alpha))\geqslant V(\boldsymbol\alpha,\xi)$ for any $(\boldsymbol\alpha,\xi)\in \mathrm{B_H}.$ 

For the equality part, suppose $(\boldsymbol{\alpha}_I,\boldsymbol{\xi})\neq \big(\pi,\dots,\pi,\frac{7\pi}{4},\dots,\frac{7\pi}{4}\big).$ If $(\alpha_{s_1},\dots,\alpha_{s_6})\neq (\pi,\dots,\pi)$ for some $s\in\{1,\dots,c\},$ then 
$\mathrm{Im}\mathcal W^{\epsilon_I}(\boldsymbol{\alpha}_I,\boldsymbol{\xi})\leqslant 2\mathrm{Vol}(\Delta_{|\boldsymbol\alpha-\boldsymbol\pi|})+2(c-1)v_8<2cv_8.$ If $\boldsymbol{\alpha}_I=(\pi,\dots,\pi)$ but $\xi_s\neq \frac{7\pi}{4}$ for some $s\in\{1,\dots,c\},$ then the strict concavity of $f(\xi)$ implies that $\mathrm{Im}\mathcal W^{\epsilon_I}(\pi,\dots,\pi,\xi)< \mathrm{Im}\mathcal W^{\epsilon_I}\big(\pi,\dots,\pi,\frac{7\pi}{4},\dots,\frac{7\pi}{4}\big)=2cv_8.$
\end{proof}

\begin{proof}[Proof of Proposition \ref{critical}] Let $\delta_0>0$ be as in Proposition \ref{convexity}. By Proposition \ref{convexity}, Proposition \ref{absm} and the compactness of $\mathrm{D_H}\setminus\mathrm{D_{\delta_0}},$
$$2cv_8>\max_{\mathrm{D_H}\setminus\mathrm{D_{\delta_0}}} \mathrm{Im}\mathcal W^{\epsilon_I}.$$
By Proposition \ref{crit} and continuity, if $\{\beta_i\}_{i\in I}$ and $\{\alpha_j\}_{j\in J}$ are sufficiently close to $\pi,$ then the critical point $\mathbf z^{\epsilon_I}$ of $\mathcal W^{\epsilon_I}$  as in Proposition \ref{crit} lies in $\mathrm{D_{\delta_0,\mathbb C}},$ and $\mathrm{Im}\mathcal W^{\epsilon_I}(\mathbf z^{\epsilon_I})=\mathrm{Vol}(M_{L_{\boldsymbol\theta}})$ is sufficiently close to $2cv_8$ so that
 $$\mathrm{Im}\mathcal W^{\epsilon_I}(\mathbf z^{\epsilon_I})>\max_{\mathrm{D_H}\setminus\mathrm{D_{\delta_0}}} \mathrm{Im}\mathcal W^{\epsilon_I}.$$
 
Therefore, we only need to estimate  the integral on $\mathrm{D_{\delta_0}}.$ By analyticity, the integral remains unchanged  if we deform the domain of integral from $\mathrm{D_{\delta_0}}$ to a different surface that shares the same boundary with $\mathrm{D_{\delta_0}}.$ Now we define such a new surface $S^{\epsilon_I}$ as drawn in Figure \ref{surface}, over which the integral is easier to estimate. I.e., $S^{\epsilon_I}=S^{\epsilon_I}_{\text{top}}\cup S^{\epsilon_I}_{\text{side}}$ in $\overline{\mathrm{D_{\delta_0,\mathbb C}}},$ where
$$S^{\epsilon_I}_{\text{top}}=\{ (\boldsymbol{\alpha}_I,\boldsymbol{\xi})\in \mathrm{D_{\delta_0,\mathbb C}}\ |\ (\mathrm{Im}(\boldsymbol{\alpha}_I),\mathrm{Im}(\boldsymbol\xi))=\mathrm{Im}(\mathbf z^{\epsilon_I})\}$$
and
$$S^{\epsilon_I}_{\text{side}}=\{ (\boldsymbol{\alpha}_I,\boldsymbol{\xi})+t\sqrt{-1}\mathrm{Im}(\mathbf z^{\epsilon_I})\ |\ (\boldsymbol{\alpha}_I,\boldsymbol{\xi})\in\partial \mathrm{D_{\delta_0}}, t\in[0,1]\}.$$

\begin{figure}[htbp]
\centering
\includegraphics[scale=0.3]{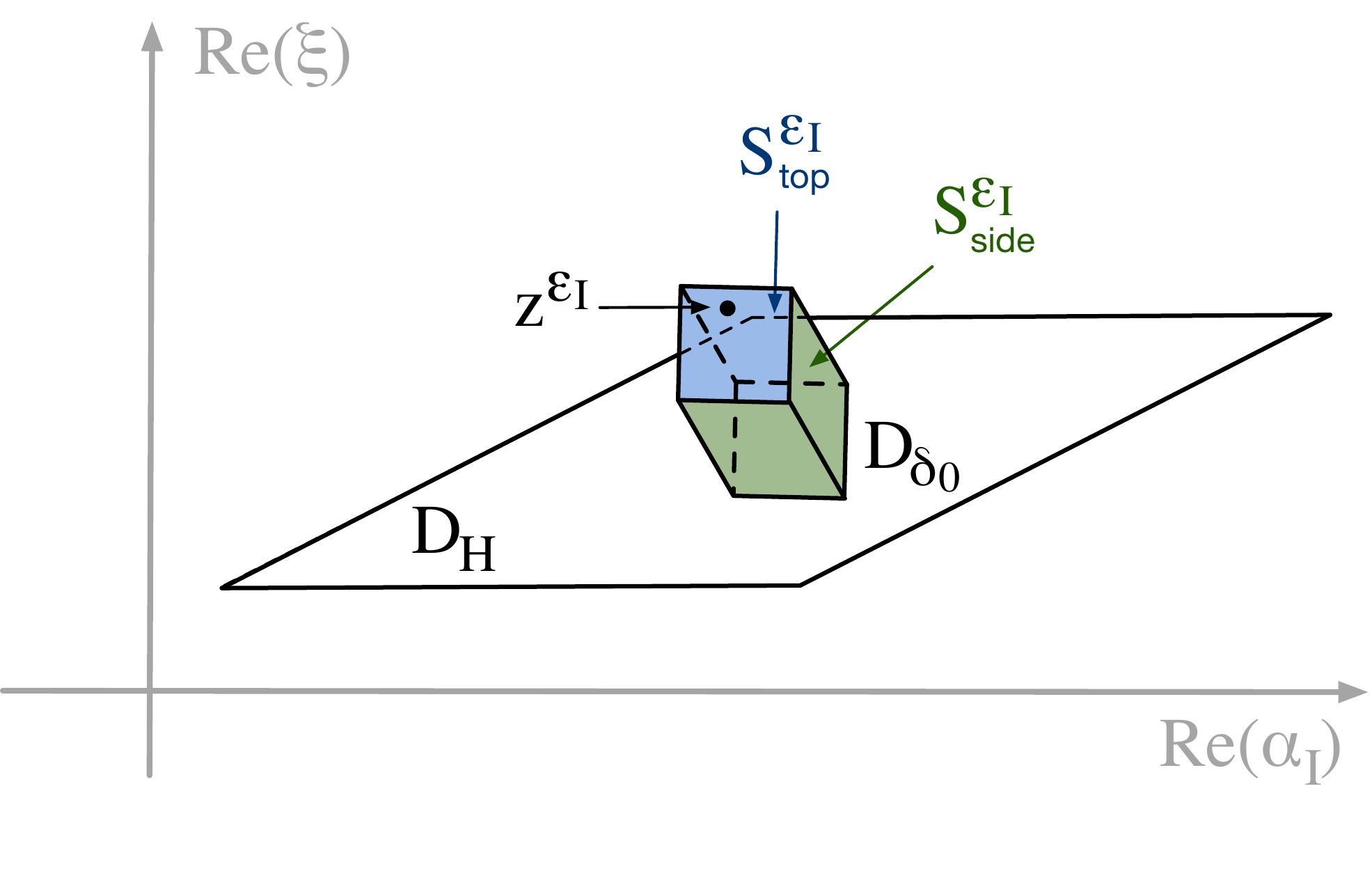}
\caption{The deformed surface $S^{\epsilon_I},$ where $S^{\epsilon_I}_{\text{top}}$ is the ``top" of the rectangular cylinder colored in blue, $S^{\epsilon_I}_{\text{side}}$ is the ``side" of the rectangular cylinder colored in green, and $D_{\delta_0}$ is the ``bottom" of the rectangular cylinder.}
\label{surface} 
\end{figure}

By Proposition \ref{convexity}, $\mathrm{Im}\mathcal W^{\epsilon_I}$  is concave down on $S^{\epsilon_I}_{\text{top}}.$ Since $\mathbf z^{\epsilon_I}$ is the critical points of $\mathrm{Im}\mathcal W^{\epsilon_I},$ it is the only absolute maximum on $S^{\epsilon_I}_{\text{top}}.$

On the side $S^{\epsilon_I}_{\text{side}},$ for each $(\boldsymbol{\alpha}_I,\boldsymbol{\xi})\in \partial \mathrm{D_{\delta_0}},$ we consider the function 
$$g^{\epsilon_I}_{(\boldsymbol{\alpha}_I,\boldsymbol{\xi})}(t)= \mathrm{Im}\mathcal W^{\epsilon_I}\big((\boldsymbol{\alpha}_I,\boldsymbol{\xi})+t\sqrt{-1}\mathrm{Im}(\mathbf z^{\epsilon_I})\big)$$
on $[0,1].$ By Proposition \ref{convexity}, $g^{\epsilon_I}_{(\boldsymbol{\alpha}_I,\boldsymbol{\xi})}(t)$ is concave up for any $(\boldsymbol{\alpha}_I,\boldsymbol{\xi})\in \partial \mathrm{D_{\delta_0}}.$ As a consequence, $g^{\epsilon_I}_{(\boldsymbol{\alpha}_I,\boldsymbol{\xi})}(t)\leqslant \max\{g^{\epsilon_I}_{(\boldsymbol{\alpha}_I,\boldsymbol{\xi})}(0), g^{\epsilon_I}_{(\boldsymbol{\alpha}_I,\boldsymbol{\xi})}(1)\}.$ Now by the previous two steps, since $(\boldsymbol{\alpha}_I,\boldsymbol{\xi})\in \partial \mathrm{D_{\delta_0}},$ 
$$g^{\epsilon_I}_{(\boldsymbol{\alpha}_I,\boldsymbol{\xi})}(0)= \mathrm{Im}\mathcal W^{\epsilon_I}(\boldsymbol{\alpha}_I,\boldsymbol{\xi})<\mathrm{Im}\mathcal W^{\epsilon_I}(\mathbf z^{\epsilon_I});$$
and since $(\boldsymbol{\alpha}_I,\boldsymbol{\xi})+\sqrt{-1} \mathrm{Im}(\mathbf z^{\epsilon_I})\in S^{\epsilon_I}_{\text{top}},$
$$g^{\epsilon_I}_{(\boldsymbol{\alpha}_I,\boldsymbol{\xi})}(1)= \mathrm{Im}\mathcal W^{\epsilon_I}\big((\boldsymbol{\alpha}_I,\boldsymbol{\xi})+\sqrt{-1}\mathrm{Im}(\mathbf z^{\epsilon_I}))<\mathrm{Im}\mathcal W^{\epsilon_I}(\mathbf z^{\epsilon_I}\big).$$ 
As a consequence,
 $$\mathrm{Im}\mathcal W^{\epsilon_I}(\mathbf z^{\epsilon_I})>\max_{S^{\epsilon_I}_{\text{side}}} \mathrm{Im}\mathcal W^{\epsilon_I}.$$

Therefore,  $\mathbf z^{\epsilon_I}$ is the unique maximum point of $\mathrm{Im}\mathcal W^{\epsilon_I}$ on $S^{\epsilon_I}\cup\big( \mathrm{D_H}\setminus\mathrm{D_{\delta_0}}\big),$ and $\mathcal W^{\epsilon_I}$ has critical value $2c\pi^2+\sqrt{-1}\big(\mathrm{Vol}(M_{L_{\boldsymbol\theta}})+\sqrt{-1}\mathrm{CS}(M_{L_{\boldsymbol\theta}})\big)$ at $\mathbf z^{\epsilon_I}.$

By Proposition \ref{nonsingular}, $\det\mathrm{Hess}\mathcal W^{\epsilon_I}(\mathbf z^{\epsilon_I})\neq 0.$

Next, we prove that the domain
$$\big\{(\boldsymbol\alpha_I,\boldsymbol \xi)\in \overline{\mathrm D_{\delta_0,\mathbb C}}\ \big|\ \mathrm{Im}\mathcal W^{\epsilon_I}(\boldsymbol\alpha_I,\boldsymbol \xi)<\mathrm{Im}\mathcal W^{\epsilon_I}(\mathbf z^{\epsilon_I})\big\} $$
deformation retracts to $S_{\text{top}}^{\epsilon_I}\setminus\{\mathbf z^{\epsilon_I}\}.$ To see this, for each $\mathbf x \in \mathrm D_{\delta_0},$ let 
$$P_{\mathbf x}=\big\{(\boldsymbol\alpha_I,\boldsymbol\xi)\in \mathrm D_{\delta_0,\mathbb C}\ \big|\ \mathrm{Re}(\boldsymbol\alpha_I,\boldsymbol \xi)=\mathbf x\big\}$$
and 
$$B_{\mathbf x}=\big\{(\boldsymbol\alpha_I,\boldsymbol\xi)\in P_{\mathbf x}\ \big|\ \mathrm{Im}\mathcal W^{\epsilon_I}(\boldsymbol\alpha_I,\boldsymbol \xi)<\mathrm{Im}\mathcal W^{\epsilon_I}(\mathbf z^{\epsilon_I})\big\}.$$ Then by Proposition \ref{convexity} that $\mathrm{Im}\mathcal W^{\epsilon_I}$ is concave up in $\mathrm{Im}(\boldsymbol\alpha_I,\boldsymbol \xi),$  $B_{\mathrm{Re}(\mathbf z^{\epsilon_I})}=\emptyset,$ and $B_{\mathbf x}$ is a non-empty convex subset of $P_{\mathbf x}$ for $\mathbf x\neq \mathrm{Re}(\mathbf z^{\epsilon_I});$  and by the fact that $\mathbf z^{\epsilon_I}$ is the unique maximum point of $\mathrm{Im}\mathcal W^{\epsilon_I}$ on $S^{\epsilon_I},$ $\mathbf x+\sqrt{-1}\mathrm{Im}(\mathbf z^{\epsilon_I})\in B_{\mathbf z}$ for $\mathbf x \neq \mathrm{Re}(\mathbf z^{\epsilon_I}).$ 
 As a consequence, $B_{\mathbf x}$ deformation retracts to $\mathbf x+\sqrt{-1}\mathrm{Im}(\mathbf z^{\epsilon_I})$ which induces the desired deformation retraction of $\big\{(\boldsymbol\alpha_I,\boldsymbol \xi)\in \overline{\mathrm D_{\delta_0,\mathbb C}}\ \big|\ \mathrm{Im}\mathcal W^{\epsilon_I}(\boldsymbol\alpha_I,\boldsymbol \xi)<\mathrm{Im}\mathcal W^{\epsilon_I}(\mathbf z^{\epsilon_I})\big\} $ to $S_{\text{top}}^{\epsilon_I}\setminus\{\mathbf z^{\epsilon_I}\}.$

 Finally,  we estimate the difference between $\mathcal W^{\epsilon_I}_r$ and $\mathcal W^{\epsilon_I}.$ By Lemma \ref{converge}, (3), we have
 $$\varphi_r\Big(\frac{\pi}{r}\Big)=\mathrm{Li}_2(1)+\frac{2\pi\sqrt{-1}}{r}\log\Big(\frac{r}{2}\Big)-\frac{\pi^2}{r}+O\Big(\frac{1}{r^2}\Big);$$
and for $z$ with $0<\mathrm{Re z}<\pi$ we have 
 $$\varphi_r\Big(z+\frac{k\pi}{r}\Big)=\varphi_r(z)+\varphi'_r(z)\frac{k\pi}{r}+O\Big(\frac{1}{r^2}\Big).$$
 Then by Lemma \ref{converge}, in $\big\{(\boldsymbol{\alpha}_I,\boldsymbol{\xi})\in \overline{\mathrm{D_{\delta_0,\mathbb C}}}\ \big|\ |\mathrm{Im}(\alpha_i)| < L\text{ for } i\in I, |\mathrm{Im}(\xi_s)| < L\text{ for }s\in\{1,\dots,c\}\big\}$ for some $L>0,$
 $${\mathcal W}^{\epsilon_I}_r(\boldsymbol{\alpha}_I,\boldsymbol{\xi})=\mathcal W^{\epsilon_I}(\boldsymbol{\alpha}_I,\boldsymbol{\xi})-\frac{4c\pi\sqrt{-1}}{r}\log\Big(\frac{r}{2}\Big)+\frac{4\pi \sqrt{-1}\kappa(\boldsymbol{\alpha}_I,\boldsymbol{\xi})}{r}+\frac{\nu_r(\boldsymbol{\alpha}_I,\boldsymbol{\xi})}{r^2},$$
with
 \begin{equation*}
 \begin{split}
&\kappa(\boldsymbol{\alpha}_I,\boldsymbol{\xi})\\
=&\sum_{s=1}^c\Big(\frac{1}{2}\sum_{i=1}^4 \sqrt{-1}\tau_{s_i}- \sqrt{-1}\xi_s-\sqrt{-1}\pi-\frac{\sqrt{-1}\pi}{2}\\
&+\frac{1}{4}\sum_{i=1}^4\sum_{j=1}^3\log\big(1-e^{2\sqrt{-1}(\eta_{s_j}-\tau_{s_i})}\big)-\frac{3}{4}\sum_{i=1}^4\log\big(1-e^{2\sqrt{-1}(\tau_{s_i}-\pi)}\big)\\
&+\frac{3}{2}\log\big(1-e^{2\sqrt{-1}(\xi_s-\pi)}\big)-\frac{1}{2}\sum_{i=1}^4\log\big(1-e^{2\sqrt{-1}(\xi_s-\tau_{s_i})}\big)-\frac{1}{2}\sum_{j=1}^3\log\big(1-e^{2\sqrt{-1}(\eta_{s_j}-\xi_s)}\big)\Big);
 \end{split}
 \end{equation*}
 and by the compactness of $ \overline{\mathrm{D_{\delta_0,\mathbb C}}},$ $|\nu_r(\boldsymbol{\alpha}_I,\boldsymbol{\xi})|$ is bounded from above by a constant independent of $r.$
 Then
  \begin{equation*}
 \begin{split}
&e^{\big(\sum_{i\in I}q_i\beta_i+\sum_{i=1}^n\big(p_i+\frac{\iota_i}{2}\big)\alpha_i+\sum_{i\in I}\epsilon_i(\alpha_i+\beta_i+\frac{2\pi}{r})\big)\sqrt{-1}+\frac{r}{4\pi \sqrt{-1}}{\mathcal W}^{\epsilon_I}_r(\boldsymbol{\alpha}_I,\boldsymbol{\xi})}\\
=&\Big(\frac{r}{2}\Big)^{-c}e^{\big(\sum_{i\in I}q_i\beta_i+\sum_{i=1}^n\big(p_i+\frac{\iota_i}{2}\big)\alpha_i+\sum_{i\in I}\epsilon_i(\alpha_i+\beta_i)\big)\sqrt{-1}+\kappa(\boldsymbol{\alpha}_I,\boldsymbol{\xi})}\cdot e^{\frac{r}{4\pi \sqrt{-1}}\big(\mathcal W^{\epsilon_I}(\boldsymbol{\alpha}_I,\boldsymbol{\xi})+\frac{\nu_r(\boldsymbol{\alpha}_I,\boldsymbol{\xi})-\sum_{i\in I}\epsilon_i 8\pi^2}{r^2}\big)}.
 \end{split}
 \end{equation*}

Now we  apply Proposition \ref{saddle} to conclude the result. Let $D_{\mathbf z}$ be the region 
$$\big\{(\boldsymbol{\alpha}_I,\boldsymbol{\xi})\in \overline{\mathrm{D_{\delta_0,\mathbb C}}}\ \big|\ |\mathrm{Im}(\alpha_i)| < L\text{ for } i\in I, |\mathrm{Im}(\xi_s)| < L\text{ for }s\in\{1,\dots,c\}\big\}$$ for some $L>0.$ Let $\mathbf a_r=((\beta_i)_{i\in I},(\alpha_j)_{j\in J})$ (recall that $\beta_i=\frac{2\pi n_i}{r}$ and $\alpha_j=\frac{2\pi m_j}{r}$ depends on $r$),
$f^{\mathbf a_r}(\boldsymbol{\alpha}_I,\boldsymbol{\xi})=\frac{\mathcal W^{\epsilon_I}(\boldsymbol{\alpha}_I,\boldsymbol{\xi})}{4\pi\sqrt{-1}},$ $g^{\mathbf a_r}(\boldsymbol{\alpha}_I,\boldsymbol{\xi})=\psi(\boldsymbol{\alpha}_I,\boldsymbol{\xi})e^{(\sum_{i\in I}q_i\beta_i+\sum_{i=1}^n\big(p_i+\frac{\iota_i}{2}\big)\alpha_i+\sum_{i\in I}\epsilon_i(\alpha_i+\beta_i))\sqrt{-1}+\kappa(\boldsymbol{\alpha}_I,\boldsymbol{\xi})},$ $f_r^{\mathbf a_r}(\boldsymbol{\alpha}_I,\boldsymbol{\xi})=\frac{{\mathcal W}_r^{\epsilon_I}(\boldsymbol{\alpha}_I,\boldsymbol{\xi})}{4\pi\sqrt{-1}}-\frac{c}{r}\log\big(\frac{r}{2}\big),$ $\upsilon_r(\boldsymbol{\alpha}_I,\boldsymbol{\xi})=\nu_r(\boldsymbol{\alpha}_I,\boldsymbol{\xi})-\sum_{i\in I}\epsilon_i 8\pi^2,$  $S_r=S^{\epsilon_I}$ and $\mathbf z^{\epsilon_I}$ is the critical point of $f^{\mathbf a_r}$ in $D_{\mathbf z}.$ Then all the conditions of Proposition \ref{saddle} are satisfied with and the result follows. 

When $\beta_i=\alpha_j=\pi,$ a direct computation shows that 
\begin{equation*}
\begin{split}
C^{\epsilon_I}(\mathbf z^{\epsilon_I})=&\frac{r^{|I|+c}}{2^{2|I|+c}\pi^{|I|+c}}\Big(\frac{2\pi}{r}\Big)^{
\frac{|I|+c}{2}}\Big(\frac{r}{2}\Big)^{-c}g\Big(\pi,\dots,\pi,\frac{7\pi}{4},\dots,\frac{7\pi}{4}\Big)\\
=&(-1)^{\sum_{i\in I}q_i+\sum_{i=1}^n\big(p_i+\frac{\iota_i}{2}\big)+c}\frac{r^{\frac{|I|-c}{2}}}{2^{\frac{3|I|+c}{2}}\pi^{\frac{|I|+c}{2}}}.
\end{split}
\end{equation*}
\end{proof}

\begin{corollary}\label{5.8} If $\epsilon>0$ is sufficiently small and all $\{\beta_i\}_{i\in I}$ and $\{\alpha_j\}_{j\in J}$ are in $(\pi-\epsilon,\pi+\epsilon),$ then 
$$\sum_{\epsilon_I\in\{1,-1\}^{|I|}}\frac{C^{\epsilon_I}(\mathbf z^{\epsilon_I})}{\sqrt{-\det\mathrm{Hess}\Big(\frac{\mathcal W^{\epsilon_I}(\mathbf z^{\epsilon_I})}{4\pi\sqrt{-1}}\Big)}}\neq 0.$$
\end{corollary}

\begin{proof} If $\beta_i=\alpha_j=\pi$ for all $i\in I$ and $j\in J,$ then all $\mathbf z^{\epsilon_I}=\big(\pi,\dots,\pi,\frac{7\pi}{4},\dots,\frac{7\pi}{4}\big)$ and all $\mathcal W^{\epsilon_I}$ are the same functions. As a consequence, all the $C^{\epsilon_I}(\mathbf z^{\epsilon_I})$'s and all Hessian determinants $\det\mathrm{Hess}\Big(\frac{\mathcal W^{\epsilon_I}(\mathbf z^{\epsilon_I})}{4\pi\sqrt{-1}}\Big)$'s are the same at this point, imply that the sum is not equal to zero. Then by continuity, if $\epsilon$ is small enough, then the sum remains none zero.
\end{proof}

\begin{remark} In \cite{WY3}, we proved that all $C^{\epsilon_I}(\mathbf z^{\epsilon_I})$'s and all $\det\mathrm{Hess}\Big(\frac{\mathcal W^{\epsilon_I}(\mathbf z^{\epsilon_I})}{4\pi\sqrt{-1}}\Big)$'s are always the same for any given $\{\beta_i\}_{i\in I}$ and $\{\alpha_j\}_{j\in J},$ and related them to the adjoint twisted Reidemeister torsion of $M_{L_{\boldsymbol\theta}}.$ 
\end{remark}

\subsection{Estimate of the other Fourier coefficients}\label{ot}

\begin{proposition}\label{other} Suppose $\{\beta_i\}_{i\in I}$ and $\{\alpha_j\}_{j\in J}$ are in $(\pi-\epsilon,\pi+\epsilon)$ for a sufficiently small $\epsilon>0.$  If $(\mathbf a_I,\mathbf b)\neq(0,\dots,0),$ then
$$\Big|\widehat{f^{\epsilon_I}_r}(\mathbf a_I,\mathbf b)\Big|<O\Big(e^{\frac{r}{4\pi}\big(\mathrm{Vol}(M_{L_{\boldsymbol\theta}})-\epsilon'\big)}\Big)$$
for some $\epsilon'>0.$
\end{proposition}

\begin{proof} Recall that if $\beta_i=\alpha_j=\pi$ for all $i\in I$ and $j\in J,$ then the total derivative
$$D\mathcal W^{\epsilon_I}\Big(\pi,\dots,\pi,\frac{7\pi}{4},\dots,\frac{7\pi}{4}\Big)=(0,\dots,0).$$
Hence by continuity, all the partial derivatives near this critical point are sufficiently small. To be precise, there exists a $\delta_1>0$ and an $\epsilon>0$ such that if $\{\beta_i\}_{i\in I}$ and $\{\alpha_j\}_{j\in J}$ are in $(\pi-\epsilon,\pi+\epsilon),$ then for all $(\boldsymbol{\alpha}_I,\boldsymbol{\xi})\in D_{\delta_1,\mathbb C}$ and for any unit vector $\mathbf u=((u_i)_{i\in I},(w_s)_{s=1}^c)\in \mathbb R^{|I|+c},$ the directional derivatives 
$$|D_{\mathbf u}\mathrm{Im}\mathcal W^{\epsilon_I}(\boldsymbol{\alpha}_I,\boldsymbol{\xi})|=\Big|\sum_{i\in I}u_i\frac{\partial \mathrm{Im}\mathcal W^{\epsilon_I}}{\partial \mathrm{Im}(\alpha_i)}+\sum_{s=1}^cw_s\frac{\partial \mathrm{Im}\mathcal W^{\epsilon_I}}{\partial \mathrm{Im}(\xi_s)}\Big|<\frac{2\pi-\epsilon''}{2\sqrt{2|I|+2c}}$$
for some $\epsilon''>0.$

On $\mathrm{D_H},$ we have 
\begin{equation*}
\begin{split}
 &\mathrm{Im}\Big(\mathcal W^{\epsilon_I}(\boldsymbol{\alpha}_I,\boldsymbol{\xi})-\sum_{i\in I}2\pi a_i\alpha_i-\sum_{s=1}^c4\pi b_s\xi_s\Big)=\mathrm{Im}\mathcal W^{\epsilon_I}(\boldsymbol{\alpha}_I,\boldsymbol{\xi}).
\end{split}
\end{equation*}
Then by Lemma \ref{convexity}, Proposition \ref{absm} and the compactness of $\mathrm{D_H}\setminus\mathrm{D_{\delta_1}},$
$$2cv_8>\max_{\mathrm{D_H}\setminus\mathrm{D_{\delta_1}}} \mathrm{Im}\Big(\mathcal W^{\epsilon_I}(\boldsymbol{\alpha}_I,\boldsymbol{\xi})-\sum_{i\in I}2\pi a_i\alpha_i-\sum_{s=1}^c4\pi b_s\xi_s\Big)+\epsilon'''$$
for some $\epsilon'''>0.$
By Proposition \ref{crit} and continuity, if $\{\beta_i\}_{i\in I}$ and $\{\alpha_j\}_{j\in J}$ are sufficiently close to $\pi,$ then the critical point $\mathbf z^{\epsilon_I}$ of $\mathcal W^{\epsilon_I}$  as in Proposition \ref{crit} lies in $\mathrm{D_{\delta_1,\mathbb C}},$ and $\mathrm{Im}\mathcal W^{\epsilon_I}(\mathbf z^{\epsilon_I})=\mathrm{Vol}(M_{L_{\boldsymbol\theta}})$ is sufficiently close to $2cv_8$ so that
\begin{equation}\label{b}
\mathrm{Im}\mathcal W^{\epsilon_I}(\mathbf z^{\epsilon_I})>\max_{\mathrm{D_H}\setminus\mathrm{D_{\delta_1}}} \mathrm{Im}\Big(\mathcal W^{\epsilon_I}(\boldsymbol{\alpha}_I,\boldsymbol{\xi})-\sum_{i\in I}2\pi a_i\alpha_i-\sum_{s=1}^c4\pi b_s\xi_s\Big)+\epsilon'''.
\end{equation}
 
Therefore, we only need to estimate  the integral on $\mathrm{D_{\delta_1}}.$ Again, by analyticity, the integral remains unchanged if we deform $\mathrm{D_{\delta_1}}$ to a different surface sharing the same boundary, over which the integral is easier to estimate. 

If $(\mathbf a_I,\mathbf b)\neq (0,\dots,0),$ then there is at least one of $\{a_i\}_{i\in I}$ or $\{b_s\}_{s=1}^c$ that is nonzero. Without loss of generality, assume that $a_1\neq 0.$

If $a_1>0,$ then consider the surface $S^+=S^+_{\text{top}}\cup S^+_{\text{side}}$ in $\overline{\mathrm{D_{\delta_1,\mathbb C}}}$ where
$$S^+_{\text{top}}=\{ (\boldsymbol{\alpha}_I,\boldsymbol{\xi})\in \mathrm{D_{\delta_1,\mathbb C}}\ |\ (\mathrm{Im}(\boldsymbol{\alpha}_I),\mathrm{Im}(\boldsymbol\xi))=(\delta_1,0,\dots,0)\}$$
and
$$S^+_{\text{side}}=\{ (\boldsymbol{\alpha}_I,\boldsymbol{\xi})+(t\sqrt{-1}\delta_1,0,\dots,0)\ |\ (\boldsymbol{\alpha}_I,\boldsymbol{\xi})\in\partial \mathrm{D_{\delta_1}},t\in[0,1]\}.$$      
On the top, for any $(\boldsymbol{\alpha}_I,\boldsymbol{\xi})\in S^+_{\text{top}},$ by the Mean Value Theorem, 
\begin{equation*}
\begin{split}
\big|\mathrm{Im}\mathcal W^{\epsilon_I}(\mathbf z^{\epsilon_I})-\mathrm{Im}\mathcal W^{\epsilon_I}(\boldsymbol{\alpha}_I,\boldsymbol{\xi})\big|
=&\big|D_{\mathbf u}\mathrm{Im}\mathcal W^{\epsilon_I}(z)\big|\cdot\big\|\mathbf z^{\epsilon_I}-(\boldsymbol{\alpha}_I,\boldsymbol{\xi})\big\|\\
<&\frac{2\pi-\epsilon''}{2\sqrt{2|I|+2c}}\cdot2\sqrt{2|I|+2c} \delta_1\\
=&2\pi\delta_1-\epsilon''\delta_1,
\end{split}
\end{equation*}
where $z$ is some point on the line segment connecting $\mathbf z^{\epsilon_I}$ and $(\boldsymbol{\alpha}_I,\boldsymbol{\xi}),$ $\mathbf u=\frac{\mathbf z^{\epsilon_I}-(\boldsymbol{\alpha}_I,\boldsymbol{\xi})}{\|\mathbf z^{\epsilon_I}-(\boldsymbol{\alpha}_I,\boldsymbol{\xi})\|}$ and $2\sqrt{2|I|+2c} \delta_1$ is the diameter of $\mathrm{D_{\delta_1,\mathbb C}}.$
Then
\begin{equation*}
\begin{split}
\mathrm{Im}\Big(\mathcal W^{\epsilon_I}(\boldsymbol{\alpha}_I,\boldsymbol{\xi})-\sum_{i\in I}2\pi a_i\alpha_i-\sum_{s=1}^c4\pi b_s\xi_s\Big)=&\mathrm{Im}\mathcal W^{\epsilon_I}(\boldsymbol{\alpha}_I,\boldsymbol{\xi})-2\pi a_1\delta_1\\
<&\mathrm{Im}\mathcal W^{\epsilon_I}(\mathbf z^{\epsilon_I})+2\pi\delta_1-\epsilon''\delta_1-2\pi\delta_1\\
=&\mathrm{Im}\mathcal W^{\epsilon_I}(\mathbf z^{\epsilon_I})-\epsilon'' \delta_1.
\end{split}
\end{equation*}

On the side, for any point $(\boldsymbol{\alpha}_I,\boldsymbol{\xi})+(t\sqrt{-1}\delta_1,0,\dots,0)\in S^+_{\text{side}},$ by the Mean Value Theorem again, we have
$$\big|\mathrm{Im}\mathcal W^{\epsilon_I}\big((\boldsymbol{\alpha}_I,\boldsymbol{\xi})+(t\sqrt{-1}\delta_1,0,\dots,0)\big)-\mathrm{Im}\mathcal W^{\epsilon_I}(\boldsymbol{\alpha}_I,\boldsymbol{\xi})\big|<\frac{2\pi-\epsilon''}{2\sqrt{2|I|+2c}} t\delta_1.$$
Then 
\begin{equation*}
\begin{split}
\mathrm{Im}\mathcal W^{\epsilon_I}\big((\boldsymbol{\alpha}_I,\boldsymbol{\xi})+(t\sqrt{-1}\delta_1,0,\dots,0)\big)-2\pi a_1 t\delta_1<&\mathrm{Im}\mathcal W^{\epsilon_I}(\boldsymbol{\alpha}_I,\boldsymbol{\xi})+\frac{2\pi-\epsilon''}{2\sqrt{2|I|+2c}} t\delta_1-2\pi t\delta_1\\
<&\mathrm{Im}\mathcal W^{\epsilon_I}(\boldsymbol{\alpha}_I,\boldsymbol{\xi})\\
<&\mathrm{Im}\mathcal W^{\epsilon_I}(\mathbf z^{\epsilon_I})-\epsilon''',
\end{split}
\end{equation*}
where the last inequality comes from that $(\boldsymbol{\alpha}_I,\boldsymbol{\xi})\in \partial \mathrm{D_{\delta_1}}\subset \mathrm{D_H}\setminus\mathrm{D_{\delta_1}}$ and (\ref{b}).

Now let $\epsilon'=\min\{\epsilon''\delta_1,\epsilon'''\},$ then on $S^+\cup \big(\mathrm{D_H}\setminus\mathrm{D_{\delta_1}}\big),$ 
$$\mathrm{Im}\Big(\mathcal W^{\epsilon_I}(\boldsymbol{\alpha}_I,\boldsymbol{\xi})-\sum_{i\in I}2\pi a_i\alpha_i-\sum_{s=1}^c4\pi b_s\xi_s\Big)<\mathrm{Im}\mathcal W^{\epsilon_I}(\mathbf z^{\epsilon_I})-\epsilon',$$
and the result follows.

If $a_1<0,$ then we consider the surface $S^-=S^-_{\text{top}}\cup S^-_{\text{side}}$ in $\overline{\mathrm{D_{\delta_1,\mathbb C}}}$ where
$$S^-_{\text{top}}=\{ (\boldsymbol{\alpha}_I,\boldsymbol{\xi})\in \mathrm{D_{\delta_1,\mathbb C}}\ |\ (\mathrm{Im}(\boldsymbol{\alpha}_I),\mathrm{Im}(\boldsymbol\xi))=(-\delta_1,0,\dots,0)\}$$
and
$$S^-_{\text{side}}=\{ (\boldsymbol{\alpha}_I,\boldsymbol{\xi})-(t\sqrt{-1}\delta_1,0,\dots,0)\ |\ (\boldsymbol{\alpha}_I,\boldsymbol{\xi})\in\partial \mathrm{D_{\delta_1}},t\in[0,1]\}.$$      
Then the same estimate as in the previous case proves that on
$S^-\cup \big(\mathrm{D_H}\setminus\mathrm{D_{\delta_1}}\big),$ 
$$\mathrm{Im}\Big(\mathcal W^{\epsilon_I}(\boldsymbol{\alpha}_I,\boldsymbol{\xi})-\sum_{i\in I}2\pi a_i\alpha_i-\sum_{s=1}^c4\pi b_s\xi_s\Big)<\mathrm{Im}\mathcal W^{\epsilon_I}(\mathbf z^{\epsilon_I})-\epsilon',$$
from which the result follows.
\end{proof}


\subsection{Estimate of the error term}\label{ee}

The goal of this section is to estimate the error term in Proposition \ref{Poisson}.

\begin{proposition}\label{error} Suppose $\{\alpha_j\}_{j\in J}$ are in $(\pi-\epsilon,\pi+\epsilon
)$ for a sufficiently small $\epsilon>0.$ Then
the error term in Proposition \ref{Poisson} is less than $O\big(e^{\frac{r}{4\pi}(\mathrm{Vol}(M_{L_{\boldsymbol\theta}})-\epsilon')}\big)$
for some $\epsilon'>0.$
\end{proposition}

For the proof we need the following estimate, which first appeared in \cite[Proposition 8.2]{GL} for $q=e^{\frac{\pi \sqrt{-1}}{r}},$ and for the root $q=e^{\frac{2\pi \sqrt{-1}}{r}}$ in \cite[Proposition 4.1]{DK}.

\begin{lemma}\label{est}
 For any integer $0<n<r$ and at $q=e^{\frac{2\pi \sqrt{-1}}{r}},$
 $$ \log\left|\{n\}!\right|=-\frac{r}{2\pi}\Lambda\left(\frac{2n\pi}{r}\right)+O\left(\log(r)\right).$$
\end{lemma}

\begin{proof}[Proof of Proposition \ref{error} ]
For a fixed $\boldsymbol{\alpha}_J=(\alpha_j)_{j\in J},$ let 
\begin{equation*}
M_{\boldsymbol{\alpha}_J}=\max\Big\{\sum_{s=1}^c2V(\alpha_{s_1},\dots,\alpha_{s_6},\xi_s)\ \Big|\ (\boldsymbol{\alpha}_I,\boldsymbol{\xi})\in\partial \mathrm {D_H}\cup\big(\mathrm {D_A}\setminus \mathrm{D_H}\big)\Big\}
\end{equation*}
where $V$ is as defined in (\ref{V}). Then by \cite[Sections 3 \& 4]{BDKY}, 
$$M_{\boldsymbol{\alpha}_J}<2cv_8;$$
and by continuity, if $\epsilon$ is sufficiently small, then $\mathrm{Vol}(M_{L_{\boldsymbol\theta}})$ is sufficiently close to $2cv_8,$ which is the volume of the fundamental shadow link complement in the complete hyperbolic metric where all the cone angles are zero.
Therefore, if $\epsilon$ is sufficiently small, then
$$M_{\boldsymbol{\alpha}_J}<\mathrm{Vol}(M_{L_{\boldsymbol\theta}})$$
for all $\{\beta_i\}_{i\in I}$ and $\{\alpha_j\}_{j\in J}$ in $(\pi-\epsilon,\pi+\epsilon).$

Now by Lemma \ref{est} and the continuity, for $\epsilon'=\frac{\mathrm{Vol}(M_{L_{\boldsymbol\theta}})-M_{\boldsymbol{\alpha}_J}}{2},$ we can choose a sufficiently small $\delta>0$ so that if $\big(\frac{2\pi \mathbf m_I}{r},\frac{2\pi \mathbf k}{r}\big)\notin \mathrm{D_H^\delta},$ then
$$\Big|g_r^{\epsilon_I}(\mathbf m_I, \mathbf k)\Big|<O\Big(e^{\frac{r}{4\pi}(M_{\boldsymbol{\alpha}_J}+\epsilon')}\Big)=O\Big(e^{\frac{r}{4\pi}(\mathrm{Vol}(M_{L_{\boldsymbol\theta}})-\epsilon')}\Big).$$
Let $\psi$ be the bump function supported on $(\mathrm{D_H}, \mathrm{D_H^{\delta}}).$ Then the error term in Proposition \ref{Poisson} is less than $O\big(e^{\frac{r}{4\pi}(\mathrm{Vol}(M_{L_{\boldsymbol\theta}})-\epsilon')}\big).$
\end{proof}


\subsection{Proof of Theorem \ref{main2}}\label{pf}

\begin{proof}[Proof of Theorem \ref{main2}] Let $\epsilon>0$ be sufficiently small so that the conditions of Propositions \ref{critical}, \ref{other} and  \ref{error} and of Corollary \ref{5.8} are satisfied, and suppose $\{\beta_i\}_{i\in I}$ and $\{\alpha_j\}_{j\in J}$ are all in $(\pi-\epsilon, \pi+\epsilon).$ 
By Propositions \ref{4.2}, \ref{Poisson}, \ref{critical}, \ref{other} and  \ref{error}, we have
\begin{equation*}
\begin{split}
&\mathrm{RT}_r(M,L,(\mathbf n_I,\mathbf m_J))\\
=&\kappa_r\Big(\sum_{\epsilon_I\in\{1,-1\}^{|I|}}\widehat{ f_r^{\epsilon_I}}(0,\dots,0)\Big)\Big(1+O\big(e^{\frac{r}{2\pi}{(-\epsilon')}}\big)\Big)\\
=&\kappa_r\bigg( \sum_{\epsilon_I\in\{1,-1\}^{|I|}}\frac{C^{\epsilon_I}(\mathbf z^{\epsilon_I})}{\sqrt{-\det\mathrm{Hess}\Big(\frac{\mathcal W^{\epsilon_I}(\mathbf z^{\epsilon_I})}{4\pi\sqrt{-1}}\Big)}}\bigg)  e^{\frac{r}{4\pi\sqrt{-1}}\Big(2c\pi^2+\sqrt{-1}\big(\mathrm{Vol}(M_{L_{\boldsymbol\theta}})+\sqrt{-1}\mathrm{CS}(M_{L_{\boldsymbol\theta}})\big)\Big)}\Big(1 + O \Big( \frac{1}{r} \Big) \Big).
\end{split}
\end{equation*}
By Proposition \ref{computation}, we have
$$\lim_{r\to\infty} \frac{4\pi}{r}\log\kappa_r=\Big(\sigma(L'\cap L_I)-\sum_{i\in I}q_i-\sum_{i=1}^np_i-2|I|\Big)\sqrt{-1}\pi^2;$$
and by Corollary \ref{5.8}, we have
 $$\sum_{\epsilon_I\in\{1,-1\}^{|I|}}\frac{C^{\epsilon_I}(\mathbf z^{\epsilon_I})}{\sqrt{-\det\mathrm{Hess}\Big(\frac{\mathcal W^{\epsilon_I}(\mathbf z^{\epsilon_I})}{4\pi\sqrt{-1}}\Big)}}\neq 0,$$
and hence
 $$\lim_{r\to\infty} \frac{4\pi}{r}\log\Big(\sum_{\epsilon_I\in\{1,-1\}^{|I|}}\frac{C^{\epsilon_I}(\mathbf z^{\epsilon_I})}{\sqrt{-\det\mathrm{Hess}\Big(\frac{\mathcal W^{\epsilon_I}(\mathbf z^{\epsilon_I})}{4\pi\sqrt{-1}}\Big)}}\Big)=0.$$
 Therefore,
\begin{equation*}
\begin{split}
&\lim_{r\to\infty} \frac{4\pi}{r}\log\mathrm{RT}_r(M,L,(\mathbf n_I,\mathbf m_J))\\
=&\Big(\sigma(L'\cap L_I)-\sum_{i\in I}q_i-\sum_{i=1}^np_i-2|I|\Big)\sqrt{-1}\pi^2-2c\sqrt{-1}\pi^2+\mathrm{Vol}(M_{L_{\boldsymbol\theta}})+\sqrt{-1}\mathrm{CS}(M_{L_{\boldsymbol\theta}})\\
=&\mathrm{Vol}(M_{L_{\boldsymbol\theta}})+\sqrt{-1}\mathrm{CS}(M_{L_{\boldsymbol\theta}})\quad\quad\text{mod }\sqrt{-1}\pi^2\mathbb Z,
\end{split}
\end{equation*}
which completes the proof.
\end{proof}



\section{Some concrete examples}

The goal of this section is to show that several important families of links have complements homeomorphic to fundamental shadow link complements $M_c\setminus L_{\text{FSL}},$ and are obtained from $(M_c,L_{\text{FSL}})$ by doing a change-of-pair operation, including the twisted octahedral fully augmented links considered by Purcell\,\cite{P} and van der Veen\,\cite{V}, and the family $\mathcal U$ considered by Kumar\,\cite[Theorem 4.1]{Ku}. As a consequence, Conjecture \ref{VC},  the Volume Conjecture for the Turaev-Viro invariants\,\cite{CY} and the Generalized Volume Conjecture\,\cite{MY, G} hold, and the answer to \cite[Question 1.7]{DKY} is positive, for these families of examples. See Theorems  \ref{Cor1}, \ref{Cor} and \ref{Cor2}. It is worth mentioning that both the family of the twisted octahedral fully augmented links and the family $\mathcal U$ are universal families in the sense that every link in $S^3$ is a sublink of a member of these families.

We first look at the twisted octahedral fully augmented links. Following the construction in \cite{V}, we start with a trivalent graph $T$ that is homomorphic to the $1$-skeleton of a Euclidean tetrahedron as shown on the left of Figure \ref{TM}, and apply a sequence of the triangle move as show on the right of Figure \ref{TM} to get a trivalent graph $G.$

\begin{figure}[htbp]
\centering
\includegraphics[scale=0.3]{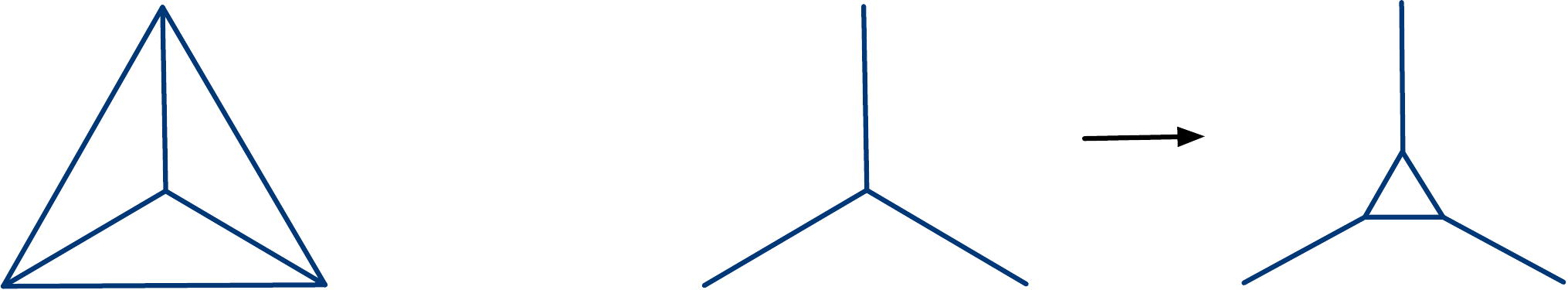}
\caption{ }
\label{TM}
\end{figure}

Then we color some edges of $G$ by red in a way that each vertex is adjacent to exactly one red edge. Such a coloring of the red edges always exists because for the initial graph $T$ we can color any pair of the opposite edges by red, and then for each triangle move, we color the new edge opposite to the red edge in the old graph by red. See Figure \ref{red}.

\begin{figure}[htbp]
\centering
\includegraphics[scale=0.3]{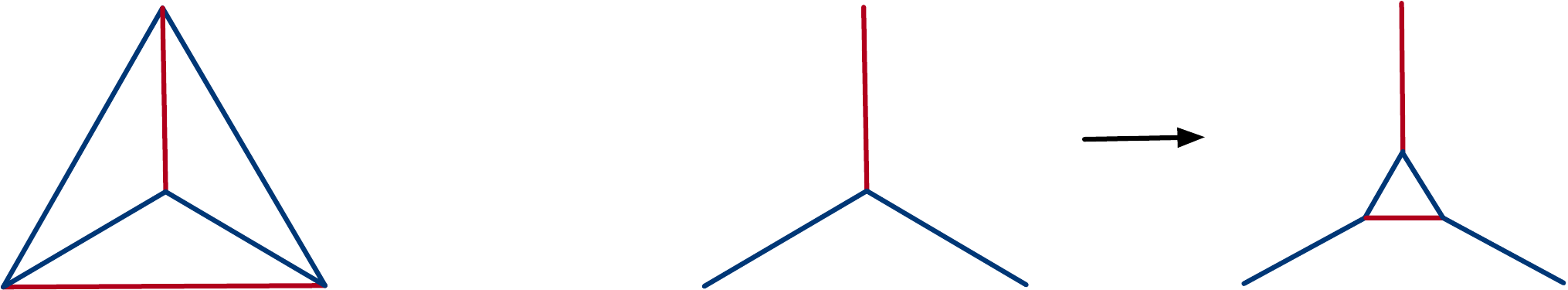}
\caption{ }
\label{red}
\end{figure}

For each choice of the red edges of $G,$ we can construct a link in $S^3$ as follows. As shown in Figure \ref{Change}, we first circulate each red edge by a trivial loop (the belt). 
\begin{figure}[htbp]
\centering
\includegraphics[scale=0.3]{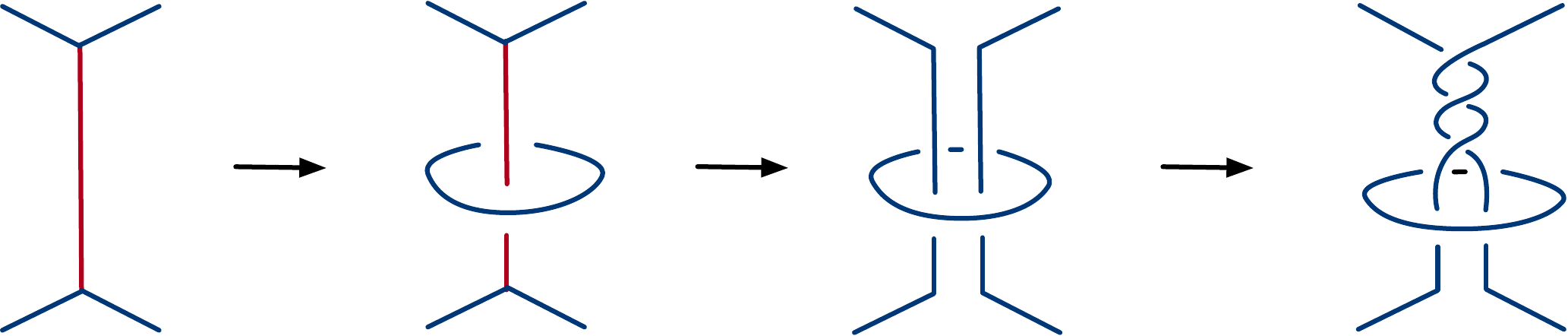}
\caption{ }
\label{Change}
\end{figure}
Then we replace each red edge by a pair of arcs parallel to it and connect the ends of the two arcs respectively to the edges of $G$ that are originally adjacent to the red edge. Finally, we do a certain number of (possibly half) twists to each pair of parallel arcs. In this way, we get a twisted octahedral fully augmented link in $S^3.$ See Figure \ref{Eg} for a concrete example starting from the trivalent graph $T.$

\begin{figure}[htbp]
\centering
\includegraphics[scale=0.3]{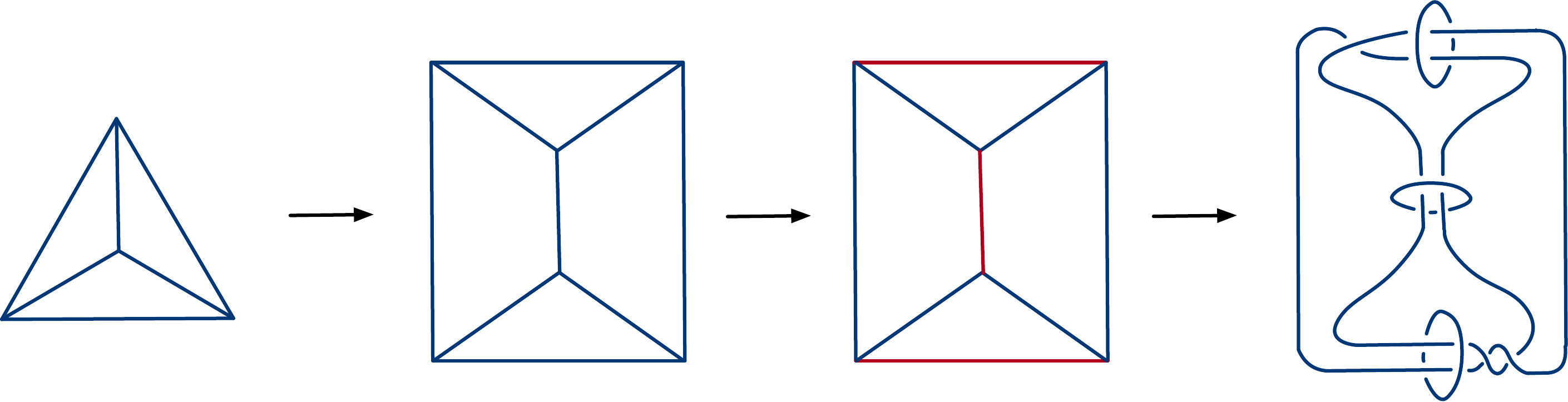}
\caption{ }
\label{Eg}
\end{figure}

The twisted octahedral fully augmented links were also described in \cite{P} using the dual nerve of the graph. Namely, we start with the graph $T$ and consider its dual nerve, which is a graph $T^*$ homeomorphic to $T$ with triangular faces. After doing a sequence of the central subdivisions of the faces, we get a graph $G^*$ with triangular faces. Then we choose a collection of red edges such that each triangular face contains exactly one red edge. Finally we consider the dual graph $G$ of $G^*$ and color the dual edge of the red edges of $G^*$ by red, and change the red  as shown in Figure \ref{Change}. In this way, we obtain a link in $S^3.$ Since the triangle move is dual to the central subdivision (see Figure \ref{dual}), the links constructed in this way are exactly the twisted octahedral fully augmented links.

\begin{figure}[htbp]
\centering
\includegraphics[scale=0.3]{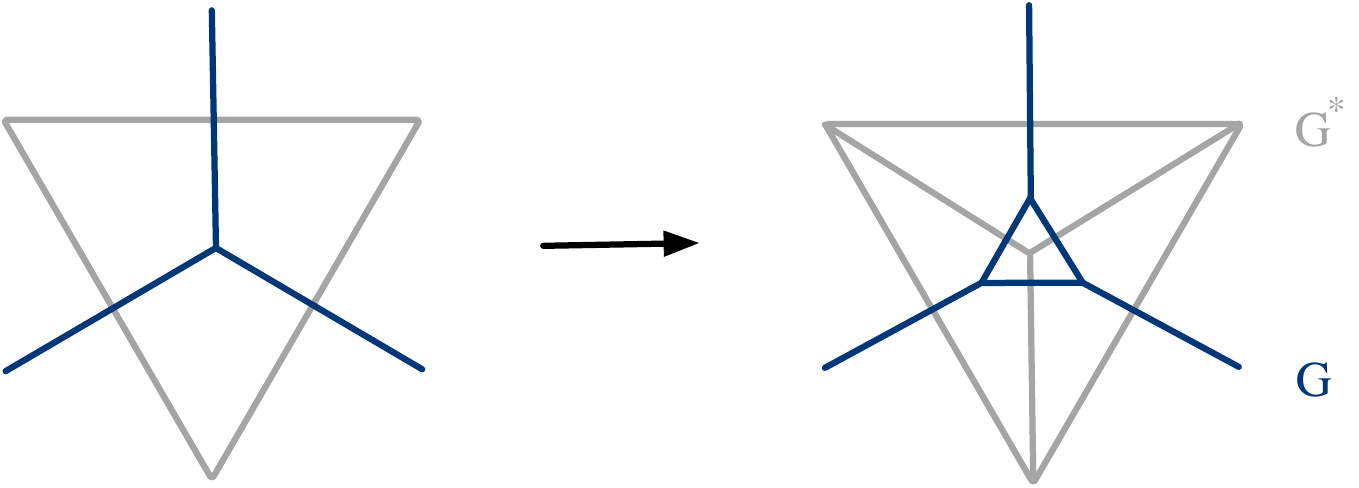}
\caption{ }
\label{dual}
\end{figure}

In \cite{V}, van der Veen proved that the complement of a twisted octahedral fully augmented link is homeomorphic to some fundamental shadow link complement. Together with the result of \cite{BDKY}, we have 

\begin{theorem}[\cite{V,BDKY}]\label{Cor1} Suppose $L$ is a framed twisted octahedral fully augmented link  in $S^3.$ Then as $r$ varies over all odd integers,
$$\lim_{r\to\infty}\frac{2\pi}{r}\log \mathrm{TV}_r{(S^3\setminus L)}=\mathrm{Vol}(S^3\setminus L).$$
\end{theorem}
\bigskip

The main observation of this section is the following Proposition \ref{TOFAL}, which is a refinement of the result of \cite{V}.

\begin{proposition}\label{TOFAL} Let $L$ be a framed twisted octahedral fully augmented link. Then $(S^3,L)$ is obtained from $(M_c,L_{\text{FSL}})$ by doing a change-of-pair operation, where $M_c=\#^{c+1}S^2\times S^1$ for some positive integer $c$ and  $L_{\text{FSL}}\subset M_c$ is a fundamental shadow link. As a consequence of Proposition \ref{equal}, 
$S^3\setminus L$ is homeomorphic to $M_c\setminus L_{\text{FSL}}.$
\end{proposition}

\begin{proof} 
 Suppose $L=L_1\cup \dots \cup L_n.$ Let $I$ be the subset of $\{1,\dots,n\}$ such that $\{L_i\}_{i\in I}$ are the belt components of $L,$ and let $L_I=\cup_{i\in I}L_i.$ Let $J=\{1,\dots,n\}\setminus I,$ and let $L_J=\cup_{j\in J}L_j.$ Recall that each $L_i$ corresponds to an red edge of the graph $G$ which splits into two parallel arcs with a number of twits. 
Later we will prove the result in two steps. In Step 1 we consider the case that there is no twist or only a half-twists for each $i\in I.$ In Step 2 we consider the general case.
 
 We first observe that $L_J$ lies in a tubular neighborhood of the $1$-skeleton of a truncated polyhedron $P_G$ in $S^3$ obtained by gluing truncated tetrahedra together along the triangles of truncation. Indeed, we can regard the initial trivalent graph $T$ as the set of edges of a truncated tetrahedron (see Figure \ref{Tetra}), and regard a triangular move as attaching another truncated tetrahedron along the triangle of truncation (see Figure \ref{Tetra2}).
 
\begin{figure}[htbp]
\centering
\includegraphics[scale=0.15]{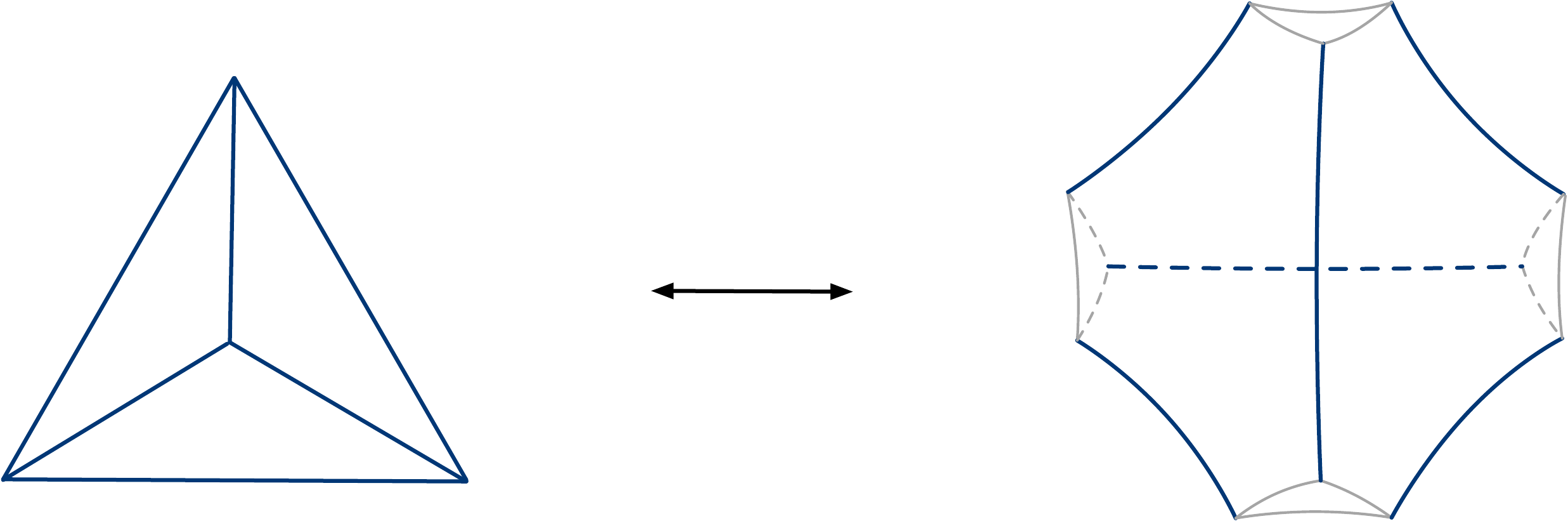}
\caption{In a truncated polyhedron, we only consider the intersection of two faces as an edge, and do not consider the intersection of a face and a triangle of truncation as an edge.}
\label{Tetra}
\end{figure}

\begin{figure}[htbp]
\centering
\includegraphics[scale=0.25]{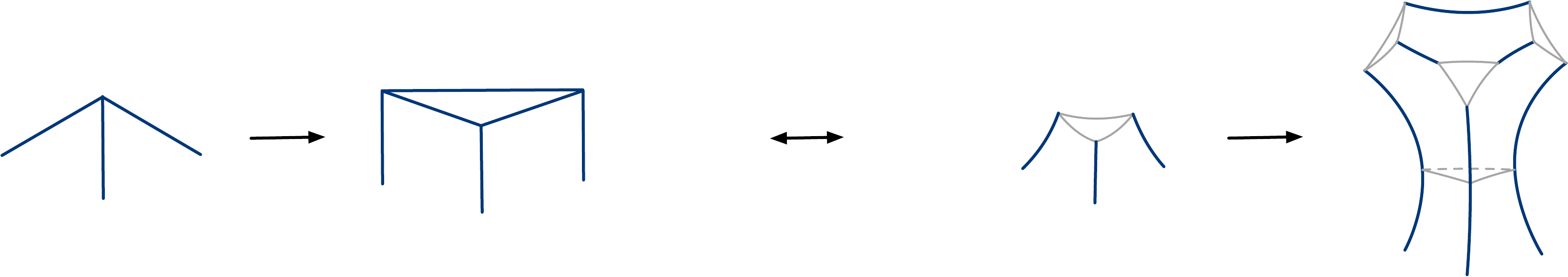}
\caption{ }
\label{Tetra2}
\end{figure}

In this way, we obtain a truncated polyhedron $P_G$ in $\mathbb R^3\subset S^3,$ and the edges of the graph $G$ correspond to the edges of $P_G.$  As shown in Figure \ref{split}, the two parallel arcs from splitting the red edge can be drawn in a tubular neighborhood of the red arc. 
\begin{figure}[htbp]
\centering
\includegraphics[scale=0.25]{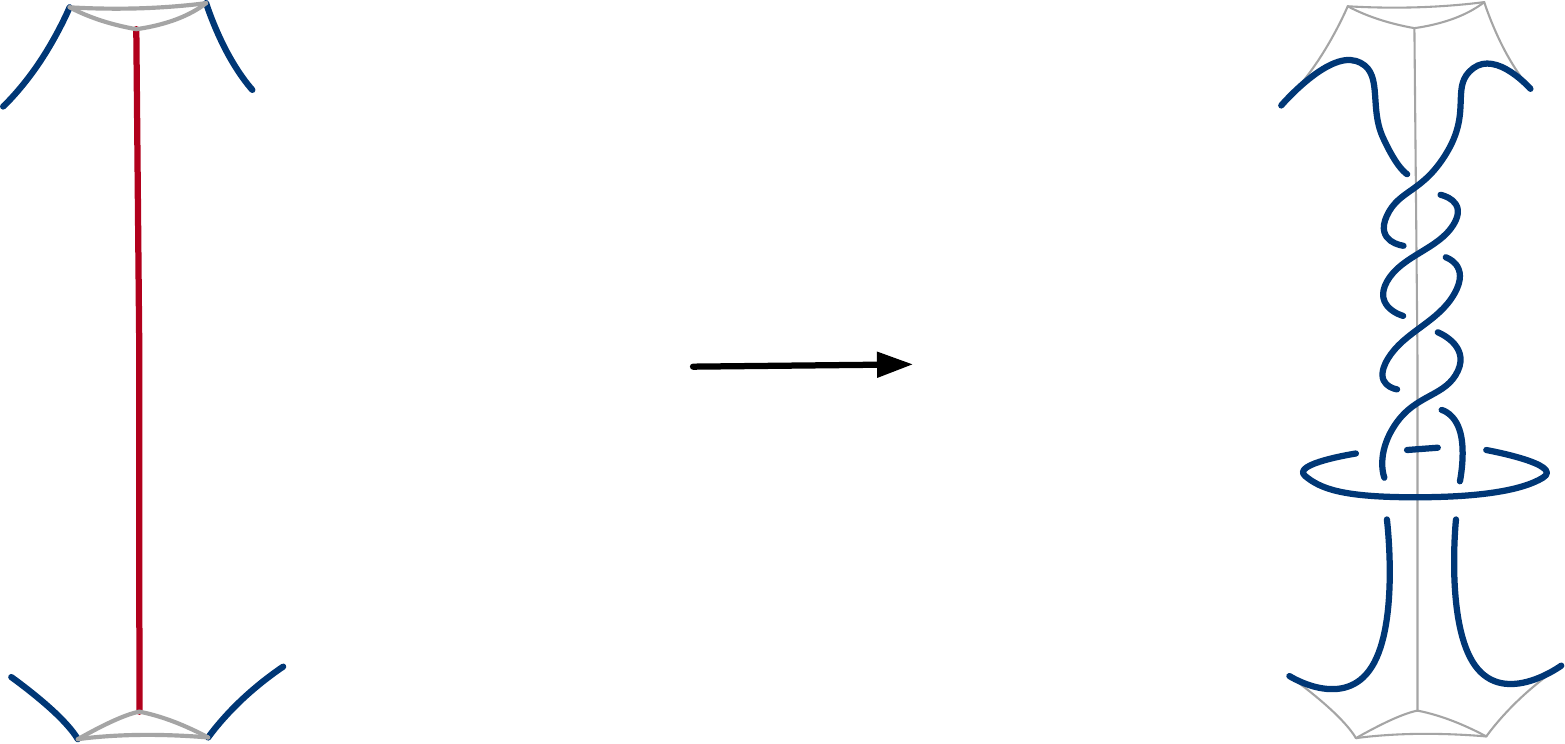}
\caption{ }
\label{split}
\end{figure}
See Figure \ref{Example} for a concrete example.

\begin{figure}[htbp]
\centering
\includegraphics[scale=0.25]{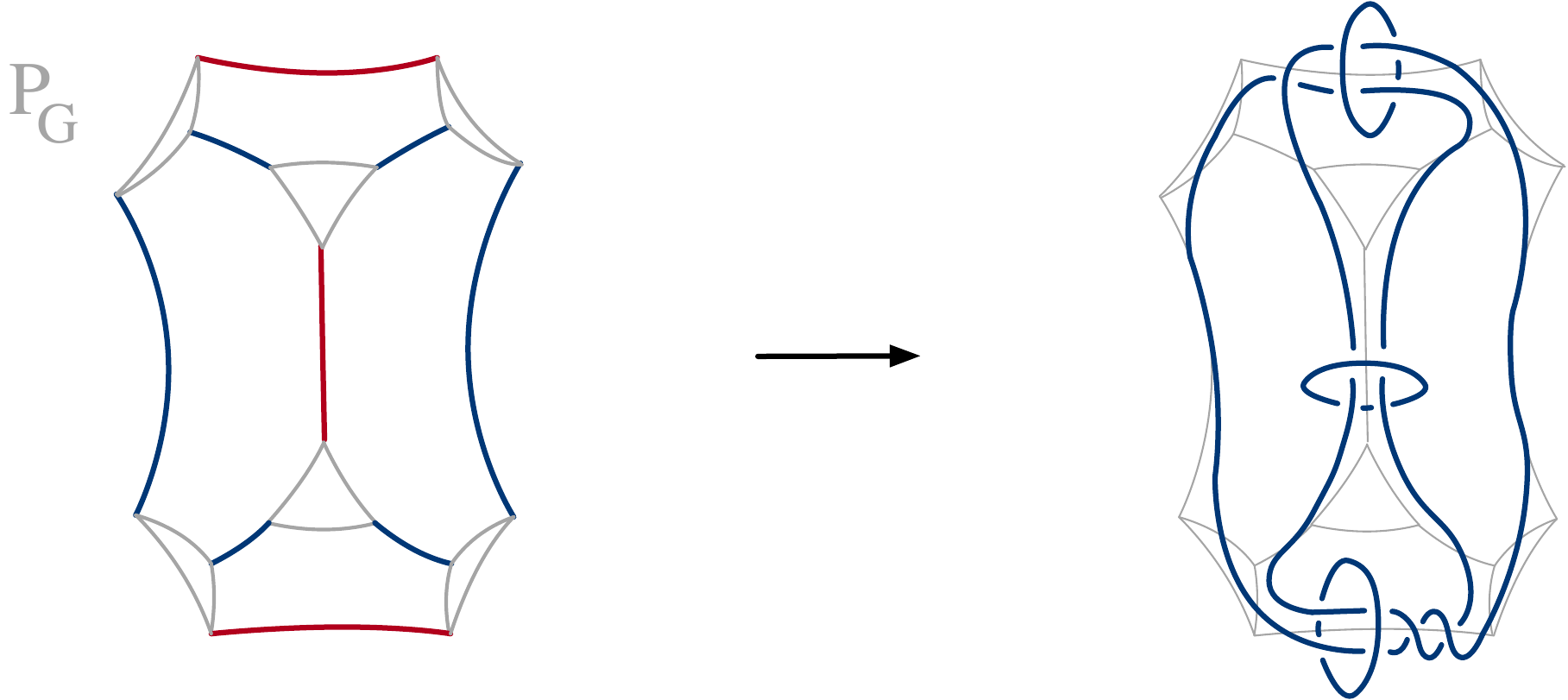}
\caption{ }
\label{Example}
\end{figure}

{\bf Step 1.} Suppose in the construction of $L$ there is no twist or only a half twist for each pair of arcs from splitting the red edges.  For each $i\in I,$ we let $L'_i$ be $L_i$ with the $0$-framing which is possibly different from the original framing of $L_i,$ and let $L_i^*$ be the framed trivial loop around $L_i$ with the $0$-framing. Let $L'_I=\cup_{i\in I} L'_i.$ We first claim that by doing the change-of-pair operation $T(L'_I;L_I^*)$ to the pair $(S^3,L)$ we get the pair $(M_{|I|},L_{\text{FSL}}),$ where $M_{|I|}=\#^{|I|}(S^2\times S^1)$ and $L_{\text{FSL}}$ is a fundamental shadow link in $M_{|I|}.$ See Figure \ref{T1}.
\begin{figure}[htbp]
\centering
\includegraphics[scale=0.35]{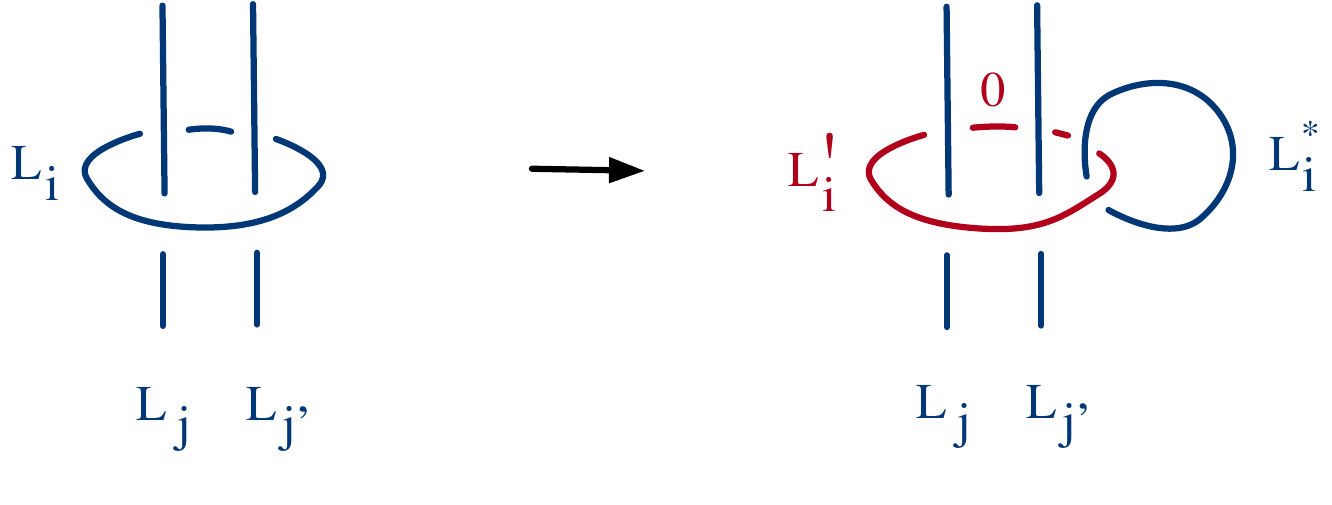}
\caption{ }
\label{T1}
\end{figure}

Indeed, let $N_I=\cup_{i\in I}N(L_i)$ be the union of the tubular neighborhoods of the belts $\{L_i\}_{i\in I}.$ Then $S^3\setminus N_I$ is a handlebody  $H_{|I|}$ of genus $|I|,$  $L_I^*\cup L_J$ is a link in $H_{|I|},$ and doing $0$-surgeries along $\{L_i\}_{i\in I}$ is the same as taking the double of $H_{|I|},$ which is homeomorphic to $M_{|I|}.$ On the other hand, the handlebody $H_{|I|}$ can also be considered as obtained from the polyhedron $P_G$ by gluing the two triangles of truncation at the end of each red edge via the orientation reversing affine homeomorphism identifying the two end points of the red edge.

 If there is no twist for all pairs of arcs from splitting the red edges, then the link $L_I^*\cup L_J$ corresponds to the link $L_{\text{FSL}}$ consisting of the union of the edges of $P_G.$ From the construction in Section \ref{fsl}, $L_{\text{FSL}}$ in the double of $H_{|I|}$ is a fundamental shadow link in $M_{|I|}.$ See Figure \ref{HB} for a concrete example.
\begin{figure}[htbp]
\centering
\includegraphics[scale=0.25]{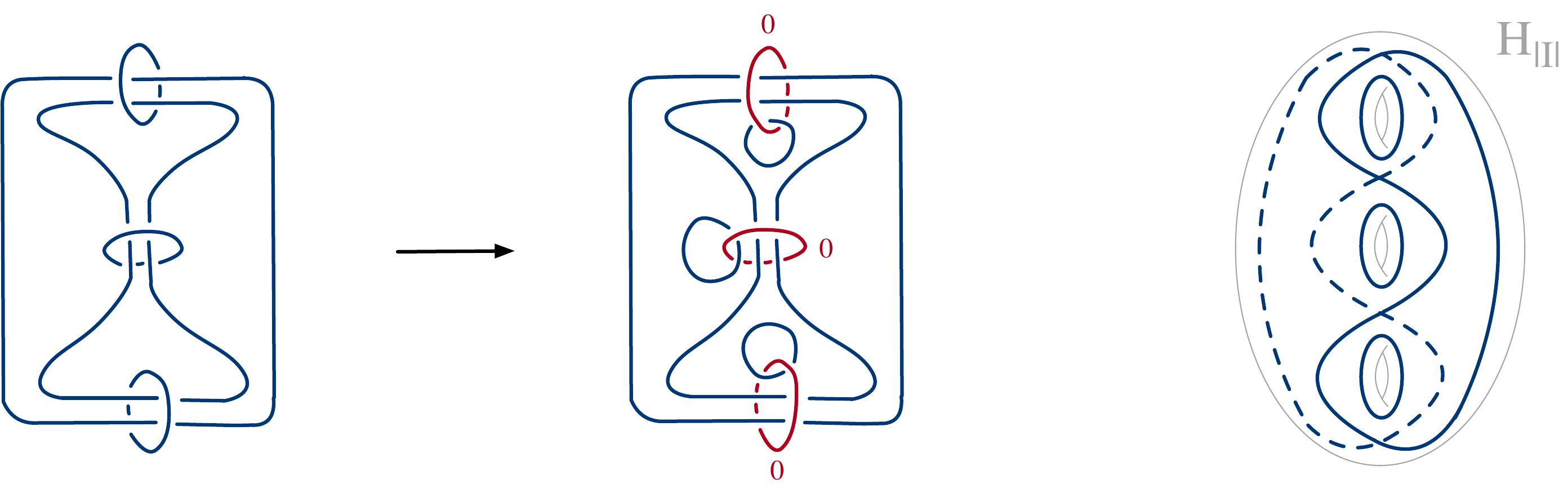}
\caption{ }
\label{HB}
\end{figure}

If there is a half twist on the pair of arcs from splitting red edge $e_i$ circulated by  $L_i,$ for some $i\in I,$ then we glue the two triangles of truncation at the end of $e_i$  via the orientation preserving affine homeomorphism identifying the two end points $e_i.$ In this way, we obtain a non-orientable handlebody $H'_{|I|}$ whose orientable double is homeomorphic to $M_{|I|},$ and the link $L_I^*\cup L_J$ corresponds to the link $L_{\text{FSL}}$ consisting of the union of the edges of $P_G,$ which is a fundamental shadow link in $M_{|I|}.$ 
\\

For each $i\in I,$ we let $L^{**}_i$ be the framed trivial loop around $L^*_i$ with the same framing as that of $L_i.$ Then we claim that $(S^3,L)$ can be obtained from $(M_{|I|},L_{\text{FSL}})$ by doing the change-of-pair operation $T(L^*_I,L^{**}_I),$ proving the result in Case 1. See Figure \ref{T2}.
\begin{figure}[htbp]
\centering
\includegraphics[scale=0.35]{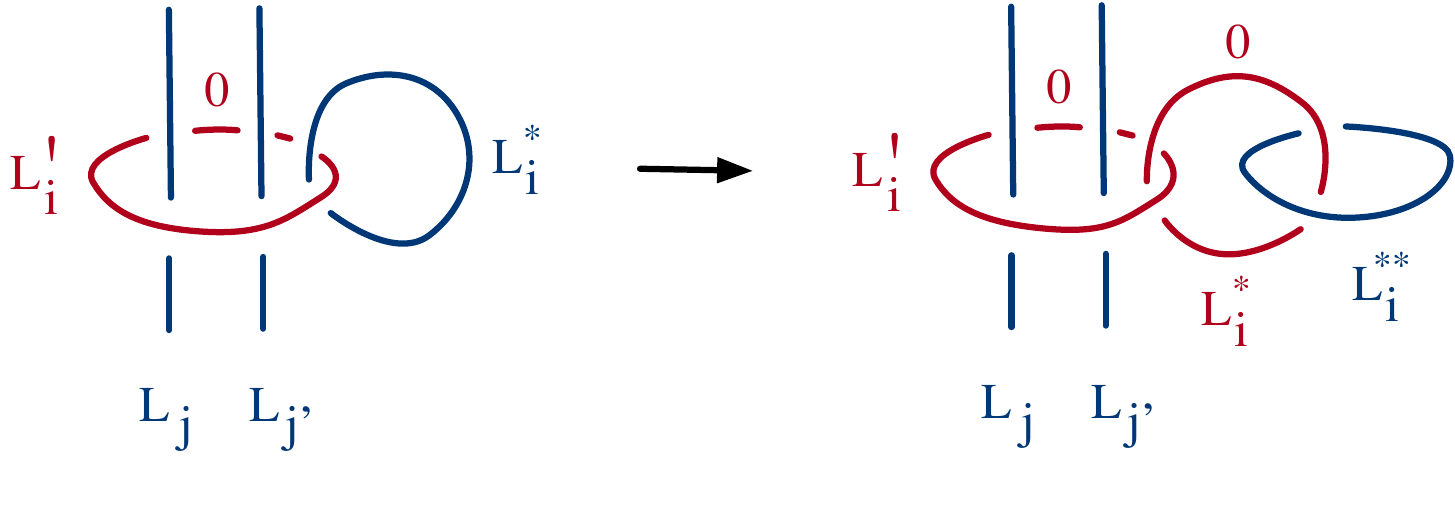}
\caption{ }
\label{T2}
\end{figure}

This could be seen by the following the second Kirby Moves (handle slides) which do not change the pair. For each $i\in I,$ we first slide $L^{**}_i$ over $L'_i$ as in Figure \ref{HS1} to get a framed trivial loop isotopic to $L_i.$ 
\begin{figure}[htbp]
\centering
\includegraphics[scale=0.35]{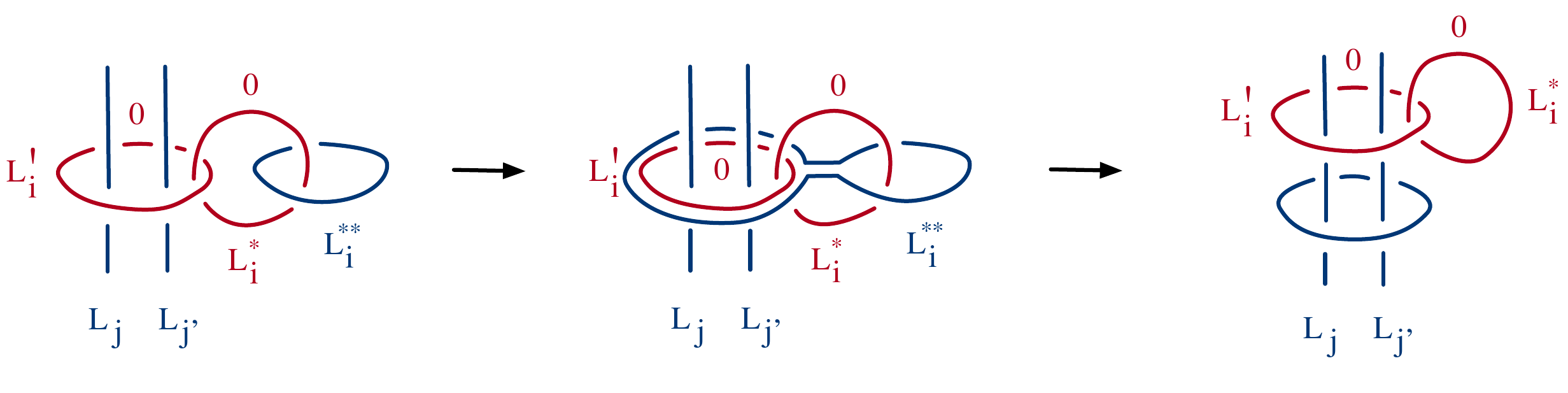}
\caption{ }
\label{HS1}
\end{figure}
Suppose $L_j$ and $L_{j'}$ (with $j$ possibly the same as $j'$) are the components of $L$ circulated by $L_i.$  Then we slide $L_j$ and $L_{j'}$ over $L_i^*$ as in Figure \ref{HS2} so that $L'_i\cup L^*_i$ is a Hopf link with $0$-framings unlinked with the rest of $L.$ Doing these operations for each $i\in I,$ we get the original link $L$ in the $3$-manifold obtained from $S^3$ by doing a surgery along the disjoint union of $|I|$ Hopf links with $0$-framings, which is still $S^3.$
\begin{figure}[htbp]
\centering
\includegraphics[scale=0.35]{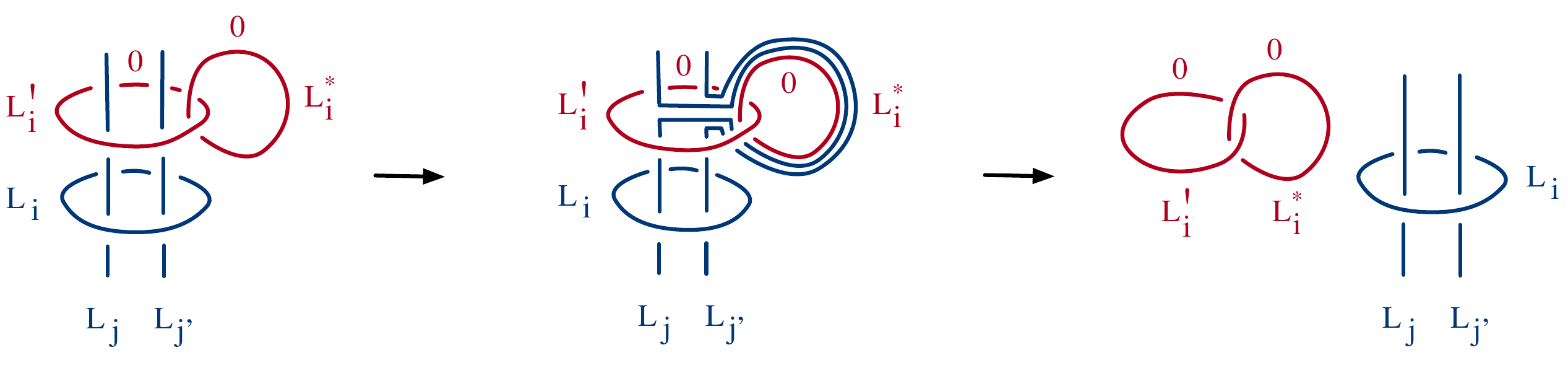}
\caption{ }
\label{HS2}
\end{figure}

{\bf Step 2.} Suppose in the construction of $L,$ the pair of arcs from splitting the red edge circulated by $L_i$ is twisted $p_i$ times or $p_i$ and a half times. Let $(M_{|I|},L_{\text{FSL}})$ be the fundamental shadow link constructed in Step 1. For each $i\in I,$ let ${L^*_i}'$ be $L^*_i$ with the $(-p_i)$-framing, and let ${L^*_I}'=\cup_{i\in I}{L^*_i}'.$ Still let $L^{**}_i$ be the framed trivial loop around $L^*_i$ with the same framing as that of $L_i.$ We claim that $(S^3,L)$ can be obtained from $(M_{|I|},L_{\text{FSL}})$ by doing the change-of-pair operation $T({L^*_I}',L^{**}_I),$ proving the result in Case 2. 

Similar to Case 1, we first slide $L^{**}_i$ over $L'_i$ as in Figure \ref{HS3} to get a framed trivial loop isotopic to $L_i.$ 
\begin{figure}[htbp]
\centering
\includegraphics[scale=0.35]{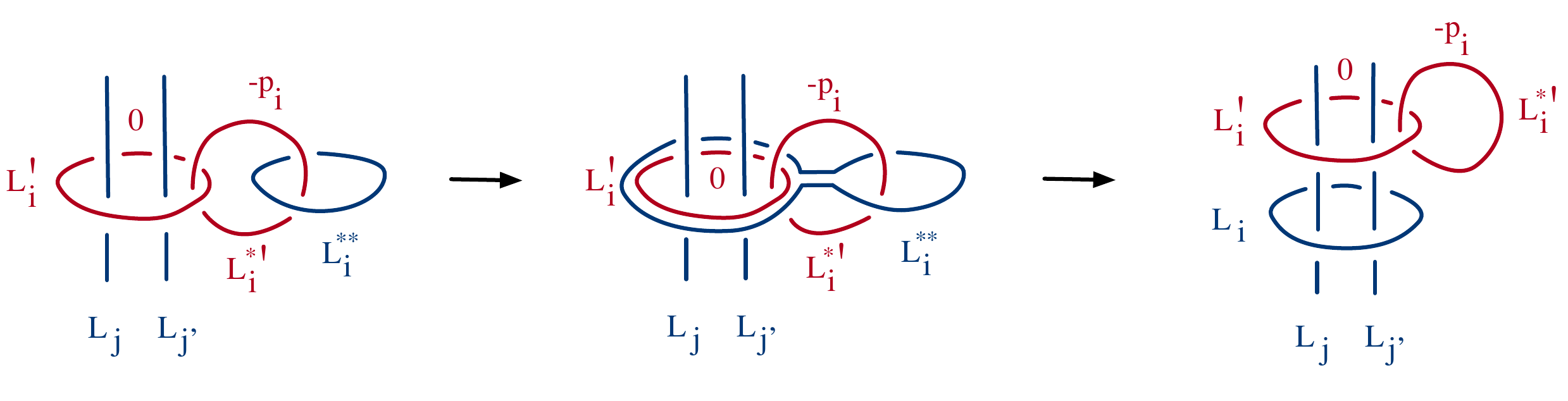}
\caption{ }
\label{HS3}
\end{figure}
By e.g. \cite[p. 273]{R}, since ${L^*_i}'$ is a trivial loop around $L_i'$ and
$$\frac{1}{p_i}=0-\frac{1}{-p_i},$$ 
doing the surgery along the union of $L'_i$ and ${L^*_i}'$ respectively with framings $0$ and $(-p_i)$ is equivalent to doing a $\frac{1}{p_i}$-surgery along a loop isotopic to $L_i,$ which is the same as doing $p_i$ full twists along the strip circulated by $L_i.$ See Figure \ref{HS4}. This completes the proof.
\begin{figure}[htbp]
\centering
\includegraphics[scale=0.35]{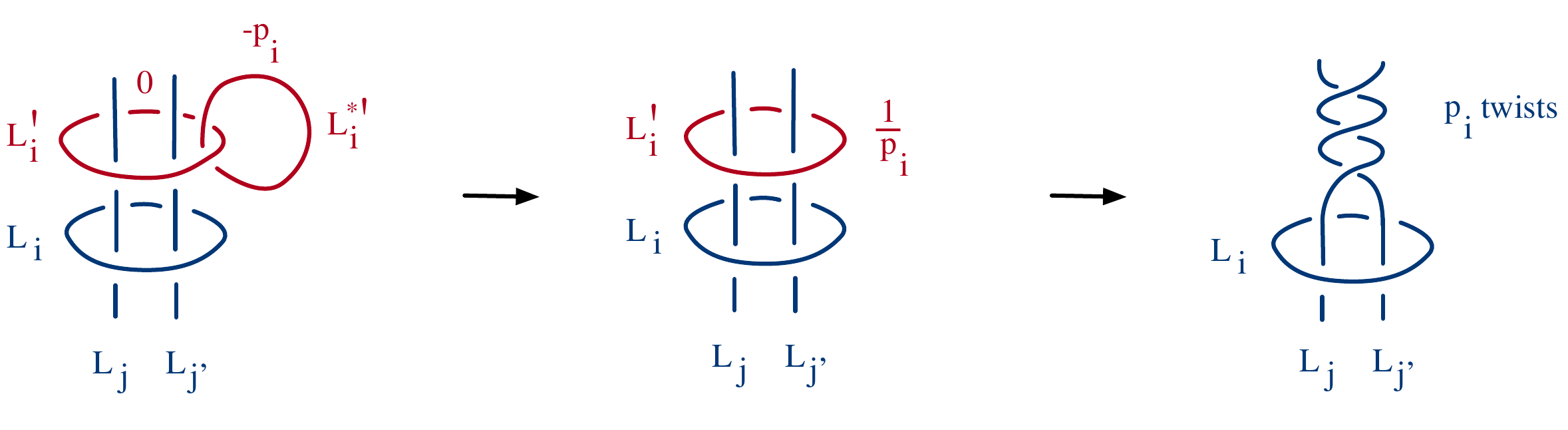}
\caption{ }
\label{HS4}
\end{figure}
\end{proof}

By Theorem \ref{main2}, Propositions \ref{TOFAL} and the relationship between the Resherikhin-Turaev invarints and the colored Jones polynomials that 
$$\mathrm{RT}_r(S^3, L, \mathbf m)=\mu_r\mathrm J_{\mathbf m, L}(q^2),$$
 we have

\begin{theorem}\label{Cor} Suppose $L$ is a framed twisted octahedral fully augmented  with $n$ components. For a sequence $\mathbf m^{(r)}=(m_1^{(r)},\dots,m_n^{(r)})$ of colorings of $L,$ let 
$$\theta_k=\Big|2\pi-\lim_{r\to\infty}\frac{4\pi m_k^{(r)}}{r}\Big|,$$
and let $\boldsymbol\theta=(\theta_1,\dots,\theta_n).$ If all the $\theta_i$'s are sufficiently small, then 
$$\lim_{r\to\infty}\frac{4\pi}{r}\log \mathrm{J}_{L,\mathbf m^{(r)}}\big(e^{\frac{4\pi i}{r}}\big)=\mathrm{Vol}(S^3_{L_{\boldsymbol\theta}})+\sqrt{-1}\mathrm{CS}(S^3_{L_{\boldsymbol\theta}})\quad\quad\text{mod } \sqrt{-1}\pi^2\mathbb Z,$$
where $r$ varies over all odd integers.
\end{theorem}

Next, we consider the universal family $\mathcal U$ of links in $S^3$ considered by Kumar\,\cite{Ku} whose complements are homeomorphic to some fundamental shadow link complements. Together with the result of \cite{BDKY}, he proved that Chen-Yang's Volume Conjecture for the Turaev-Viro invariants is true for the complements of these families of links. The way he found the fundamental shadow links is essentially by doing a change-of-pair operation along the belt components of the links. Then by the same argument in Step 1 of the proof of Proposition \ref{TOFAL}, we have

\begin{proposition}\label{Kumar}
Let $L$ be a framed link in $\mathcal U$ of \cite{Ku}. Then $(S^3,L)$ is obtained from $(M_c,L_{\text{FSL}})$ by doing a change-of-pair operation. 
\end{proposition}

As a consequence of Theorem \ref{main2} and Proposition \ref{Kumar}, we have

\begin{theorem}\label{Cor2} Suppose $L$ is a framed link of $n$ components which is a member of $\mathcal U$ of \cite{Ku}. For a sequence $\mathbf m^{(r)}=(m_1^{(r)},\dots,m_n^{(r)})$ of colorings of $L,$ let 
$$\theta_k=\Big|2\pi-\lim_{r\to\infty}\frac{4\pi m_k^{(r)}}{r}\Big|,$$
 and let $\boldsymbol\theta=(\theta_1,\dots,\theta_n).$ If all the $\theta_i$'s are sufficiently small, then 
$$\lim_{r\to\infty}\frac{4\pi}{r}\log \mathrm{J}_{L,\mathbf m^{(r)}}\big(e^{\frac{4\pi i}{r}}\big)=\mathrm{Vol}(S^3_{L_{\boldsymbol\theta}})+\sqrt{-1}\mathrm{CS}(S^3_{L_{\boldsymbol\theta}})\quad\quad\text{mod } \sqrt{-1}\pi^2\mathbb Z,$$
where $r$ varies over all odd integers.
\end{theorem}

\begin{remark} The proof of Proposition \ref{TOFAL} essentially provides an algorithm of constructing a fundamental shadow link from a given twisted octahedral fully augmented link or a given element of $\mathcal U.$  In \cite{V}, van der Veen for each link $L$ in $S^3$ provided an algorithm of constructing a twisted octahedral fully augmented link $L'$ such that $S^3\setminus L$ is obtained from $S^3\setminus L'$ by $\frac{1}{0}$-filling suitable boundary components. In \cite{Ku}, Kumar for each link $L$ in $S^3$ considered as the closure of a braid provided an algorithm of constructing a link $L'$ in $\mathcal U$ such that $S^3\setminus L$ is obtained from $S^3\setminus L'$ by $\frac{1}{0}$-filling suitable boundary components. Therefore, together with Propositions \ref{TOFAL} and \ref{Kumar}, we for each $L$ in $S^3$ have two algorithms of constructing a fundamental shadow link $L_{\text{FSL}}$ in $M_c$ such that $S^3\setminus L$ is obtained from $M_c\setminus L_{\text{FSL}}$ by filling suitable boundary components. It is an interesting question to know whether the universal families of the twisted octahedral fully augmented links and $\mathcal U$ are actually the same family. 
\end{remark}


\appendix
\section{A proof of Proposition \ref{saddle}}

The goal of this appendix is to prove Proposition \ref{saddle}. We need the following two Lemmas whose proofs are included at the end of the appendix. 
\begin{lemma}\label{estimate}
For any $\epsilon>0,$ there exists a $\delta > 0$ such that\\
(1) 
$$ \int_{-\epsilon}^\epsilon e^{-r z^2} dz = \sqrt{\frac{\pi}{r}} + O(e^{-\delta r}),$$

and\\
(2)
$$ \int_{-\epsilon}^\epsilon z^2 e^{-r z^2} dz = \frac{1}{2}\sqrt{\frac{\pi}{r^3}} + O(e^{-\delta r}).$$
\end{lemma}

\begin{lemma}\label{Taylor} Let $D_{\mathbf z}$ be a region in $\mathbb C^n$ containing the origin $\mathbf 0$ and let $g^{\mathbf a}$ a family of complex valued functions on $D_{\mathbf z}$ smoothly parametrized by $\mathbf a$ in a region $D_{\mathbf a}$ of $\mathbb R^k.$  Then there exist families of functions $h_1^{\mathbf a},\dots, h_n^{\mathbf a},$ and $k_1^{\mathbf a},\dots,k_n^{\mathbf a}$  such that 
\begin{enumerate}[(1)]
\item all of $h_i^{\mathbf a}$'s and $k_i^{\mathbf a}$'s are smoothly parametrized by $\mathbf a$ in $D_{\mathbf a},$
\item for each $\mathbf a\in D_{\mathbf a},$  $h_i^{\mathbf a}$ has variables $z_{i+1},\dots,z_n$ and is holomorphic in them, 
\item for each $\mathbf a\in D_{\mathbf a},$  $k_i^{\mathbf a}$ has variables  $z_i,\dots,z_n$ and is holomorphic in them, and 
\item $$g^{\mathbf a}( z_1,\dots, z_n)=g^{\mathbf a}(\mathbf 0)+\sum_{i=1}^nh_i^{\mathbf a}(z_{i+1},\dots, z_n)z_i+\sum_{i=1}^nk_i^{\mathbf a}(z_i,\dots,z_n)z_i^2.$$
\end{enumerate}

\end{lemma}

\begin{lemma}\label{OPFCML} (Complex Morse Lemma) Let $D_{\mathbf z}$ be a region in $\mathbb C^n,$ let $D_{\mathbf a}$ be a region in $\mathbb R^k,$ and let $f: D_{\mathbf z} \times D_{\mathbf a} \to \mathbb C$ be a complex valued function that is holomorphic in $\mathbf z\in D_{\mathbf z}$ and smooth in $\mathbf a\in D_{\mathbf a}.$ For $\mathbf a\in D_{\mathbf a},$ let $f^{\mathbf a}: D_{\mathbf z}\to\mathbb C$ be the function defined by $f^{\mathbf a}(\mathbf z)=f(\mathbf z,\mathbf a).$ Suppose for each $\mathbf a\in D_{\mathbf a},$ $f^{\mathbf a}$ has a non-degenerate critical point $c_{\mathbf a}$ which smoothly depends on $\mathbf a.$ Then for each $\mathbf a_0\in D_{\mathbf a},$ there exists an open set $V \subset \mathbb C^n$ containing $\mathbf 0,$ an open set $A \subset D_{\mathbf a}$ containing $\mathbf a_0,$ and a smooth function $\psi: V \times A \to   D_{\mathbf z}$ such that, if we denote $\psi^{\mathbf a}(\mathbf Z) = \psi(\mathbf Z, \mathbf a),$ then for each $\mathbf a\in D_{\mathbf a},$ $\mathbf z = \psi^{\mathbf a}(\mathbf Z)$ is a holomorphic change of variable on $V$ such that 
$$\psi^{\mathbf a}(\mathbf 0) = \mathbf c_{\mathbf a},$$
$$ f^{\mathbf a}(\psi^{\mathbf a}(\mathbf Z)) = f^{\mathbf a}(\mathbf c_{\mathbf a}) - Z_1^2 - \dots - Z_n^2,$$
and
$$ \det \Big( D (\psi^{\mathbf a})(\mathbf 0)  \Big)
= \frac{2^{\frac{n}{2}}}{\sqrt{- \det \mathrm{Hess}(f^{\mathbf a})(\mathbf c_{\mathbf a})}}.$$
\end{lemma}

\begin{proof}[Proof of Proposition \ref{saddle}] We write  $\mathbf z=(z_1,\dots, z_n) \in \mathbb C^n,$ $\mathbf Z = (Z_1,\dots, Z_n) \in \mathbb C^n,$ $\mathbf W = (W_1,\dots, W_n) \in \mathbb C^n,$ $d\mathbf z=dz_1\dots dz_n$ and $\mathbf 0=(0,\dots,0) \in \mathbb C^n.$

We first consider the special case $\mathbf c_r=\mathbf 0,$ $S_r=[-\epsilon,\epsilon]^n\subset \mathbb R^n\subset \mathbb C^n,$
and 
$$f^{\mathbf a_r}(\mathbf z)=-\sum_{i=1}^nz_i^2$$
for each $r.$ In this case, let
$$\sigma^{\mathbf a_r}_{ r}(\mathbf z)=\upsilon_{r}(\mathbf z, \mathbf a_r)\int_0^1e^{\frac{\upsilon_{r}(\mathbf z, \mathbf a_r)}{r}s}ds.$$
Then we can write 
$$e^{\frac{\upsilon_{r}(\mathbf z, \mathbf a_r)}{r}}=1+\frac{\sigma_{r}^{\mathbf a_r}(\mathbf z)}{r},$$
and 
\begin{equation}\label{A1}
g^{\mathbf a_r}(\mathbf z)e^{rf^{\mathbf a_r}_{r}(\mathbf z)}=g^{\mathbf a_r}(\mathbf z)e^{rf^{\mathbf a_r}(\mathbf z)}+\frac{1}{r}g^{\mathbf a_r}(\mathbf z)\sigma^{\mathbf a_r}_{r}(\mathbf z)e^{rf^{\mathbf a_r}(\mathbf z)}.
\end{equation}
Since $|\upsilon_{r}(\mathbf z, \mathbf a_r)|<M$ for some $M>0$ independent of $r,$ 
$$|\sigma^{\mathbf a_r}_{r}(\mathbf z)|<M\int_0^1e^{\frac{M}{r}s}ds=M\bigg(\frac{e^{\frac{M}{r}}-1}{\frac{M}{r}}\bigg)<2M.$$
Since $\{\mathbf a_r\}$ is convergent and $g$ is smooth in $\mathbf a,$ if $M$ is big enough, then $|g^{\mathbf a_r}(\mathbf z)|<M$ for all $z \in S_r=[-\epsilon,\epsilon]^n$ for $r$ large enough. By Lemma \ref{estimate} (1), we have
\begin{equation}\label{A2}
\begin{split}
\Big|\int_{S_r}\frac{1}{r}g^{\mathbf a_r} (\mathbf z)\sigma^{\mathbf a_r}_r(\mathbf z)e^{rf^{\mathbf a_r}(\mathbf z)}d\mathbf z\Big|&<\frac{2M^2}{r}\int_{S_r}e^{rf^{\mathbf a_r}(\mathbf z)}d\mathbf z\\
&=\frac{2M^2}{r}\Big(\frac{\pi}{r}\Big)^{\frac{n}{2}}+O(e^{-\delta r})=O\Big(\frac{1}{\sqrt{r^{n+2}}}\Big).
\end{split}
\end{equation}

By Lemma \ref{Taylor}, we have
$$g^{\mathbf a_r}( z_1,\dots, z_n)=g^{\mathbf a_r}(\mathbf 0)+\sum_{i=1}^nh_i^{\mathbf a_r}(z_{i+1},\dots, z_n)z_i+\sum_{i=1}^nk_i^{\mathbf a_r}(z_i,\dots,z_n)z_i^2$$
for some  holomorphic functions $\{h^{\mathbf a_r}_{i}\}$ and $\{k^{\mathbf a_r}_i\},$ $i\in\{1,\dots,n\}.$ 
Then by Lemma \ref{estimate} (1), we have
\begin{equation}\label{A3}
\int_{S_r} g^{\mathbf a_r}(\mathbf 0)e^{rf^{\mathbf a_r}(\mathbf z)}d\mathbf z=g^{\mathbf a_r}(\mathbf 0)\Big(\frac{\pi}{r}\Big)^{\frac{n}{2}}+O\Big(\frac{1}{r}\Big).
\end{equation}
Since each $z_ie^{-rz_i^2}$ is odd, we have
$$\int_{-\epsilon}^\epsilon z_ie^{-rz_i^2}dz_i=0.$$
As a consequence, for each $i,$ we have 
\begin{equation}\label{A4}
\begin{split}
&\int_{[-\epsilon,\epsilon]^n} h_i^{\mathbf a_r}(z_{i+1},\dots,z_n)z_ie^{rf^{\mathbf a_r}(\mathbf z)}d\mathbf z\\
=&\int_{[-\epsilon,\epsilon]^{n-1}}h_i^{\mathbf a_r}(z_{i+1},\dots,z_n)e^{-r\sum_{j\neq i}z_j^2}\prod_{j\neq i}dz_j\cdot \int_{-\epsilon}^\epsilon z_ie^{-rz_i^2}dz_i
=0.
\end{split}
\end{equation}
Since $\{\mathbf a_r\}$ is convergent and $k^{\mathbf a}_{i}$ is smooth in $\mathbf a,$  if $M$ is big enough, then for $r$ large enough, $|k^{\mathbf a_r}_i(\mathbf z)|<M$ for all $\mathbf z \in S_r,$ $i\in\{1,\dots, n\}.$ By Lemma \ref{estimate} we have for each $i\in\{1,\dots, n\}$ 
\begin{equation}\label{A6}
\begin{split}
\Big|\int_{S_r} k^{\mathbf a_r}_{i}(\mathbf z)z_i^2e^{rf^{\mathbf a_r}(\mathbf z)}d\mathbf z\Big|<M\Big(\int_{-\epsilon}^\epsilon z_i^2e^{-rz_i^2}dz_i\Big)\prod_{j\neq i}\Big( \int_{-\epsilon}^\epsilon e^{-rz_j^2}dz_j\Big) = O \Big(\frac{1}{\sqrt{r^{n+2}}}\Big).
\end{split}
\end{equation}

Putting (\ref{A3}), (\ref{A4}) and (\ref{A6}) together, we have the result for this special case.
\\

For the general case, by assumption (6) of Proposition \ref{saddle}, for $\mathbf a$ sufficiently close to $\mathbf a_0,$  $f^{\mathbf a}$ has a unique non-degenerate critical point $\mathbf c_{\mathbf a}$ in a sufficiently small neighborhood of $\mathbf c_0.$ Then we can apply Lemma \ref{OPFCML} to the function $f$ and $\mathbf a_0.$ Let $V, A$ and $\psi$ respectively be the two open sets and the change of variable function as described in Lemma \ref{OPFCML}. For $r$ sufficiently large, let 
$$U_r=\psi^{\mathbf a_r}\Big(\prod_{i=1}^n\big\{Z_i\in\mathbb C\ \big|\ -\epsilon<\mathrm{Re}(Z_i)<\epsilon, -\epsilon<\mathrm{Im}(Z_i)<\epsilon \big\}\Big)$$
for some sufficiently small $\epsilon >0 .$
Let $\mathrm{Vol}(S_r\setminus U_r)$ be the Euclidean volume of $S_r\setminus U_r.$ By the compactness of $S_r\setminus U_r$ and by assumptions (2), (4) and (5) of Proposition \ref{saddle}, there exist constants $M>0$ and $\delta>0$ independent of $r$ such that 
$|g^{\mathbf a_r}(\mathbf z)|<M,$  $\mathrm{Vol}(S_r\setminus U_r)<M$
and
\begin{equation}\label{A7'}
\mathrm{Re} f_r^{\mathbf a_r}(\mathbf z)<\mathrm{Re} f^{\mathbf a_r}(\mathbf c_r)-\delta
\end{equation}
on $S_r\setminus U$ for $r$ large enough. Then
\begin{equation}\label{A7}
\Big|\int_{S_r\setminus U}g^{\mathbf a_r}(\mathbf z)e^{rf_r^{\mathbf a_r}(\mathbf z)}d\mathbf z\Big|< M^2\Big(e^{r(\mathrm{Re}f^{\mathbf a_r}(\mathbf z)(\mathbf c_r)-\delta)}\Big) 
= O\Big(e^{r(\mathrm{Re}f^{\mathbf a_r}(\mathbf z)(\mathbf c_r)-\delta)}\Big) .
\end{equation}

In Figure \ref{sad} below, the shaded region is where $\mathrm{Re}(-\sum_{i=1}^nZ_i^2)<0.$ For each $\mathbf a_r,$ in $\overline{(\psi^{\mathbf a_r})^{-1}(U_r)}$ there is a homotopy $H_{r}$ from $\overline{(\psi^{\mathbf a_r})^{-1}(S_r\cap U_r)}$ to $[-\epsilon,\epsilon]^n\subset \mathbb R^n$ defined by ``pushing everything down'' to the real part.  This is where we use condition (3). Let $S_r'=H_r(\partial (\psi^{\mathbf a_r})^{-1}(S_r\cap U_r)\times [0,1]).$ Then $\overline{(\psi^{\mathbf a_r})^{-1}(S_r\cap U_r)}$ is homotopic to $S_r'\cup [-\epsilon,\epsilon]^n.$
\begin{figure}[htbp]
\centering
\includegraphics[scale=0.5]{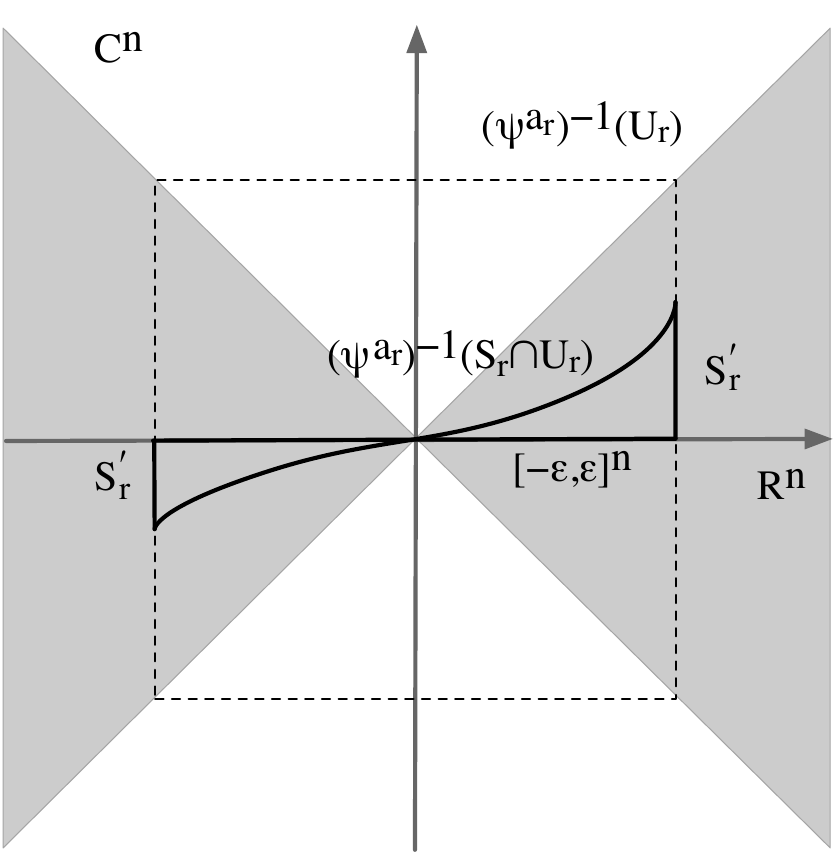}
\caption{}
\label{sad}
\end{figure}
Then by analyticity, 
\begin{align}\label{A8}
&\int_{S_r\cap U}g^{\mathbf a_r}(\mathbf z)e^{rf_r^{\mathbf a_r}(\mathbf z)}d\mathbf z \notag\\
=&\int_{(\psi^{\mathbf a_r})^{-1}(S_r\cap U)}g^{\mathbf a_r}(\psi^{\mathbf a_r}(\mathbf Z))\det \mathrm D(\psi^{\mathbf a_r}(\mathbf Z))e^{rf_r^{\mathbf a_r}(\psi^{\mathbf a_r}(\mathbf Z))}d\mathbf Z \notag\\
=&\int_{S_r'}g^{\mathbf a_r}(\psi^{\mathbf a_r}(\mathbf Z))\det \mathrm D(\psi^{\mathbf a_r}(\mathbf Z))e^{rf_r^{\mathbf a_r}(\psi^{\mathbf a_r}(\mathbf Z))}d\mathbf Z 
 +\int_{[-\epsilon,\epsilon]^n}g^{\mathbf a_r}(\psi(\mathbf Z))\det \mathrm D(\psi^{\mathbf a_r}(\mathbf Z))e^{rf_r^{\mathbf a_r}(\psi^{\mathbf a_r}(\mathbf Z))}d\mathbf Z.
\end{align}
Since $\psi^{\mathbf a_r}(S')\subset S_r\setminus U_r,$ by (\ref{A7'})
\begin{equation}\label{A9}
\int_{S_r'}g^{\mathbf a_r}(\psi^{\mathbf a_r}(\mathbf Z))\det \mathrm D(\psi^{\mathbf a_r}(\mathbf Z))e^{rf_r^{\mathbf a_r}(\psi^{\mathbf a_r}(\mathbf Z))}d\mathbf Z
= \int_{\psi^{\mathbf a_r}(S_r')} g^{\mathbf a_r}(\mathbf z)e^{rf_r^{\mathbf a_r}(\mathbf z)} d\mathbf z
= O\Big(e^{r(\mathrm{Re}f^{\mathbf a_r}(\mathbf c_r)-\delta)}\Big);
\end{equation}
and by the special case
\begin{equation*}
\begin{split}
&\int_{[-\epsilon,\epsilon]^n}g^{\mathbf a_r}(\psi^{\mathbf a_r}(\mathbf Z))\det \mathrm D(\psi^{\mathbf a_r}(\mathbf Z))e^{rf_r^{\mathbf a_r}(\psi^{\mathbf a_r}(\mathbf Z))}d\mathbf Z\\
=&e^{rf^{\mathbf a_r}(\mathbf c_r)}\int_{[-\epsilon,\epsilon]^n}g^{\mathbf a_r}(\psi^{\mathbf a_r}(\mathbf Z))\det \mathrm D(\psi^{\mathbf a_r}(\mathbf Z))e^{r\big(-\sum_{i=1}^nZ_i^2 + \frac{\upsilon_{ r}(\mathbf z,\mathbf a_r)}{r^2} \big)}d\mathbf Z\\
=&e^{rf^{\mathbf a_r}(\mathbf c_r)}g^{\mathbf a_r}(\mathbf \psi^{\mathbf a_r}(\mathbf 0))\det\mathrm D(\psi^{\mathbf a_r}(\mathbf 0))\Big(\frac{\pi}{r}\Big)^{\frac{n}{2}}\Big(1+O\Big(\frac{1}{r}\Big)\Big)\\
=&\Big(\frac{2\pi}{r}\Big)^{\frac{n}{2}}\frac{g^{\mathbf a_r}(\mathbf c_r)}{\sqrt{-\det\mathrm{Hess}(f^{\mathbf a_r})(\mathbf c_r)}}  e^{rf^{\mathbf a_r}(\mathbf c_r)}\Big( 1 + O \Big( \frac{1}{r} \Big) \Big).
\end{split}
\end{equation*}
Together with (\ref{A7}), (\ref{A8}) and (\ref{A9}), we have the result.
\end{proof}

\begin{proof}[Proof of Lemma \ref{estimate}] For (1), we  have
$$\int_{-\epsilon}^\epsilon e^{-rz^2}dz=\int_{-\infty}^\infty e^{-rz^2}dz-\int_{-\infty}^{-\epsilon} e^{-rz^2}dz-\int_{\epsilon}^\infty e^{-rz^2}dz,$$
where the first term 
$$\int_{-\infty}^\infty e^{-rz^2}dz=\sqrt{\frac{\pi}{r}}$$
is a Gaussian integral, and the other two terms
$$\int_{-\infty}^{-\epsilon} e^{-rz^2}dz=\int_{\epsilon}^\infty e^{-rz^2}dz\leqslant \int_{\epsilon}^\infty e^{-r\epsilon z}dz=\frac{e^{-r\epsilon^2}}{r\epsilon}=O(e^{-\delta r}).$$

For (2), by integration by parts, we have
$$\int_{-\epsilon}^\epsilon e^{-rz^2}dz=ze^{-rz^2}\Big|_{-\epsilon}^\epsilon+2r\int_{-\epsilon}^\epsilon z^2e^{-rz^2}dz,$$
hence by (1)
$$\int_{-\epsilon}^\epsilon z^2e^{-rz^2}dz=\frac{1}{2r}\Big(\int_{-\epsilon}^\epsilon e^{-rz^2}dz-2\epsilon e^{-r\epsilon^2}\Big)=\frac{1}{2}\sqrt{\frac{\pi}{r^3}}+O(e^{-\delta r}).$$
\end{proof}

\begin{proof}[Proof of Lemma \ref{Taylor}] We use induction on $n.$ For $n=1,$ if $z_1\neq 0,$ then we can write 
$$g^{\mathbf a}(z_1)=g^{\mathbf a}(0)+\frac{dg^{\mathbf a}}{dz_1}(0)z_1+\bigg(\frac{g^{\mathbf a}(z_1)-g^{\mathbf a}(0)-\frac{dg^{\mathbf a}}{dz_1}(0)z_1}{z_1^2}\bigg)z_1^2,$$
and let 
$$h_1^{\mathbf a}=g^{\mathbf a}(0)$$
and
$$k_1^{\mathbf a}(z_1)=\frac{g^{\mathbf a}(z_1)-g^{\mathbf a}(0)-\frac{dg^{\mathbf a}}{dz_1}(0)z_1}{z_1^2}.$$ By computing the Laurent expansion of $k_1^{\mathbf a}(z_1),$ one sees $z_1=0$ is a removable singularity, and $k_1^{\mathbf a}(z_1)$ extends  as a holomorphic function. From the formulas, we also see that $h_1^{\mathbf a}$ and $k_1^{\mathbf a}$ smoothly depend on $\mathbf a.$  This proves the case $n=1.$ 

Now assume that the result holds when $n=l.$ For $n=l+1,$ if $z_1\neq 0,$ then we have 
\begin{equation*}
\begin{split}
g^{\mathbf a}(z_1,\dots,z_{l+1})=&g^{\mathbf z}(0,z_2,\dots,z_{l+1})+\frac{\partial g^{\mathbf a}}{\partial z_1}(0,z_2,\dots,z_{l+1})z_1\\
&+\bigg(\frac{g^{\mathbf a}(z_1,\dots,z_{l+1})-g^{\mathbf a}(0,z_2,\dots,z_{l+1})-\frac{\partial g^{\mathbf a}}{\partial z_1}(0,z_2,\dots,z_{l+1})z_1}{z_1^2}\bigg)z_1^2,
\end{split}
\end{equation*}
and let 
$$h_1^{\mathbf a}(z_2,\dots,z_{l+1})=\frac{\partial g^{\mathbf a}}{\partial z_1}(0,z_2,\dots,z_{l+1})$$
and
$$k_1^{\mathbf a}(z_1,\dots,z_{l+1})=\frac{g^{\mathbf a}(z_1,\dots,z_{l+1})-g^{\mathbf a}(0,z_2,\dots,z_{l+1})-\frac{\partial g^{\mathbf a}}{\partial z_1}(0,z_2,\dots,z_{l+1})z_1}{z_1^2}.$$
By computing the Laurent expansion again, one can see that  $k_1^{\mathbf a}$ holomorphically extends to $z_1=0;$ and from the formulas, $h_1^{\mathbf a}$ and $k_1^{\mathbf a}$ smoothly depend on $\mathbf a.$  Since $g^{\mathbf a}(0,z_2,\dots, z_{l+1})$ has $l$ variables, by the induction assumption, 
$$g^{\mathbf a}(0,z_2,\dots,z_{l+1})=g^{\mathbf a}(\mathbf 0)+\sum_{i=2}^{l+1}h_i^{\mathbf a}(z_{i+1},\dots,z_{l+1})z_i+\sum_{i=2}^{l+1}k_i^{\mathbf a}(z_i,\dots,z_{l+1})z_i^2$$
for holomorphic functions $\{h_i^{\mathbf a}\}$ and $\{k_i^{\mathbf a}\}$ smoothly depending on $\mathbf a.$ As a consequence, we have 
$$g^{\mathbf a}(z_1,z_2,\dots,z_{l+1})=g^{\mathbf a}(\mathbf 0)+\sum_{i=1}^{l+1}h_i^{\mathbf a}(z_{i+1},\dots,z_{l+1})z_i+\sum_{i=1}^{l+1}k_i^{\mathbf a}(z_i,\dots,z_{l+1})z_i^2.$$
This completes the proof.
\end{proof}

\begin{proof}[Proof of Lemma \ref{OPFCML}] By doing the linear transformation $(\mathbf z,\mathbf a)\mapsto (\mathbf z+\mathbf c_{\mathbf a},\mathbf a),$ we may assume that $\mathbf c_{\mathbf a} = \mathbf 0$  for all $\mathbf a\in D_{\mathbf a}.$ Then by the Taylor Theorem, for each $\mathbf a\in D_{\mathbf a}$ and $\mathbf z \in D_{\mathbf z},$ we can write 
$$ f^{\mathbf a} (\mathbf z) = f^{\mathbf a}(\mathbf 0)+ \sum_{i=1}^n z_i b^{\mathbf a}_{i}(\mathbf z) $$
for some holomorphic functions $b^{\mathbf a}_{i},$ $i =1,\dots, n.$
Since $\mathbf 0$ is a critical point of $f^{\mathbf a},$ we have
$$ b^{\mathbf a}_{i}(\mathbf 0) = \frac{\partial}{\partial z_i} f^{\mathbf a} (\mathbf 0) = 0.$$
As a result, by Taylor theorem again, we can write
$$  f^{\mathbf a} (\mathbf z) 
= f^{\mathbf a}(\mathbf 0)+ \sum_{i=1}^n z_i b^{\mathbf a}_{i}(\mathbf z) 
=  f^{\mathbf a}(\mathbf 0)+\sum_{i, j=1}^n z_i z_j h^{\mathbf a}_{ij}(\mathbf z)$$
for some holomorphic functions $h^{\mathbf a}_{ij},$ $i , j =1,\dots, n.$
Since
$$ \sum_{i, j=1}^n z_i z_j h^{\mathbf a}_{ij}(\mathbf z)
=
\sum_{i, j=1}^n z_i z_j \Big(\frac{h^{\mathbf a}_{ij}(\mathbf z) + h^{\mathbf a}_{ji}(\mathbf z)}{2}\Big),
$$
we may assume that $h^{\mathbf a}_{ij}$ is symmetric in $i$ and $j.$ 
Since $\mathbf 0$ is a non-degenerate critical point of $f^{\mathbf a},$ and $$ \frac{\partial^2}{\partial z_i z_j} f^{\mathbf a} (\mathbf 0) = 2h^{\mathbf a}_{ij}(\mathbf 0),$$ we have $\det (h^{\mathbf a}_{ij}(\mathbf 0)) \neq 0.$

Next, suppose for some $m$ with $0 \leqslant m \leqslant n,$  there exist an open set $V_m \subset \mathbb C^n$ containing $\mathbf 0,$ an open set $A_m \subset D_{\mathbf a}$ containing $\mathbf a_0,$ and a smooth function $\psi_m: V_m \times A_m \to \mathbb C^n$ such that, if we denote $\psi^{\mathbf a}_m(\mathbf Z) = \psi_m(\mathbf Z, \mathbf a),$ then $\psi^{\mathbf a}_m$ gives a holomorphic change of variable with
$$ f^{\mathbf a} (\psi^{\mathbf a}_m(\mathbf Z) ) = f^{\mathbf a}(\mathbf 0)- Z_1^2 - \dots - Z_{m-1}^2
+ \sum_{i,j = m}^n Z_i Z_j H^{\mathbf a}_{m, ij} (\mathbf Z) ,$$
where $H^{\mathbf a}_{m, ij} (\mathbf Z) $ is holomorphic in $\mathbf Z$ and symmetric in $i$ and $j.$ Based on this, we are going to find an open set $V_{m+1}$ of $\mathbb C^n$ containing $\mathbf 0,$ an open set $A_{m+1} \subset A_m$ containing $\mathbf a_0,$ and a smooth function $\psi_{m+1}: V_{m+1} \times A_{m+1} \to \mathbb C^n$ such that, if we denote $\psi^{\mathbf a}_{m+1}(\mathbf Z) = \psi_{m+1}(\mathbf Z, \mathbf a),$ then $\psi^{\mathbf a}_{m+1}$ gives a holomorphic change of variable with
$$ f^{\mathbf a} (\psi^{\mathbf a}_{m+1} (\mathbf Z) ) = f^{\mathbf a}(\mathbf 0)- Z_1^2 - \dots - Z_{m}^2
+ \sum_{i,j = m+1}^n Z_i Z_j H^{\mathbf a}_{m+1, ij} (\mathbf Z) $$ 
for some holomorphic functions $H^{\mathbf a}_{m+1, ij} (\mathbf Z)$ that are symmetric in $i$ and $j.$ 

To do so, we by the Chain Rule have
$$ \frac{\partial^2 f^{\mathbf a}}{\partial Z_i Z_j} (\psi^{\mathbf a}_m(\mathbf 0))
=
(D \psi^{\mathbf a}_m (\mathbf 0))^T
\Big(\frac{\partial^2 f^{\mathbf a}}{\partial z_i z_j} (\mathbf 0) \Big)
(D \psi^{\mathbf a}_m (\mathbf 0)) ,
$$
where $D \psi^{\mathbf a}_m (\mathbf 0)$ is the Jacobian matrix of $\psi^{\mathbf a}_m$ at $\mathbf 0.$ Thus, we have
$$2^{m-1}\det(2H^{\mathbf a}_{m, ij}(\mathbf 0))=\det \Big(\frac{\partial^2 f^{\mathbf a} }{\partial Z_i Z_j} (\psi^{\mathbf a}_m(\mathbf 0))  \Big)\neq 0,
$$
implying that $(H^{\mathbf a}_{m, ij}(\mathbf 0))$ is a $(n-m+1)\times(n-m+1)$ non-singular matrix. Therefore, there exists $k \geqslant m$ such that $H^{\mathbf a}_{m, km}(\mathbf 0)\neq 0.$ Reordering the variables if necessary, we may assume that $H^{\mathbf a}_{m, mm}(\mathbf 0) \neq 0.$ By continuity of $H^{\mathbf a}_{m, mm}(\mathbf Z)$ in $\mathbf Z$ and $\mathbf a,$ there exists an open set $V'_m \subset V_m$ containing $\mathbf 0$ and an open set $A'_m \subset A_m$ containing $\mathbf a_0$ such that $H^{\mathbf a}_{m, mm}(\mathbf 0) \neq 0$ for all $(\mathbf Z, \mathbf a) \in V'_m \times A'_m.$ Then we can let 
$$ \widetilde{H}^{\mathbf a}_{m, ij}(\mathbf Z) = \frac{H^{\mathbf a}_{m, ij}(\mathbf Z)}{H^{\mathbf a}_{m, mm}(\mathbf Z)}$$
and have
\begin{align*}
f^{\mathbf a}(\psi^{\mathbf a}_m (\mathbf Z)) =& f^{\mathbf a}(\mathbf 0) -Z_1^2 - \dots - Z_{m-1}^2
+ \sum_{i,j = m}^n Z_i Z_j H^{\mathbf a}_{m, ij} (\mathbf Z) \\
=& f^{\mathbf a}(\mathbf 0) -Z_1^2 - \dots -Z_{m-1}^2
+ H^{\mathbf a}_{m, mm}(\mathbf Z)\sum_{i,j = m}^n Z_i Z_j \widetilde H^{\mathbf a}_{m, ij}(\mathbf Z) \\
=& f^{\mathbf a}(\mathbf 0) -Z_1^2 - \dots -Z_{m-1}^2 
+ H^{\mathbf a}_{m,mm}(\mathbf Z)\Big( Z_m + \sum_{j=m+1}^n Z_j \widetilde H^{\mathbf a}_{m, mj} (\mathbf Z) \Big)^2 
 \\
& -   H^{\mathbf a}_{m,mm}(\mathbf Z)\Big[\Big( \sum_{j=m+1}^n Z_j \widetilde H^{\mathbf a}_{m, mj} (\mathbf Z) \Big)^2
+ \sum_{i,j = m+1}^n Z_i Z_j \widetilde H^{\mathbf a}_{m, ij}(\mathbf Z) \Big].
\end{align*}
Define $\mathbf W =\mathbf W(\mathbf{Z})$ by
$$W_l = Z_l$$
for  $l \neq m,$ and
$$ W_m =  \sqrt{-H^{\mathbf a}_{m,mm}(\mathbf Z)} \Big( Z_m + \sum_{j=m+1}^n Z_j \widetilde H^{\mathbf a}_{m, mj} (\mathbf Z) \Big).$$
We note that $$\frac{\partial W_l}{\partial Z_k} (\mathbf 0)
= \delta_{l,k}$$
 for $l\neq m,$ and
$$
\frac{\partial W_m}{\partial Z_k} (\mathbf 0)
= \sqrt{-H^{\mathbf a}_{mm}(\mathbf 0)} \Big( \delta_{m,k} + \sum_{j=m+1}^n \delta_{j,k} \widetilde H^{\mathbf a}_{m, mj} (\mathbf 0) \Big).$$
Then the Jacobian matrix $DW(\mathbf 0)$ is an upper triangular matrix with the $(m,m)$-th entry $\sqrt{-H^{\mathbf a}_{m,mm}(\mathbf 0)}\neq 0$ and all the other diagonal entries $1,$ hence $\det DW(\mathbf 0)\neq 0.$

Now consider the map $G: V'_m \times A'_m \to \mathbb C^n \times \mathbb R^k$ defined by
$$ G(\mathbf Z, \mathbf a) = (\mathbf W(\mathbf Z), \mathbf a).$$
Then the Jacobian matrix $DG (\mathbf 0, \mathbf a_0)$ is of the form
$$  
DG(\mathbf 0, \mathbf a_0)
= 
\begin{pmatrix}
DW(\mathbf 0) & * \\
0 & I_k
\end{pmatrix},
$$
where $I_k$ is the $k\times k$ identity matrix. 
Moreover, $\det(DG(\mathbf 0, \mathbf a_0)) = \det(DW(\mathbf 0)) \neq 0.$ Thus, by the Inverse Function Theorem, there exists an open set $V''_m\subset V'_m$ and an open subset with compact closure $A_{m+1} \subset A'_m$ containing $\mathbf a_0$ such that $G: V''_{m} \times A_{m+1} \to \mathbb C^n \times A_m$ is a diffeomorphism to its image. By slightly shrinking $A_{m+1}$ if necessary, $G(V''_{m} \times A_{m+1})$ contains an open subset of the form $V_{m+1}\times A_{m+1}.$ For each $\mathbf a\in A_{m+1},$ let $\psi^{\mathbf a}_{m+1} =  \psi^{\mathbf a}_m \circ \mathbf W^{-1}:V_{m+1}\to D_{\mathbf z}.$ Then we have
\begin{align*}
f^{\mathbf a}( \psi^{\mathbf a}_{m+1} (\mathbf{W}))
=  f^{\mathbf a}(\mathbf 0)-W_1^2 - \dots -W_{m}^2
+ \sum_{i,j = m+1}^n W_i W_j H^{\mathbf a}_{m+1, ij} (\mathbf W)
\end{align*}
for some holomorphic functions $H^{\mathbf a}_{m+1, ij} (\mathbf Z)$ that are symmetric in $i$ and $j.$ 

Inductively doing the above procedure on $m,$ and letting $V=V_{n},$ $A=A_{n}$ and $\psi^{\mathbf a}=\psi^{\mathbf a}_n,$ we prove the result. 
Moreover, by  the Chain Rule, we have
$$ \mathrm{Hess}( f^{\mathbf a} \circ \psi^{\mathbf a})(\mathbf 0))
=
(D \psi^{\mathbf a} (\mathbf 0))^T
\Big(\mathrm{Hess} (f^{\mathbf a}) (\mathbf 0) \Big)
(D \psi^{\mathbf a} (\mathbf 0)).
$$
Since $\mathrm{Hess}( f^{\mathbf a} \circ \psi^{\mathbf a})(\mathbf 0)$ is equal to the negative of the  $n\times n$ identity matrix, by taking the determinant on both sides, we get
$$ \det \Big( D (\psi^{\mathbf a})(\mathbf 0)  \Big)
= \frac{2^{\frac{n}{2}}}{\sqrt{- \det \mathrm{Hess}(f^{\mathbf a})(\mathbf 0)}}.$$
\end{proof}


\noindent
Ka Ho Wong\\
Department of Mathematics\\  Texas A\&M University\\
College Station, TX 77843, USA\\
(daydreamkaho@tamu.edu)
\\

\noindent
Tian Yang\\
Department of Mathematics\\  Texas A\&M University\\
College Station, TX 77843, USA\\
(tianyang@math.tamu.edu)

\end{document}